\documentclass{article}
\usepackage[utf8]{inputenc} 					
\usepackage[T1]{fontenc}    					

\usepackage[english]{babel}  					
\usepackage{lmodern}							
\usepackage{amsmath}							
\usepackage{amssymb}							
\usepackage{amsthm}								
\usepackage{geometry}
\usepackage[dvipsnames]{xcolor}				
\usepackage[all]{xy}							

\usepackage{ifthen} 							

\usepackage{tikz}								
	\usetikzlibrary{cd}							
	\usetikzlibrary{calc}						
	\usetikzlibrary{shapes}						
	\usetikzlibrary{fit}						
\usetikzlibrary{arrows,arrows.meta}
\tikzcdset{arrow style=tikz, diagrams={>=latex}}

\usepackage{stmaryrd}							
\usepackage{tabularx}							
\usepackage{multirow}
\usepackage{colortbl}							
\usepackage{enumitem}							
\usepackage{pdflscape}							

\usepackage{makeidx}							
\usepackage[french,refpage]{nomencl}			

\makeindex										
\makenomenclature								

\usepackage{hyperref}							

\hypersetup{									
	colorlinks=true,							
	breaklinks=true,							
	citecolor=Green,       						
	filecolor=magenta,      					
	urlcolor=blue,           					
	linkcolor= violet,							
	bookmarksopen=false,							
	pdftitle={Cluster algebras associated with open Richardson varieties: an algorithm to compute initial seeds},		
	pdfauthor={Etienne Ménard},					
	pdfsubject={Cluster Algebras},			
	pdffitwindow=true,     					
}

\newcommand{\rouge}[1]{\textcolor{red}{\overline{#1}}}
\newcommand{\bleu}[1]{\textcolor{blue}{\underline{#1}}}

\newcommand{\G}{\mathbb{G}}

\newcommand{\N}{\mathbb{N}}

\newcommand{\Z}{\mathbb{Z}}
\newcommand{\CC}{\mathbb{C}}

\newcommand{\Ind}{\mathrm{Ind}}

\newcommand{\C}{\mathcal{C}}
\newcommand{\E}{\mathcal{E}}
\newcommand{\R}{\mathcal{R}}

\newtheorem{theorem}{Theorem}[section]
\newtheorem{remark}[theorem]{Remark}
\newtheorem{conjecture}[theorem]{Conjecture}
\newtheorem{proposition}[theorem]{Proposition}
\newtheorem{lemma}[theorem]{Lemma}
\newtheorem{corollary}[theorem]{Corollary}
\newtheorem{hypothesis}[theorem]{Hypothesis}

\theoremstyle{definition}					

\newtheorem{definition}[theorem]{Definition}
\newtheorem{ex}[theorem]{Example}
\newtheorem{cex}[theorem]{Counter-example}
\newtheorem{nota}[theorem]{Notation}

\newcommand{\Max}{\mathrm{max}}
\newcommand{\Min}{\mathrm{min}}

\newcommand{\Mini}{\mathrm{mini}}
\renewcommand{\mod}{\mathbf{mod}}
\newcommand{\Hom}{\mathrm{Hom}}

\newcommand{\add}{\mathrm{add}}
\newcommand{\Fac}{\mathrm{Fac}}
\newcommand{\Sub}{\mathrm{Sub}}

\newcommand{\End}{\mathrm{End}}
\newcommand{\Ext}{\mathrm{Ext}}
\newcommand{\Pol}{\mathrm{Pol}}

\newcommand{\soc}{\mathrm{soc}}

\newcommand{\Norm}{\mathrm{Norm}}

\newcommand{\decsoc}[1]{\vcenter{\vbox{\xymatrix@=-3pt{#1}}}}

\newcommand{\0}{\mathbf{0}}

\newcommand{\bs}{\backslash}

\renewcommand{\bar}[1]{\overline{#1}}

\newcommand{\rien}{\phantom{ }}		

\title{Cluster algebras associated with open Richardson varieties: an algorithm to compute initial seeds}
\author{Etienne Ménard}

\begin{document}
\maketitle
\begin{abstract}
	We present a new algorithm to compute initial seeds for cluster structures on categories associated with coordinate rings of open Richardson varieties. This allows us to explicitely determine seeds first considered in \cite{leclerc2016ClusterStructuresStrata}.
\end{abstract}
\section{Introduction}

Follwing \cite{leclerc2016ClusterStructuresStrata}, we first recall the context of this paper. 

Let $G$ be a simple simply connected algebraic $\CC$-group. We assume $G$ to be simply-laced and fix $H$ a maximal torus in $G$, $B$ a Borel subgroup of $G$ containing $H$, and we denote by $B^-$ the Borel subgroup opposite to $B$ with respect to $H$. Let $W=\Norm_G(H)/H$ be the Weyl group with length function $w\mapsto\ell(w)$ and longest element $w_0\in W$, of length $r=\ell(w_0)$. We denote by $S=\{s_1,\dots,s_n\}$ its subset of standard Coxeter generators. Reduced representatives of $w\in W$ will be denoted by $\bar{w}=[i_{\ell(w)},\dots,i_1]$ corresponding to products of $\ell(w)$ simple reflections $s_{i_{\ell(w)}}\cdots,s_{i_1}$. Our convention for the Dynkin diagram of $G$ is:
\begin{figure}[ht]
\[
	A_n : \begin{tikzcd}[column sep=10pt, row sep=0pt]1\ar[r,dash]&2\ar[r,dashed,dash]&\cdots\ar[r,dash,dashed]&n-1\ar[r,dash]&n\end{tikzcd}, n\geq 1
\]
\[
	D_n :\begin{tikzcd}[column sep=10pt, row sep=0pt]
	&&&&n-1\\	
	1\ar[r,dash]&2\ar[r,dash,dashed]&\cdots\ar[r,dash,dashed]&n-2\ar[ur,dash]\ar[dr,dash]&\\
	&&&&n
	\end{tikzcd}, n\geq 4
\]
\[
	E_n :\begin{tikzcd}[column sep=10pt, row sep=10pt]
	1\ar[r,dash]&3\ar[r,dash]&4\ar[r,dash]\ar[d,dash]&5\ar[r,dash]&6\ar[r,dash]&7\ar[r,dash]&8\\
	&&2&&&&
	\end{tikzcd}, n=6,7,8
\]
\caption{Dynkin diagrams}
\label{figureDynkinDiagrams}
\end{figure}
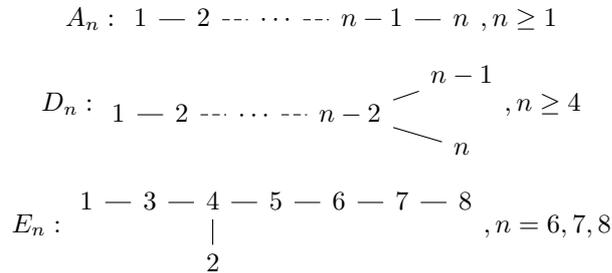

We consider the flag variety $X=B^-\bs G$ and we denote by $\pi:G\rightarrow X$ the natural projection $\pi(g):= B^-g$. The Bruhat decomposition (resp. Birkhoff decomposition):
\[
	G=\bigsqcup_{w\in W}B^-wB^-,\quad (\text{resp. } G=\bigsqcup_{w\in W}B^-wB,)
\]
projects to the Schubert decomposition (resp. \emph{opposite} Schubert decomposition)
\[
	X=\bigsqcup_{w\in W}C_w,\quad (\text{resp. }X=\bigsqcup\limits_{w\in W}C^w,)
\]
where $C_w=\pi(B^-wB^-)$ is the \emph{Schubert cell} attached to $w$, isomorphic to $\CC^{\ell(w)}$ and $C^w=\pi(B^-wB)$ is the \emph{opposite Schubert cell} attached to $w$, isomorphic to $\CC^{\ell(w_0)-\ell(w)}$.

The intersection 
\[
	\R_{v,w}:=C^v\cap C_w
\]
is called an \emph{open Richardson variety}. In \cite{kazhdan1980SchubertVarietiesPoincare,deodhar1985GeometricAspectsBruhat} it is shown that $\R_{v,w}$ is non-empty if and only $v\leq w$ for the Bruhat order of $W$ and that it is a smooth irreducible locally closed subset of $C_w$ of dimension $\ell(w)-\ell(v)$. We get a finer stratification
\[
	X=\bigsqcup_{v\leq w}\R_{v,w}
\]
though the strata $\R_{v,w}$ are no longer isomorphic to affine spaces (in general).

In \cite{lusztig1998TotalPositivityPartial} Lusztig used this stratification for studying the nonnegative part $X_{\geq 0}$ of the flag variety $X$ (more generally a partial flag variety). In \cite{leclerc2016ClusterStructuresStrata}, Leclerc showed that there exists a cluster subalgebra $\widetilde{R}_{v,w}$ contained in the coordinate ring $\CC[\R_{v,w}]$ and proposed the following conjecture:
\begin{conjecture}[Leclerc's conjecture]
	For any strata $\R_{v,w}$ we have $\CC[\R_{v,w}]=\widetilde{R}_{v,w}$.
\end{conjecture}

In order to help to solve this conjecture, we want to work with explicit seeds for $\widetilde{R}_{v,w}$. Some seeds are defined (but not explicitly computed) in \cite{leclerc2016ClusterStructuresStrata} using a categorification that we now introduce.

In \cite{geiss2011KacMoodyGroups}, Geiss, Leclerc and Schröer introduced a category called $\C_w$ whose definition will be recalled below. Independently, Buan, Iyama, Reiten and Scott introduced the same category in \cite{buan2009ClusterStructures2Calabi}. The latter also defined the cluster structure on 2-Calabi-Yau categories: a categorical analog of the cluster algebra structure with a notion of seed and its mutations.

In \cite{geiss2011KacMoodyGroups} the authors build explicitely a seed for the cluster structure on $\C_w$. Via the cluster character $\varphi$ that they defined in \cite[Section 9]{geiss2006RigidModulesPreprojective} they found a seed for the cluster algebra on $\CC[C_w]$ as the image of a seed $M$ of $\C_w$ via the cluster character.

We will here use the same strategy to find a seed for $\widetilde{R}_{v,w}$ based on a category $\C_{v,w}$, defined below. In order to do so, we need a seed for the cluster structure on $\C_{v,w}$ and we will compute this seed thanks to a mutation algorithm that we will design and prove. This algorithm takes as entry data a seed for the cluster structure on $\C_w$ computed thanks to a reduced representative $\bar{w}$ of $w$, and the Weyl group element $v$.

In this paper, relying on the existence of a cluster structure for $\C_{v,w}$, proved in \cite{leclerc2016ClusterStructuresStrata}, we will define and prove an algorithm giving an explicit way to find a seed for this cluster structure (in the sense of \cite{buan2009ClusterStructures2Calabi}). This paper summarize the results of the author's thesis \cite{menard2021AlgebresAmasseesAssociees} written during his PhD internship at Université de Caen. 

\subsection{Organisation}

In the second section we will define the categories $\C_w$, $\C^v$, and $\C_{v,w}$, their properties, and we will recall how to compute an initial seed $V_{\bar{w}}$ for $\C_w$. In the third section we will introduce some combinatorial notation that we will use for our algorithm. In the fourth section we will study the quiver associated to the initial seed $V_{\bar{w}}$, its combinatorial definition, and some properties of his graph. In the fifth section we will consider a way to represent the elements of $\C_w$ thanks to an integer vector called $\Delta$-vector. We will study its properties, how it can be combinatorially defined for $V_{\bar{w}}$ and why it plays a key role in our algorithm proof.

We will then define our algorithm in the sixth section and give two equivalent formulations. In the following section we will prove it and end this article with some results obtained thanks to this algorithm.

In appendix we put a discussion about the different configurations of the quiver and their evolution thanks to mutation. We end this article with a step-by-step example of running the algorithm with given data.

\section{Categorification of cluster algebras}
	In this section we will recall some results from Buan, Iyama, Reiten and Scott \cite{buan2009ClusterStructures2Calabi}, from Geiss, Leclerc and Schröer \cite{geiss2011KacMoodyGroups} and from Leclerc \cite{leclerc2016ClusterStructuresStrata} about the categories $\C_w,\C^v$ and $\C_{v,w}$.
	
	\subsection{The preprojective algebra}
	Given a Dynkin diagram of type $\Delta$, we construct the preprojective algebra of type $\Delta$. To do so we take any orientation on the Dynkin diagram to have a quiver $Q=(Q_0,Q_1)$, double each arrow (for each arrow $\alpha\in Q_1$ we now have $\alpha$ and $\alpha^\ast\in Q_1^\ast$, forming a 2-cycle), and obtain the quiver $\bar{Q}=(Q_0,Q_1\sqcup Q_1^\ast)$. 
	
	We then take the $\CC$-path algebra $\CC\bar{Q}$ of $\bar{Q}$ and quotient it by the two-sided ideal generated by:
	\[
		c=\sum\limits_{\alpha\in Q_1}\alpha\alpha^\ast-\alpha^\ast\alpha.
	\]
	We then obtain the preprojective algebra $\Lambda$ of type $\Delta$:
	
	\[
		\Lambda=\CC\bar{Q}/(c)
	\]
	The reader will be able to find more information about this construction in \cite[Section 3]{geiss2008PreprojectiveAlgebrasCluster}. 
	
	In the following we will denote by $\mod(\Lambda)$ the category of finitely generated $\Lambda$-modules. The simple 1-dimensional $\Lambda$-modules in correspondence with the vertices of the Dynkin diagram will be denoted by $S_i$ $(i\in I)$ and their respective injective envelopes by $Q_i$ $(i\in I)$. We will denote by $P_i$ $(i\in I)$ the projective cover of $S_i$.
	
	We will now introduce the categories $\C_{w}$, $\C^v$ and $\C_{v,w}$ and their links with the cluster algebra structures on the rings $\CC[C_w]$, $\CC[C^v]$ and $\CC[\R_{v,w}]$.
	\subsection{Cluster structure on subcategories of \texorpdfstring{$\mod(\Lambda)$}{mod(Lambda)}}\label{sectionClusterStructureModLambda}
		In the following, $\C$ will be an extension closed subcategory of $\mod(\Lambda)$.
		
		The cluster structure will be introduced by first defining a seed (a module and its quiver) and then the notion of mutation of cluster seeds. A seed for this structure is a module with some properties whose definitions are recalled now, adapting \cite{buan2009ClusterStructures2Calabi} to the specific context of  subcategories of $\mod(\Lambda)$.
		
		\begin{definition}
			For a $\Lambda$-module $M$, $\add(M)$ is the full subcategory of $\mod(\Lambda)$ consisting of all modules isomorphic to direct summands of finite direct sums of copies of $M$.			
			
			$M$ is \emph{rigid} if $\Ext^1_\Lambda(M,M)=0$.
			
			It is $\C$-\emph{maximal rigid} if for any $T'\in\C$ such that $M\oplus T'$ is rigid, then $T'\in\add(M)$.

			$M$ is said to be \emph{cluster-tilting} if we have the following equivalence:
			\[
				(X\in \C\text{ and }\Ext^1_{\Lambda}(M,X)=0)\Leftrightarrow(X\in\add(M))
			\]
			
			$M$ is \emph{basic} if its indecomposable direct summands are pairwise non-isomorphic.
		\end{definition}
		
		Then to any basic $\C$-cluster-tilting module $M$ we can associate a seed in the following way:
		\begin{definition}[Seed for a cluster structure]\label{definitionClusterSeed}
			Let $M=\bigoplus\limits_{i=1}^r M_i$ be the indecomposable direct summands decomposition of $M$. Then the collection $\{M_1,\dots,M_r\}$ is a \emph{cluster} whose projective direct summands are coefficients and the others are cluster variables.
			
			To $M$ one can associate a quiver $\Gamma_M$ by taking the Gabriel quiver of the algebra $\End_\Lambda(M)$ associating the vertices $\Hom_\Lambda(M,M_i)$ of the latest to the vertex labelled by $M_i$ in $\Gamma_M$ and considering quivers to be defined up to arrows between two vertices corresponding to coefficients. 			
		\end{definition}		 
		\begin{remark}
			Here we decide not to use the quiver of $\End^{\mathrm{op}}(M)$ contrary to many articles without any loss of generality.
		\end{remark}
		
		To define mutation we have to consider short exact sequences:
		\begin{theorem}[\cite{buan2009ClusterStructures2Calabi}]
			For a seed $M\in\C$, $M_i$ a cluster variable and the quiver $\Gamma_M=(Q_0,Q_1,s,t)$ of $M$, there exists a unique (up to isomorphism) indecomposable module $M_i^\ast$ such that the two following exchange sequences exist:
			\[
				M_i\rightarrow\bigoplus_{\substack{\alpha\in Q_1\\s(\alpha)=M_i}}M_{t(\alpha)}\rightarrow M_i^\ast,\quad M_i^\ast\rightarrow\bigoplus_{\substack{\alpha\in Q_1\\t(\alpha)=M_i}}M_{s(\alpha)}\rightarrow M_i
			\]
			Then $M^\ast=\bigoplus\limits_{\substack{k=1\\k\neq i}}^r M_k \oplus M_i^\ast$ is another cluster-tilting module called the mutation of $M$ in direction $i$, denoted by $\mu_i(M)$. The mutation of the quiver $\mu_i(\Gamma_M)$ is the classical mutation of quivers. 
		\end{theorem}
		
	\subsection{The category \texorpdfstring{$\C_{v,w}$}{Cv,w}}
		Given a Coxeter generator $s_i$ $(i\in I)$, as in \cite[Section 3.2]{leclerc2016ClusterStructuresStrata}, we define the endofunctor $\E_i=\E_{s_i}$ on the objects of $\mod(\Lambda)$ as the kernel of the surjection $X\rightarrow S_i^{\oplus m_i(X)}$ where $m_i(X)$ is the multiplicity of $S_i$ in the head of $X$. 
		
		Dually we define the endofunctor $\E^\dagger_i$ on the objects of $\mod(\Lambda)$ as the cokernel of the injection $ S_i^{\oplus m^\dagger_i(X)}\rightarrow X$ where $m_i^\dagger(X)$ is the multiplicity of $S_i$ in the socle of $X$. 
		
		$\E_i$ acts on $X$ by removing the $S_i$-isotypical part of its head, $\E^\dagger_i$ by removing the $S_i$-isotypical part of its socle. As the $\E_i$ and $\E_i^\dagger$ verify the braid relations of $W$, we can define unambiguously $\E_w$ on any $w\in W$ by taking any reduced representative and composing the corresponding endofunctors.
		
		Given these functors we can now define the following categories which will be our main interest:
		\begin{definition}
			For $w\in W$ let $u:=w^{-1}w_0$, $I_w:=\E_u\left(\bigoplus_{i\in I}Q_i\right)$ and $J_w:=\E^\dagger_{w^{-1}}\left(\bigoplus_{i\in I}Q_i\right)$.
			
			We define:
			\[
				\C_w:=\Fac(I_w)=\E_u(\mod(\Lambda)),\quad 
				\C^v:=\Sub(J_v)=\E^\dagger_{v^{-1}}(\mod(\Lambda))
			\]
		\end{definition}
		
		As explained in \cite{leclerc2016ClusterStructuresStrata} and \cite{geiss2011KacMoodyGroups}, these two categories give categorical models of the strata $\R_{e,w}$ ($\R_{w,w_0}$ respectively) of the flag variety $X$.
		
		\begin{ex}
			Let $W$ be of type $A_3$ and $w=s_1s_2s_3s_1$ and $v=s_2s_1$. To give a straightforward visual representation of the endofunctors' action we will represent the modules by their socle decomposition. 
			
			The preprojective algebra of type $A_3$ has 3 maximal indecomposable projective-injective modules:
			\[
				Q_1=\begin{tikzcd}[row sep =-2pt,column sep =-2pt]&&3\ar[dl,dash]\\&2\ar[dl,dash]&\\1&&\end{tikzcd},\quad Q_2=\begin{tikzcd}[row sep =-2pt,column sep =-2pt]&2\ar[dl,dash]\ar[dr,dash]&\\1\ar[dr,dash]&&3\ar[dl,dash]\\&2&\end{tikzcd},\quad Q_3=\begin{tikzcd}[row sep =-2pt,column sep =-2pt]1\ar[dr,dash]&&\\&2\ar[dr,dash]&\\&&3\end{tikzcd}
			\]
			
			We compute $u=w^{-1}w_0=s_2s_1$ and $v^{-1}=s_1s_2$. We have:
			\[
				\E_1(Q_1)=Q_1,\quad \E_1(Q_2)=Q_2,\quad \E_1(Q_3)=\begin{tikzcd}[row sep =-2pt,column sep =-2pt]&2\ar[dr,dash]&\\&&3\end{tikzcd}
			\]
			\[
				\E_u(Q_1)=Q_1=\begin{tikzcd}[row sep =-2pt,column sep =-2pt]&&3\ar[dl,dash]\\&2\ar[dl,dash]&\\1&&\end{tikzcd},\quad \E_u(Q_2)=\begin{tikzcd}[row sep =-2pt,column sep =-2pt]1\ar[dr,dash]&&3\ar[dl,dash]\\&2&\end{tikzcd},\quad \E_u(Q_3)=\begin{tikzcd}3\end{tikzcd}=S_3
			\]
			
			Similarly we have 
			\[
				\E^\dagger_{v^{-1}}(Q_1)=\begin{tikzcd}3\end{tikzcd}=S_3,\quad \E_{v^{-1}}^\dagger(Q_2)=\begin{tikzcd}[row sep =-2pt,column sep =-2pt]&2\ar[dl,dash]\ar[dr,dash]&\\1&&3\end{tikzcd},\quad \E^\dagger_{v^{-1}}(Q_3)=Q_3=\begin{tikzcd}[row sep =-2pt,column sep =-2pt]1\ar[dr,dash]&&\\&2\ar[dr,dash]&\\&&3\end{tikzcd}
			\]
			So
			\[
				I_w=\begin{tikzcd}[row sep =-5pt,column sep =-5pt]&&3\\&2&\\1&&\end{tikzcd}\oplus\begin{tikzcd}[row sep =-2pt,column sep =-2pt]1\ar[dr,dash]&&3\ar[dl,dash]\\&2&\end{tikzcd}\oplus S_3,\quad J_v=S_3\oplus\begin{tikzcd}[row sep =-2pt,column sep =-2pt]&2\ar[dl,dash]\ar[dr,dash]&\\1&&3\end{tikzcd}\oplus\begin{tikzcd}[row sep =-2pt,column sep =-2pt]1\ar[dr,dash]&&\\&2\ar[dr,dash]&\\&&3\end{tikzcd}
			\]
			and eventually the categories $\C_w$ and $\C^v$ have the following sets of indecomposable modules:
			\[
				\Ind(\C_w)=\left\{S_3,S_1,\begin{tikzcd}[row sep =-2pt,column sep =-2pt]1\ar[dr,dash]&&3\ar[dl,dash]\\&2&\end{tikzcd},\begin{tikzcd}[row sep =-2pt,column sep =-2pt]&3\ar[dl,dash]\\2&\end{tikzcd},\begin{tikzcd}[row sep =-2pt,column sep =-2pt]&&3\ar[dl,dash]\\&2\ar[dl,dash]&\\1&&\end{tikzcd}\right\},\quad \Ind(\C^v)=\left\{S_3, S_1, \begin{tikzcd}[row sep =-2pt,column sep =-5pt]&2\ar[dl,dash]\ar[dr,dash]&\\1&&3\end{tikzcd},\begin{tikzcd}[row sep =-2pt,column sep =-5pt]2\ar[dr,dash]&\\&3\end{tikzcd},\begin{tikzcd}[row sep =-2pt,column sep =-5pt]1\ar[dr,dash]&&\\&2\ar[dr,dash]&\\&&3\end{tikzcd}\right\}
			\]
			
			One can easily verify that these categories are closed under extensions, and closed either under factors (for $\C_w$), or under submodules (for $\C^v$).
		\end{ex}
		
		\begin{definition}[\cite{leclerc2016ClusterStructuresStrata}]
			Given $v\leq w\in W$ we define the category:
			\[
				\C_{v,w}=\C^v\cap\C_w
			\]
		\end{definition}
			In \cite[Corollary 3.11]{leclerc2016ClusterStructuresStrata}, Leclerc showed that this category has a cluster structure but did not give an explicit seed of it. 
			
			Given a $\C_{w}$-cluster tilting object $T$, and denoting $t_v(T)$ the maximal submodule of $T$ in $\C_v$, he showed that $T/t_v(T)$ is a $\C_{v,w}$-cluster-tilting object but it is neither basic in general nor do we know its quiver.
	
			We recall some definitions, as in Section \ref{sectionClusterStructureModLambda}, $\C$ is an extension-closed subcategory of $\mod(\Lambda)$.
			\begin{definition}
				Let $T$ be a $\C$-module, $\Sigma(T)$ is the number of isoclasses of indecomposable direct summands of $T$. 
				
				A module $T$ having the maximal number $\Sigma(T)$ in $\C$ is said to be $\C$-\emph{complete}.
			\end{definition}
			
			We then have the following equivalence:
			\begin{theorem}[{\cite[Theorem 2.2.]{geiss2006RigidModulesPreprojective},\cite[Theorem 2.9.]{geiss2011KacMoodyGroups}}]\label{thmMaximalRigidity}
				For a rigid $\Lambda$-module $T$ the following are equivalent:
				\begin{itemize}
					\item $T$ is $\C_w$-maximal rigid;
					\item $T$ is $\C_w$-complete rigid;
					\item $T$ is $\C_w$-cluster-tilting.
				\end{itemize}
			\end{theorem}
				The same holds for $\C^v$ by duality.

			It is then possible to decide if a rigid module is cluster-tilting by counting its isoclasses of indecomposable summands and we have in addition the following information about the number of isoclasses:
			\begin{theorem}[{\cite[Proposition 4.3]{leclerc2016ClusterStructuresStrata}}]
				A $\C_{v,w}$-maximal rigid basic module $M$ has $\Sigma(M)=\ell(w)-\ell(v)$ indecomposable summands.
			\end{theorem}

	\subsection{The initial module \texorpdfstring{$V_{\bar{w}}$}{Vw}}
		
		In order to get some examples of Leclerc's conjecture, we need to have an explicit initial seed for the cluster structure on $\C_{v,w}$. We will not compute it directly as the authors did for $\C_w$ in \cite[Section 2.4]{geiss2011KacMoodyGroups} but we will in fact start from such a seed for $\C_w$. In this section we will recall the construction of this seed thanks to a reduced representative $\bar{w}$ of $w\in W$.
			\begin{definition}[Definition of $\soc$ notation]
				For a $\Lambda$-module $X$ and a simple module $S_j$, let $\soc_{(j)}(X):=\soc_{S_j}(X)$ be the $S_j$-isotypical part of the socle of $X$.
				
				For a sequence $[j_1,\dots,j_t]$ with $1\leq j_p\leq n$ for any $p$, there is a unique chain of submodules of $X$
				\[
					0=X_0\subseteq X_1\subseteq\dots\subseteq X_t\subseteq X
				\]
				such that $X_p/X_{p-1}=\soc_{(j_p)}(X/X_{p-1})$. Then we define $\soc_{(j_1,\dots,j_t)}(X):=X_t$.
			\end{definition}
			\begin{ex}\label{exampleSoc}
				Let $W$ be of type $A_4$ and $X=Q_3$, we have:
				\[
					Q_3=\begin{tikzcd}[row sep=-2pt, column sep =-2pt]
						&2\ar[dl,dash]\ar[dr,dash]&&\\
						1\ar[dr,dash]&&3\ar[dl,dash]\ar[dr,dash]&\\
						&2\ar[dr,dash]&&4\ar[dl,dash]\\
						&&3&
					\end{tikzcd}.
				\]
				
				Then $\soc_{(3)}(Q_3)= S_3$, $\soc_{(j)}(Q_3)=0$ for all $j\neq 3$. We have the following chain of submodules of $Q_3$:
				\[
					0\subseteq S_3\subseteq S_3\subseteq\begin{tikzcd}[row sep=-2pt, column sep =-2pt]
						2\ar[dr,dash]&\\
						&3
					\end{tikzcd}\subseteq\begin{tikzcd}[row sep=-2pt, column sep =-2pt]
						2\ar[dr,dash]&\\
						&3
					\end{tikzcd}\subseteq\begin{tikzcd}[row sep=-2pt, column sep =-2pt]
						2\ar[dr,dash]&&4\ar[dl,dash]\\
						&3&
					\end{tikzcd}\subseteq\begin{tikzcd}[row sep=-2pt, column sep =-2pt]
						&2\ar[dl,dash]\ar[dr,dash]&&\\
						1\ar[dr,dash]&&3\ar[dl,dash]\ar[dr,dash]&\\
						&2\ar[dr,dash]&&4\ar[dl,dash]\\
						&&3&
					\end{tikzcd}
				\]
				allowing us to define $\soc_{(3,1,2,3,4)}(Q_3)=\begin{tikzcd}[row sep=-2pt, column sep =-2pt]
						2\ar[dr,dash]&&4\ar[dl,dash]\\
						&3&
					\end{tikzcd}$.
			\end{ex}
			\begin{definition}[Modules $V_{k,\bar{w}}$]\label{definitionVk}
				Given a reduced representative $\bar{w}=[i_{\ell(w)},\dots,i_1]$ of $w\in W$ and $1\leq k\leq\ell(w)$, let
				\[
					V_{k,\bar{w}}:=\soc_{(i_k,\dots,i_1)}(Q_{i_k}).
				\]
			\end{definition}
			\begin{remark}
				Note that defining $V_{k,\bar{w}}$ requires to read $\bar{w}$ from left to right even if it is indexed from right to left.
			\end{remark}
			\begin{ex}\label{exampleVk}
				For $W$ of type $A_4$ let $\bar{w}=[4,2,3,1,2,3,4]$ then, thanks to Example \ref{exampleSoc}, we have 
				\[
					V_{5,\bar{w}}=\begin{tikzcd}[row sep=-2pt, column sep =-2pt]
						2\ar[dr,dash]&&4\ar[dl,dash]\\
						&3&
					\end{tikzcd}.
				\]
			\end{ex}
			These $V_{k,\bar{w}}$ then gives what will be the starting point of our algorithm:
			\begin{definition}[Module $V_{\bar{w}}$]
				Under the same assumptions as before, let:
				\[
					V_{\bar{w}}=\bigoplus\limits_{i=1}^{\ell(w)}V_{i,\bar{w}}.
				\]
			\end{definition}
			\begin{ex}
				For $\bar{w}=[4,2,3,1,2,3,4]$ as in Example \ref{exampleVk}, we have (summands ordered by ascending index):
				\[
					V_{\bar{w}}=S_4
					\oplus\begin{tikzcd}[row sep=-2pt, column sep =-2pt]&4\ar[dl,dash]\\3&\end{tikzcd}
					\oplus\begin{tikzcd}[row sep=-2pt, column sep =-2pt]&&4\ar[dl,dash]\\&3\ar[dl,dash]&\\2&&\end{tikzcd}
					\oplus\begin{tikzcd}[row sep=-2pt, column sep =-2pt]&&&4\ar[dl,dash]\\&&3\ar[dl,dash]&\\&2\ar[dl,dash]&&\\1&&&\end{tikzcd}
					\oplus \begin{tikzcd}[row sep=-2pt, column sep =-2pt]2\ar[dr,dash]&&4\ar[dl,dash]\\&3&\end{tikzcd}
					\oplus\begin{tikzcd}[row sep=-2pt, column sep =-2pt]&2\ar[dl,dash]\ar[dr,dash]&&4\ar[dl,dash]\\1\ar[dr,dash]&&3\ar[dl,dash]&\\&2&&\end{tikzcd}
					\oplus\begin{tikzcd}[row sep=-2pt, column sep =-2pt]2\ar[dr,dash]&&\\&3\ar[dr,dash]&\\&&4\end{tikzcd}.
				\]
			\end{ex}
			
			$V_{\bar{w}}$ is enough to define the whole category $\C_w$ which has the following properties:
			\begin{theorem}[{\cite[Section 2.4]{geiss2011KacMoodyGroups}}]\label{thmNumberOfIndecomposables}
				$V_{\bar{w}}$ is a $\C_w$-cluster-tilting module, $\C_w=\Fac(V_{\bar{w}})$, and in $\C_w$ a module $M$ is a cluster-tilting module iff it is $\C_w$-maximal rigid iff $\Sigma(M)=\ell(w)$.  
			\end{theorem}
		
\section{Combinatorial notations}
		In the following we will have to study in detail reduced representatives of Weyl group elements and their indices.
		
		\begin{definition}\label{definitionkplus}
			To a product $s_{i_k}\cdots s_{i_1}$ of simple reflections representing the Weyl group element $w$, we associate the integer tuple $\bar{w}=[i_k,\dots,i_1]$. Given a Weyl group element $w$ there exists a minimal number of factors among products representing this element. This number is called the length of $w$, denoted $\ell(w)$. A representative having as few factors as the length is called \emph{reduced}. We will always suppose the representatives to be reduced.
			
			Given a reduced representative $\bar{w}=[i_{\ell(w)},\dots,i_1]$ of $w\in W$ and $1\leq k\leq \ell(w)$ we will define the following notions:
			\begin{itemize}
				\item the \emph{color} of the index $k$ is the numerical value of the integer $i_k$,
				\item the \emph{successor}, $k^+$, of $k$ is the smallest index strictly greater than $k$ such that $k$ and $k^+$ have the same color. If $k$ is the greatest index of its color, then for any $j$ such that $i_j=i_k$ (i.e. of the same color) we note $j_\Max:=k$ and $k^+=\ell(w)+1$,
				\item the \emph{predecessor}, $k^-$, of $k$ is the greatest index strictly smaller than $k$ such that $k$ and $k^-$ have the same color. If $k$ is the smallest index of its color, then for any $j$ of the same color as $k$, we note $j_\Min:=k$ and $k^-=0$.
			\end{itemize}
			
			Two colors are said to be adjacent if the associated vertex of the Dynkin diagram are linked by (at least) one edge.
		\end{definition}
		\begin{ex}
			For instance, if we consider $W$ of type $A_5$ and $w=s_2s_3s_4s_5s_1s_2s_3s_4s_1s_2s_3s_2s_1$, one reduced representative is $\bar{w}=[\textcolor{blue}{2},\textcolor{red}{1},\textcolor{Green}{3},\textcolor{orange}{4},\textcolor{blue}{2},\textcolor{red}{1},\textcolor{Green}{3},\textcolor{violet}{5},\textcolor{blue}{2},\textcolor{orange}{4},\textcolor{Green}{3},\textcolor{blue}{2},\textcolor{red}{1}]=:[i_{13},i_{12},i_{11},i_{10},i_9,i_8,i_7,i_6,i_5,i_4,i_3,i_2,i_1]$
			
			The color of the fifth letter $i_5$ is the integer $\textcolor{blue}{2}$, the color of the tenth is $i_{10}=\textcolor{orange}{4}$.
			
			The successor of the seventh letter is the eleventh: $7^+=11$ as $7$ and $11$ are of color $\textcolor{Green}{3}$ and there is no letter of this colour after the seventh and before the eleventh; its predecessor is $7^-=3$. In the same sense, we have $8^+=12$, $8^-=1$. In order to illustrate the boundary definitions: $6^+=14=\ell(w)+1$ and $6^-=0$ as the sixth letter is the only one of its color ($\textcolor{violet}{5}$).
		\end{ex}
	
\section{Structure of \texorpdfstring{$\Gamma_{V_{\bar{w}}}$}{Gamma V w}}\label{sectionInitialSeedCw}
	\subsection{Definition}\label{sectionInitialQuiverDefinition}
		As explained in Definition \ref{definitionClusterSeed}, the quiver attached to a maximal-rigid basic module $M\in\mod(\Lambda)$ is the quiver of its endomorphism algebra $\End_\Lambda(M)$. This quiver has its vertices labelled by the $\End_\Lambda(M)$-modules $\Hom_\Lambda(M,M_i)$ where $M=\bigoplus M_i$ is a decomposition of $M$ into indecomposable direct summands, and the arrows correspond to irreducible morphisms.
		
		In general it is difficult to determine such a quiver, but thanks to \cite[Section 2.4]{geiss2011KacMoodyGroups} we have a combinatorial description of the quiver $\Gamma_{V_{\bar{w}}}$ (or $\Gamma_{\bar{w}}$). Note that this is a special case of the quivers associated in \cite{berenstein2005ClusterAlgebrasIII} to double Bruhat cells.
		
		The vertices are labelled by the direct summands $V_k$, $1\leq k\leq \ell(w)$. In order to fix a position we will put all vertices on a $n\times \ell(w)$ grid, the vertex labelled by $V_k$ being on the $k$th column (numbered from right to left) and on the $i_k$-th line (numbered from up to bottom).
		
		We will then add arrows by splitting them into two sets:
		\begin{itemize}
			\item the horizontal arrows will go from $k$ to $k^+$ for any $1\leq k\leq\ell(w)$ whenever possible,
			\item the ordinary arrows will go from $k$ to $j$ if $i_j\neq i_k$ and the inequality $j^+\geq k^+>j>k$ is verified.
			
			Then there are $-a_{i_j,i_k}$ arrows $V_{i_j}\rightarrow V_{i_k}$ where $A=(a_{i,j})$ is the Cartan matrix associated to $W$.
		\end{itemize}
		\begin{remark}
			As we are in simply-laced Dynkin cases, we have one arrow $V_{i,j}\rightarrow V_{i,k}$ if $s_{i_k}s_{i_j}s_{i_k}=s_{i_j}s_{i_k}s_{i_j}$ and the inequality is verified and no arrow if $s_{i_j}s_{i_k}=s_{i_k}s_{i_j}$.
		\end{remark}
		\begin{ex}
			Let us consider the element $w$ of the Weyl group $W$ of type $A_4$ and the representative $\bar{w}=[3,4,2,3,1,2,4,1]$. Then we have (indecomposable direct summands ordered by ascending index):
			\[
				V_{\bar{w}}=1\oplus 4\oplus \begin{tikzcd}[column sep=-2pt,row sep=-2pt]1\ar[dr,dash]&\\&2\end{tikzcd}\oplus\begin{tikzcd}[column sep=-2pt,row sep=-2pt]&2\ar[dl,dash]\\1&\end{tikzcd}\oplus\begin{tikzcd}[column sep=-2pt,row sep=-2pt]1\ar[dr,dash]&&&\\&2\ar[dr,dash]&&4\ar[dl,dash]\\&&3&\end{tikzcd}\oplus\begin{tikzcd}[column sep=-2pt,row sep=-2pt]&2\ar[dash,dl]\ar[dash,dr]&&4\ar[dl,dash]\\1\ar[dr,dash]&&3\ar[dl,dash]&\\&2&&\end{tikzcd}\oplus\begin{tikzcd}[column sep=-2pt,row sep=-2pt]1\ar[dr,dash]&&&\\&2\ar[dr,dash]&&\\&&3\ar[dr,dash]&\\&&&4\end{tikzcd}\oplus\begin{tikzcd}[column sep=-2pt,row sep=-2pt]&2\ar[dl,dash]\ar[dr,dash]&&\\1\ar[dr,dash]&&3\ar[dl,dash]\ar[dr,dash]&\\&2\ar[dr,dash]&&4\ar[dl,dash]\\&&3&\end{tikzcd}.
			\]
			
			If we compute the quiver of the endomorphism algebra $\End_\Lambda(V_{\bar{w}})$ we get:
			\[
				\begin{tikzcd}[column sep=-20pt,row sep=5pt]
					&&&&\Hom(V_4,V_{\bar{w}})\ar[dr]&&&\Hom(V_1,V_{\bar{w}})\ar[lll]\\
					&&\Hom(V_6,V_{\bar{w}})\ar[urr]\ar[dr]&&&\Hom(V_3,V_{\bar{w}})\ar[lll]\ar[urr]&&\\
					\Hom(V_8,V_{\bar{w}})\ar[urr]\ar[dr]&&&\Hom(V_5,V_{\bar{w}})\ar[lll]\ar[urr]\ar[drrr]&&&&\\
					&\Hom(V_7,V_{\bar{w}})\ar[urr]&&&&&\Hom(V_2,{\bar{w}})\ar[lllll]&\\
				\end{tikzcd}
			\]
			and thus, the quiver $\Gamma_{\bar{w}}$ is Figure \ref{FigureGammaBarw}.
			\begin{figure}[ht]
			\[
				\begin{tikzcd}[column sep=5pt,row sep=5pt]
					&&&&V_4\ar[dr]&&&V_1\ar[lll]\\
					&&V_6\ar[urr]\ar[dr]&&&V_3\ar[lll]\ar[urr]&&\\
					V_8\ar[urr]\ar[dr]&&&V_5\ar[lll]\ar[urr]\ar[drrr]&&&&\\
					&V_7\ar[urr]&&&&&V_2\ar[lllll]&\\
				\end{tikzcd}
			\]
			\caption{Quiver $\Gamma_{\bar{w}}$}
			\label{FigureGammaBarw}
			\end{figure}
			
			If we follow the steps described above we first draw a $4\times 8$ table in which we put the vertices according to their index and color:
			\begin{center}
			\begin{tabular}{|c|c|c|c|c|c|c|c|c|}
			\hline 
			8 & 7 & 6 & 5 & 4 & 3 & 2 & 1 &  \\ 
			\hline 
			 &  &  &  & $V_4$ &  &  & $V_1$ & \textcolor{red}{1} \\ 
			\hline 
			 &  & $V_6$ &  &  & $V_3$ &  &  & \textcolor{blue}{2} \\ 
			\hline 
			$V_8$ &  &  & $V_5$ &  &  &  &  & \textcolor{Green}{3} \\ 
			\hline 
			 & $V_7$ &  &  &  &  & $V_2$ &  & \textcolor{orange}{4} \\ 
			\hline 
			
			\end{tabular} ,
			\end{center}
			we add the horizontal arrows (between two successive vertices on the same line going from right to left) and, eventually, we add the ordinary arrows and we get Figure \ref{FigureGammaBarw}. For instance we have the arrow $V_5\rightarrow V_2$ as $i_5=\textcolor{Green}{3}\neq\textcolor{orange}{4}=i_2$, $s_{\textcolor{orange}{4}}s_{\textcolor{Green}{3}}s_{\textcolor{orange}{4}}=s_{\textcolor{Green}{3}}s_{\textcolor{orange}{4}}s_{\textcolor{Green}{3}}$ (so $\textcolor{Green}{3}$ and $\textcolor{orange}{4}$ are linked in the Dynkin diagram by one edge) and $5^+=8\geq 2^+=7>5>2$.
			
			There is no arrow $V_6\rightarrow V_1$ as, even if $i_6\neq i_1$ and these colors are related by a braid move, the inequality is not verified:
			\[
				6^+=9\geq 2^+=5\not>6>2
			\]
			
			There is no arrow $V_2\rightarrow V_1$ as $i_2$ and $i_1$ are two colors not linked in the Dynkin diagram and thus generate no arrow, even if the inequality $2^+=7\geq 1^+=4>2>1$ is verified.
		\end{ex}
		
	\begin{definition}[Colored lines]
		Thanks to this way of displaying the quiver, we can now associate a color to each line, the color of the $i_k$-th line being $i_k$.
	\end{definition}
	\subsection{Saw teeth structure}
		We now want to exhibit a remarkable structure on this quiver that we will need later for the proof of our algorithm. We need first to introduce another notion:
		\begin{definition}[Bicolor subquiver]
			Given a quiver $\Gamma$ having a structure of colored lines $c_1,\dots,c_n$, the bicolor subquiver of color $(c_j,c_k)$ is the subquiver of $\Gamma$ obtained by taking:
			\begin{itemize}
				\item all vertices of color $c_j$ or $c_k$
				\item all arrows between vertices both of color $c_j$
				\item all arrows between a vertex of color $c_j$ and one of color $c_k$ (and conversely)
			\end{itemize}
		\end{definition}
		\begin{remark}
			The definition is not symmetric: $c_j$ and $c_k$ do not play the same role.
		\end{remark}
		\begin{ex}
			Taking again Figure \ref{FigureGammaBarw}, we can extract the $(\textcolor{blue}{2},\textcolor{Green}{3})$-bicolor subquiver and get Figure \ref{FigureBicolorSubquiver1} a). We can also look at the $(\textcolor{Green}{3},\textcolor{blue}{2})$-bicolor subquiver of Figure \ref{FigureBicolorSubquiver1} b) and check that they are different.
			\begin{figure}[ht]
			\begin{center}
			\begin{tabular}{cc}$
				\begin{tikzcd}[column sep=5pt]
					&&\textcolor{blue}{V_6}\ar[dr]&&&\textcolor{blue}{V_3}\ar[lll]&&\\
					\textcolor{Green}{V_8}\ar[urr]&&&\textcolor{Green}{V_5}\ar[urr]&&&&
				\end{tikzcd}
			$
			&
			$
				\begin{tikzcd}[column sep=5pt]
					&&\textcolor{blue}{V_6}\ar[dr]&&&\textcolor{blue}{V_3}&&\\
					\textcolor{Green}{V_8}\ar[urr]&&&\textcolor{Green}{V_5}\ar[lll]\ar[urr]&&&&
				\end{tikzcd}
			$\\
			a)&b)
			\end{tabular}
			\end{center}
			\caption{a) $(\textcolor{blue}{2},\textcolor{Green}{3})$-bicolor subquiver of $\Gamma_{\bar{w}}$, b) $(\textcolor{Green}{3},\textcolor{blue}{2})$-bicolor subquiver of $\Gamma_{\bar{w}}$}
			\label{FigureBicolorSubquiver1}
			\end{figure}
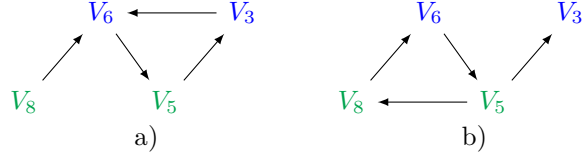
		\end{ex}
		\begin{definition}[Saw teeth]\label{definitionSawTeethStructure}
			Let $\Gamma$ be a bicolor subquiver of colors $(i_k,i_j)$. We call a saw tooth of $\Gamma$, a cycle of $\Gamma$ arrows having the following elements:
			\begin{itemize}
				\item an arrow from a vertex of color $i_k$ (let say $R_k$) to a vertex of color $i_j$ (let say $R_j$)
				\item an arrow from $R_j$ to a vertex of color $i_k$ and of index stricly less than $k$, let say $R_{k^{\alpha -}}$, $\alpha>0$
				\item a sequence of arrows between vertices of color $i_k$: $R_k\leftarrow R_{k^-}\leftarrow\cdots\leftarrow R_{k^{(\alpha-1)-}}\leftarrow R_{k^{\alpha -}}$.
			\end{itemize}
			We then say that $R_{k^{\alpha -}}$ is the right end of the saw tooth, $R_k$ its left end and $R_j$ its summit.
			
			We call successive sequence of saw teeth of length $m$ a collection of $m$ teeth that we can index by $1,\dots,m$ such that the left end of the tooth $p$ is the right end of the tooth $p+1$ $(1\leq p\leq m-1)$. We call right end of the sequence the right end of the tooth $1$ and left end of the sequence the left end of the tooth $m$.
		\end{definition}
		\begin{definition}[Saw teeth structure]
			A bicolor subquiver of colors $(i_k,i_j)$ is said to have a saw teeth structure if it has the following structure (by ascending order of indices on vertices on the $i_k$-th line):
			\begin{enumerate}
				\item a sequence (possibly empty) of vertices of color $i_k$ where successive vertices are linked by a horizontal arrow
					\[
						R_{(k_\Min)^{\alpha +}}\leftarrow R_{(k_\Min)^{(\alpha-1)+}}\leftarrow\cdots\leftarrow R_{(k_\Min)^+}\leftarrow R_{k_\Min}
					\]
					where $\alpha$ is the length of this sequence
				\item a unique arrow, called initial barb, $R_{(k_\Min)^{\alpha+}}\rightarrow R_{(j_\Min)^{\beta +}}$ (optional)
				\item a sequence of $m$ saw teeth of right end $R_{(k_\Min)^{\alpha+}}$ and of left end $R_{(k_\Max)^{\gamma-}}$ with $\gamma\geq \alpha$ (equality if empty)
				\item a unique arrow, called final barb, $R_{(j_\Min)^{\delta +}}\rightarrow R_{(k_\Max)^{\gamma-}}$ with $\delta >\beta$ (optional)
				\item a sequence (possibly empty) of vertices of color $i_k$ where successive vertices are linked by a horizontal arrow
					\[
						R_{k_\Max}\leftarrow R_{(k_\Max)^-}\leftarrow\cdots\leftarrow R_{(k_\Max)^{(\gamma-1)-}}\leftarrow R_{(k_\Max)^{\gamma-}}
					\]
			\end{enumerate}
			and some vertices of color $i_j$ not linked to any vertex of color $i_k$ by the arrows of the bicolor subquiver called isolated vertices.
			
			A saw teeth structure is said to be pure if there is no initial barb.
			
			A quiver $\Gamma$ having a structure of colored lines is said to have a saw teeth structure if, any bicolor subquiver of $\Gamma$ has a saw teeth structure.
		\end{definition}
			\begin{ex}
			We will consider the quiver given by the representative 
			\[
				\bar{w}=[4,3,5,2,3,4,1,2,3,5,3,4,1,2,3,1,2]
			\]
			of $w\in W$ of type $D_5$ (see Figure \ref{figureDynkinDiagrams} p. \pageref{figureDynkinDiagrams} for the Dynkin diagram associated with this type). We have the quiver of Figure \ref{figureQuiverD5}.
			
			\begin{figure}[ht]
			\[
				\begin{tikzcd}[column sep=3pt,row sep =10pt]
					&&&&&&V_{11}\ar[dr]&&&&&&V_5\ar[llllll]\ar[dr]&&&V_2\ar[lll]\ar[dr]&\\
					&&&V_{14}\ar[urrr]\ar[dr]&&&&V_{10}\ar[urrrrr]\ar[dr]\ar[llll]&&&&&&V_4\ar[llllll]\ar[urr]\ar[dr]&&&V_1\ar[lll]\\
					&V_{16}\ar[urr]\ar[drrrr]\ar[ddr]&&&V_{13}\ar[urrr]\ar[ddrrrrr]\ar[lll]&&&&V_9\ar[llll]\ar[urrrrr]\ar[drrr]&&V_7\ar[ll]&&&&V_3\ar[llll]\ar[urr]&&\\
					V_{17}\ar[ur]&&&&&V_{12}\ar[lllll]\ar[urrr]&&&&&&V_6\ar[llllll]\ar[urrr]&&&&&\\
					&&V_{15}\ar[uurr]&&&&&&&V_8\ar[lllllll]\ar[uur]&&&&&&&
				\end{tikzcd}
			\]
			\caption{$\Gamma_{\bar{w}}$}
			\label{figureQuiverD5}
			\end{figure}
			
			This quiver has a saw teeth structure. For instance if we check the $(3,2)$ bicolor subquiver we get Figure \ref{figureBicolorSubquiver3} where the initial and final sequence are empty, the initial barb is dashed and the teeth sequence is of length three, with the tooth $V_3\rightarrow V_7\rightarrow V_9\rightarrow V_4\rightarrow V_3$, the tooth $V_{13}\rightarrow V_{10}\rightarrow V_9\rightarrow V_{13}$ and the tooth $V_{16}\rightarrow V_{14}\rightarrow V_{13}\rightarrow V_{16}$ which are consecutives. Here there is neither final barb nor isolated vertex.
			
			Due to the presence of the initial barb, this subquiver is not of pure saw teeth structure. 
			
			\begin{figure}[ht]
				\[
				\begin{tikzcd}[column sep=10pt]
					&&&V_{14}\ar[dr]&&&&V_{10}\ar[dr]&&&&&&V_4\ar[dr]&&&V_1\\
					&V_{16}\ar[urr]&&&V_{13}\ar[urrr]\ar[lll]&&&&V_9\ar[llll]\ar[urrrrr]&&V_7\ar[ll]&&&&V_3\ar[llll]\ar[dashed,urr]&&
				\end{tikzcd}
			\]
			\caption{$(3,2)$ bicolor subquiver}
			\label{figureBicolorSubquiver3}
			\end{figure}
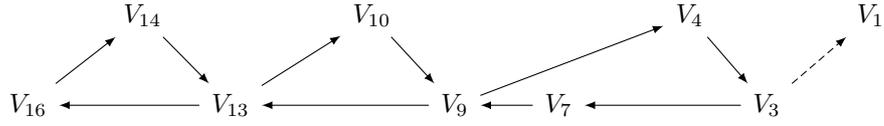
			
			The reader will be able to see a non-empty initial sequence by looking at the $(3,5)$ bicolor subquiver, both an initial and a final barb (as well as an isolated vertex) by looking at the $(5,3)$ bicolor subquiver. 
		\end{ex}
		\begin{cex}
			Let suppose we have the following quiver:
			
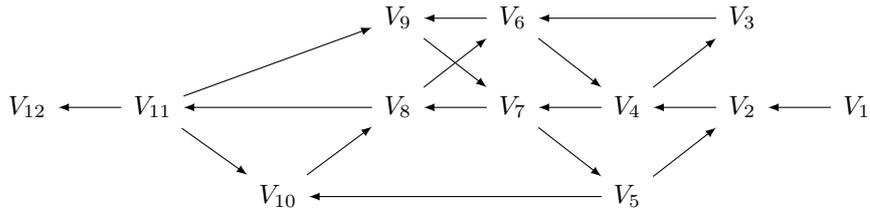
\[
				\begin{tikzcd}
					&&&V_9\ar[dr]&V_6\ar[l]\ar[dr]&&V_3\ar[ll]&\\
					V_{12}&V_{11}\ar[urr]\ar[dr]\ar[l]&&V_8\ar[ll]\ar[ur]&V_7\ar[l]\ar[dr]&V_4\ar[l]\ar[ur]&V_2\ar[l]&V_1\ar[l]\\
					&&V_{10}\ar[ur]&&&V_5\ar[ur]\ar[lll]&&
				\end{tikzcd}
			\]
			where the first and second lines and second and third lines are of adjacent colors and first and third are not adjacent. Then the $(1,2)$ (Figure \ref{figureBicolorSubquiver5}) and the $(2,3)$ bicolor subquivers (Figure \ref{figureBicolorSubquiver6}) (among others) do not have a saw teeth structure.

			\begin{figure}[h!t]
				\[
				\begin{tikzcd}
					&&&V_9\ar[dr]&V_6\ar[dr]&&V_3&\\
					V_{12}&V_{11}\ar[urr]\ar[l]&&V_8\ar[ll]\ar[ur]&V_7\ar[l]&V_4\ar[l]\ar[ur]&V_2\ar[l]&V_1\ar[l]
				\end{tikzcd}
			\]
			\caption{$(1,2)$-bicolor subquiver}
			\label{figureBicolorSubquiver5}
			\end{figure}
			\begin{figure}[h!t]
				\[
				\begin{tikzcd}
					V_{12}&V_{11}\ar[dr]\ar[l]&&V_8\ar[ll]&V_7\ar[l]\ar[dr]&V_4\ar[l]&V_2\ar[l]&V_1\ar[l]\\
					&&V_{10}\ar[ur]&&&V_5\ar[ur]&&
				\end{tikzcd}
			\]
			\caption{$(2,3)$-bicolor subquiver}
			\label{figureBicolorSubquiver6}
			\end{figure}
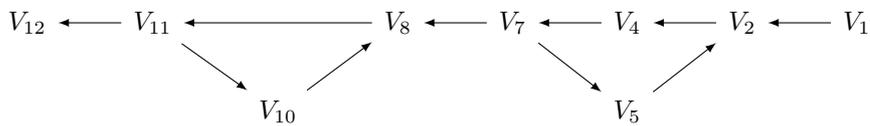
			
			Indeed, on the $(2,1)$ bicolor subquiver, there are some ordinary arrows after the final barb $V_8\rightarrow V_6$, while on the $(1,2)$ bicolor subquiver, the tooth $V_8\rightarrow V_6\rightarrow V_4\rightarrow\cdots$ ends after the tooth $V_{11}\rightarrow V_9\rightarrow V_7$ has begun. In the $(2,3)$ bicolor subquiver, the left end of the first tooth $(V_7)$ is not the right end of the second tooth $(V_8)$. In these two latter cases the teeth are not consecutives.
		\end{cex}
		
		\begin{proposition}\label{propositionStabilitySawTeethStructure}
			If a bicolor subquiver $(i_k,i_j)$ has a saw teeth structure then this structure is preserved by removing the first vertex of any line.
		\end{proposition}
		\begin{remark}
			This does not hold in general for the pure saw teeth structure.
		\end{remark}
		\begin{proof}
			We will proceed by case exhaustion. Let us call $R_{j_\Mini}$ the vertex of line $i_j$ of minimal index and $R_{k_\Mini}$ the one of line $i_k$. 
			
			If we remove $R_{j_\Mini}$ we have 4 possible cases:
			\begin{enumerate}
				\item $R_{j_\Mini}$ is an isolated vertex: the structure is preserved
				\item $R_{j_\Mini}$ is the target of an initial barb: the initial barb is removed and now the subquiver starts with a sequence of teeth, a final barb or a sequence of vertices. The structure is preserved.
				\item $R_{j_\Mini}$ is the summit of the first tooth: removing it removes the first tooth and now the line starts by a one vertex more sequence of vertices. There cannot have an initial barb here by minimality of $R_{j_\Mini}$. The structure is preserved.
				\item $R_{j_\Mini}$ is the source of the final barb. In that case it is the only ordinary arrow of the toothless bicolor subquiver and thus after removal, the bicolor subquiver consists in a set of isolated vertices on line $i_j$ and a sequence of vertices on line $i_k$. The structure is preserved.
				\end{enumerate}
				So, in any case, removing $R_{j_\Mini}$ will preserve the struture.
				
				Now we study the effect of removing $R_{k_\Mini}$. We have once again 2 possible cases:
				\begin{enumerate}
					\item $R_{k_\Mini}$ is the first vertex of a vertices sequence linked by consecutives horizontal arrows. Then the sequence is of length one less (or disappeared). The structure is preserved.
					\item $R_{k_\Mini}$ is the source of an initial barb or the right end of the teeth sequence or the target of the final barb (the last two possibilities being incompatible). Then the initial barb disappears (provided it existed) and the line starts by a new initial barb (resulting in the subsequent deletion of the right ordinary arrow of the first tooth provided it existed), or a sequence of horizontal arrows in a subquiver now with no ordinary arrows (initial and final barb being the only ones and both deleted by the vertex deletion). The structure is preserved.				
				\end{enumerate}
				Because of the possible apparition of initial barb due to the deletion of the right side of the initial teeth, even if there was no initial barb before the deletion in the last case, we cannot ensure that the pure saw teeth structure is preserved.
			
		\end{proof}
		\subsection{Saw teeth structure in \texorpdfstring{$\C_w$}{Cw} initial seeds}
			The point of introducing the previous notions was the following proposition that will be of use for the proof of the algorithm.
			
			\begin{proposition}\label{propositionSawTeethStructureVwbar}
				For any representative $\bar{w}$ of an element $w\in W$, $\Gamma_{\bar{w}}$ has a saw teeth structure.
			\end{proposition}
			\begin{proof}
				The proof will rely on the combinatorial description of the quiver, given in Section \ref{sectionInitialQuiverDefinition}.
				
				We need to look at any bicolor subquiver of $\Gamma_{\bar{w}}$ to prove the proposition. First, if the two chosen colors are not adjacent (in the sense of Definition \ref{definitionkplus}) then the structure is trivial: there is no ordinary arrow in the subquiver thus all vertices of the first color form an initial sequence, there is not tooth, no barb and all vertices of the second color are isolated vertices. The saw teeth structure is verified.
				
				Let us then consider a bicolor subquiver of adjacent colors. As for the description of the structure, we will here follow the first line (of color $i_k$ and of vertices $k_\Min,\dots,k_\Max$ to fix the notation) by ascending order of indices.
				
				We start by considering the special case $k_\Max<j_\Min$ and $j_\Min\neq j_\Max$, or $k_\Min>j_\Max$ and $k_\Min\neq k_\Max$. Then there is no ordinary arrows and we are again in the trivial case. Apart from this case, while $k^+<j_\Min$ there is no ordinary arrow allowed between $k$ and $j_\Min$ and we are in the initial sequence.
				
				If we have $k_\Min>j_\Min$ then there exists an integer $\gamma$ such that $(j_\Min)^{(\gamma+1)+}>k_\Min>(j_\Min)^{\gamma+}$. Then there is a maximal index $k$ (possibly greater than $k_\Min$) such that $k^+\geq (j_\Min)^{(\gamma+1)+}>k>(j_\Min)^{\gamma+}$. Then there is an ordinary arrow $k\rightarrow (j_\Min)^{\gamma+}$ and the vertices of color $i_j$ and of index $<(j_\Min)^{\gamma+}$ are isolated vertices. This is the initial barb.
				
				If we have $k_\Min>j_\Min$ (so no initial barb) there is a minimal $k$ such that $k^+>j_\Min>k$. Let us call $j$ the greatest index such that $k^+>j>k$ (and thus $j^+\geq k^+$). If we have an initial barb, let $k$ be the source of the initial barb and $j$ the successor of its target. 
				
				($\ast$) In both cases we then have $j^+\geq k^+>j>k$ and then there is an arrow $V_j\rightarrow V_k$. If $j=j_\Max$ then there is no more ordinary arrow beyond this point and this arrow is a final barb and all following vertices of line $i_k$ will form the final sequence. In the opposite case, this arrow will be the right side of the first tooth of the teeth sequence.
				
				In the case where this is not the final barb, there exists a $\gamma>0$ such that $k^{(\gamma+1)+}\geq j^+>k^{\gamma+}>j$ and then we have an ordinary arrow $k^{\gamma+}\rightarrow j$ closing the teeth $k^{\gamma+}\rightarrow j\rightarrow k\rightarrow k^+\rightarrow\cdots\rightarrow k^{\gamma+}$. 
				
				Given these $k$ and $j$ either there is a $\zeta>0$ such that $j^{(\zeta+1)+}\geq k^{(\zeta+1)+}>j^{\zeta +}>k^{\gamma +}$ or not. If it is not true, then we have considered all ordinary arrows and vertices of color $i_j$ and the last remaining $i_k$ vertices form the final sequence. If such a $\zeta$ exist, we start again the reasoning at the point ($\ast$) for the formation of the next tooth or the final barb.
				
				Eventually we end with a possibly empty final sequence and then our subquiver has a saw teeth structure.
			\end{proof}
\section{\texorpdfstring{$\Delta$}{Delta}-vectors}\label{sectionDeltaVectors}
	$\Delta$-vectors are a notion coming from Ringel's work on quasi-hereditary algebras (\cite{ringel1991CategoryModulesGood}) which plays a crucial role in our paper. We present it here in the narrower context of our category $\C_w$.
	\subsection{Definition}
		In the quiver of the algebra $\End_\Lambda(\Gamma_{\bar{w}})$, the horizontal arrows correspond to irreducible morphisms which are in fact inclusions. These inclusions still hold for our $\Gamma_{\bar{w}}$ in the sense that $V_{k}\subset V_{k^+}$ for all $1\leq k\leq \ell(w)$ such that $1\leq k^+\leq \ell(w)$.
		
		This allows us to define a new collection of modules, related to an initial seed $V_{\bar{w}}$.
		\begin{definition}[Modules $M_{k,\bar{w}}$]
			Given a module from an initial seed $V_{\bar{w}}$, we define:
			\[
				M_{k,\bar{w}}:=
				\left\{\begin{array}{lc}V_k&\text{if }k=k_\Min\\V_k/V_{k^-}&\text{otherwise}\end{array}\right.,\quad 1\leq k\leq \ell(w),\quad \text{and }M_{\bar{w}}=\bigoplus\limits_{k=1}^{\ell(w)} M_{k,\bar{w}}.
			\]
		\end{definition}
		
		Following \cite[Section 10]{geiss2011KacMoodyGroups}, we denote $\C_{M_{\bar{w}}}$ the category of all modules having a stratification by the modules $(M_{k,\bar{w}})_k$ in ascending order.
		
		\begin{lemma}[{\cite[Lemma 10.2]{geiss2011KacMoodyGroups}}]\label{lemmaStratificationCw}
			We have $\C_w=\C_{M_{\bar{w}}}$. Every module of $\C_w$ admits a stratification by the family of modules $(M_{k,\bar{w}})_k$ in ascending order.
		\end{lemma}

		\begin{definition}
			Let $B_{\bar{w}}:=\End_\Lambda(V_{\bar{w}})$ and define the standard modules of $B_{\bar{w}}$ by:
			\[
				\Delta_k:=\Hom(V_{\bar{w}},M_k), 1\leq k\leq \ell(w)
			\]
		\end{definition}
		The following results will allow us to introduce $\Delta$-vectors.
		\begin{proposition}[{\cite[Lemmas 10.3 \& 10.4]{geiss2011KacMoodyGroups}}]	
			The dimension vectors $\dim_{B_{\bar{w}}}(\Delta_k)$, $1\leq k\leq r$ are linearly independent and we have
			\[
				\dim_{B_{\bar{w}}}(\Hom_\Lambda(V_{\bar{w}},V_k))=\dim_{B_{\bar{w}}}(\Delta_k)+\dim_{B_{\bar{w}}}(\Delta_{k^-})+\cdots+\dim_{B_{\bar{w}}}(\Delta_{k_\Min})
			\]
		\end{proposition}
		
		\begin{proposition}[{\cite[Proposition 10.5]{geiss2011KacMoodyGroups}}]\label{propositionEquivGLS11}
			Let $X\in\C_w$ and $\underline{a}=(a_1,\dots,a_{\ell(w)})$ be. The following assumptions are equivalent:
				\begin{enumerate}
					\item $X\in \C_{M_{\bar{w}}}$ with the sequence $\underline{a}$ as multiplicity of the stratas;
					\item there exists a short exact sequence: $0\rightarrow \bigoplus\limits_{k=1}^{\ell(w)}V_{k^-}^{a_k}\rightarrow\bigoplus\limits_{k=1}^rV_k^{a_k}\rightarrow X\rightarrow 0;$
					\item $\dim_{B_{\bar{w}}}(\Hom_{\Lambda}(V_{\bar{w}},X))=\dim_{B_{\bar{w}}}(\Hom_{\Lambda}(V_{\bar{w}},M_{\bar{w}}(X)))=\sum\limits_{k=1}^r a_k\dim_{B_{\bar{w}}}(\Delta_k)$ where 
				\[ M_{\bar{w}}(X):= M_1^{a_1}\oplus\cdots\oplus M_{\ell(w)}^{a_{\ell(w)}}.\]
				\end{enumerate}
			
		\end{proposition}
		\begin{proposition}[{\cite[Corollary 12.3]{geiss2011KacMoodyGroups}}]
			Let $X$ and $Y$ be indecomposable rigid modules in $\C_w$. If 
			$
				\dim_{B_{\bar{w}}}(\Hom_{\Lambda}(V_{\bar{w}},X))=\dim_{B_{\bar{w}}}(\Hom_\Lambda(V_{\bar{w}},Y))
			$
			then $X\cong Y$.
		\end{proposition}
		Thus, thanks to these propositions we can unambiguously determine an indecomposable rigid $\Lambda$-module $X\in\C_w$ by the $\dim_{B_{\bar{w}}}$-vector of $\Hom_{\Lambda}(V_{\bar{w}},X)$ or equivalently by the corresponding vector $(a_1,\dots,a_{\ell(w)})$, which gives rise to the following definition.
		\begin{definition}
			The previous sequence $\underline{a}=(a_1,\dots,a_{\ell(w)})$ is called the $\Delta_{\bar{w}}$-vector of $X$, written $\Delta_{\bar{w}}(X)=(a_1,\dots,a_{\ell(w)})$. We will write $(\Delta_{\bar{w}}(X))_i=\Delta_{\bar{w},i}(X)$ for the $i$th coordinate of the $\Delta_{\bar{w}}$ vector of $X$.
		\end{definition}
	\subsection{Mutation of \texorpdfstring{$\Delta$}{Delta}-vectors}
		In \cite{geiss2011KacMoodyGroups} the authors give an explicit way to compute $\Delta$-vectors after a mutation. We recall here the formula.
		\begin{definition}
			Given the family of modules $(M_{k,\bar{w}})_{k=1}^{\ell(w)}$ we define the vector 
			\[
				d_\Delta:=(\dim(\Delta_1),\dots,\dim(\Delta_{\ell(w)})).
			\]
			where $\dim(\Delta_s)=\sum\limits_{k=1}^{\ell(w)}\dim\Hom_{\Lambda}(V_k,M_s)$.
		\end{definition}
		\begin{proposition}[{\cite[Proposition 12.6]{geiss2011KacMoodyGroups}}]\label{propositionMutationDeltaVectors}
			In a seed $(R,\Gamma)$, the $\Delta$-vector of $R_{k}^\ast$, image of $R_{k}$ by $\mu_k$ ($1\leq k\leq \ell(w)$) is:
			\[
				\Delta_{\bar{w}}(R_k^\ast):=\left\{\begin{array}{cl}-\Delta_{\bar{w}}(R_k)+\sum\limits_{(R_i\rightarrow R_k)\in\Gamma}\Delta_{\bar{w}}(R_i)&\text{if }\sum\limits_{(R_i\rightarrow R_k )\in\Gamma}\Delta_{\bar{w}}(R_i)\cdot d_\Delta>\sum\limits_{(R_k\rightarrow R_j)\in\Gamma} \Delta_{\bar{w}}(R_k)\cdot d_{\Delta}\\-\Delta_{\bar{w}}(R_k)+\sum\limits_{(R_k\rightarrow R_j)\in\Gamma}\Delta_{\bar{w}}(R_j)&\text{otherwise}\end{array}\right.
			\]
			where $\Delta_{\bar{w}}(M)\cdot d_\Delta=\sum\limits_{i=1}^{\ell(w)}\Delta_{\bar{w},i}(M)d_{\Delta,i}$ is the usual inner product.
		\end{proposition}
		\begin{ex}
			Examples of such computation are given in Example \ref{exampleAlgo2}.
		\end{ex}
	\subsection{\texorpdfstring{$\Delta$}{Delta}-vectors related to different representatives}\label{sectionDeltaVectorRepresentatives}
		In our definition of $\Delta$-vector, we have to choose a particular reduced representative, $\bar{w}$, defining our initial seed $V_{\bar{w}}$ and the family of modules $(M_{\bar{w},k})_k$. But given another representative $\bar{w}_2$ of the same Weyl group element $w$ we can define another family of module $(M_{\bar{w}_2,k})_k$ defining another notion of $\Delta_{\bar{w}_2}$-vector. 
		
		Clearly we have:
		\begin{proposition}\label{propositionIntervalModules}
			We have $\Delta_{\bar{w}}(V_{\bar{w},k})=\sum\limits_{j=1}^{k}\delta_{i_k,i_j}e_j$ where $e_j$ is the $j$-th vector of the canonical basis.
		\end{proposition}
		\begin{proof}
			By definition of the modules $(M_{\bar{w},k})_k$.
		\end{proof}
		
		However it is far more difficult to compute the $\Delta_{\bar{w}_1}$-vector of the indecomposable summands of the initial seed $V_{\bar{w}_2}$ when $\bar{w}_2\neq\bar{w}_1$.
		
		There exists an algorithm to obtain it, based on the fact that if we know a rewriting path from $\bar{w}_2$ to $\bar{w}$, we can decompose this rewriting path in $2$ and $3$-moves (commutation of letters or braid moves) thanks to Matsumoto's lemma \cite{matsumoto1964generateurs}. Then \cite[Propositions 5.26 \& 5.27]{baumann2014AffineMirkovicVilonenPolytopes} gives us a way to compute changes in $\Delta$-vectors and \cite[Theorem 3.5]{shapiro2000SimplyLacedCoxeter} changes in the quiver. It is the idea used in the algorithmic implementation \cite{menard2021ImplementationMenardAlgorithm} of this article's algorithm. 
		
		The drawback of this method is that we do not have a closed formula for the output, other than making the computation. However, we can obtain a closed formula for some particular $\Delta$-coordinates using another method relying on Mirkovi\'c-Vilonen polytopes that we now introduce.
	\subsection{Relation with Mirkovi\'c-Vilonen polytopes}
		In \cite[Section 8]{kashiwara1995CrystalBases}, Kashiwara introduces the crystal $B(\infty)$ as the crystal base of the negative part of the quantized universal envelopping algebra of a Lie algebra $\mathfrak{g}$, $U_q^-(\mathfrak{g})$. Lusztig then gives a more geometric description of $B(\infty)$, linking it with representations of preprojective algebras in \cite{lusztig1990CanonicalBasesArising}. 
		
		More precisely, if we take a quiver $Q=(Q_0,Q_1)$, being an orientation of a Dynkin quiver associated to the root system $\Phi$ of the Lie algebra $\mathfrak{g}$, and a dimension vector $\nu\in\N Q_0$, any $\Lambda$-module can be seen as equivalent to the data of its linear maps. This data can been seen as points of a variety called Lusztig nilpotent variety and denoted by $\Lambda(\nu)$. We denote by $\mathfrak{B}(\nu)$ the set of irreducible components of $\Lambda(\nu)$ and:
		\[
			\mathfrak{B}:=\bigsqcup_{\nu\in\N Q_0}\mathfrak{B}(\nu)
		\]
		\begin{theorem}[{\cite[Theorem 5.3.2]{kashiwara1997GeometricConstructionCrystal} \& \cite[Theorem 1.8]{lusztig1990CanonicalBasesArising}}]
			$\mathfrak{B}$ has a crystal structure and there exists a unique crystal isomorphism:
			\[
				\Phi:\begin{array}{ccc}
					B(\infty)&\rightarrow&\mathfrak{B}\\
					b&\mapsto&\Lambda_b
				\end{array}
			\]
		\end{theorem}
		
		Any irreducible component of $\mathfrak{B}$ can also be associated with a formal sum of positive roots of $\Phi$ given an indexation of elements of $B(\infty)$ by Poincaré-Birkhoff-Witt bases.
		
		 On the other side, Kamnitzer in \cite{kamnitzer2007CrystalStructureSet} associated to each element of $B(\infty)$ a Mirkovi\'c-Vilonen polytope following an idea of Anderson-Mirkovi\'c. Thus we have two bijections linking $B(\infty)$ with irreducible components of the varieties of quiver representations on one hand and with Mirkovi\'c-Vilonen polytopes on the other hand. In \cite{baumann2012PreprojectiveAlgebrasMV}, Baumann and Kamnitzer study directly the connection between Mirkovi\'c-Vilonen polytopes and Luzstig's nilpotent varieties. More precisely they introduce the Harder-Naramsimhan polytopes as:
		 \begin{definition}[\cite{baumann2014AffineMirkovicVilonenPolytopes}]
			The Harder-Naramsimhan polytope $\Pol(T)$ of a $\Lambda$-module $T$ is the convex hull of the dimension vectors of all submodules of $T$ in $\mathbb{R} I$ (where $I$ is the set of vertices in the Dynkin diagram). We denote it $\Pol(T)$.
			\end{definition}
			and they prove:
			\begin{theorem}[{\cite[Section 1.3]{baumann2014AffineMirkovicVilonenPolytopes}}]
			For each $b\in\mathcal{B}(\infty)$ the map $T\mapsto \Pol(T)$ takes a generic value $\Psi(b)$ on the irreducible component $\Lambda_b$. Moreover the map $\Psi\circ\Phi^{-1}:\mathfrak{B}\rightarrow\mathcal{MV}(\infty)$ is a crystal isomorphism.
			\end{theorem}
			where $\mathcal{MV}(\infty)$ is the set of Mirkovi\'c-Vilonen polytopes starting at the origin of the underlying space.
			
			Then, the crystal isomorphism $\Psi\circ\Phi^{-1}:\mathfrak{B}\rightarrow \mathcal{MV}(\infty)$ associate to any irreductible component of the Lusztig variety a Mirkovi\'c-Vilonen polytope.			
		
		Finally, in \cite[Theorem 4.15]{naito2009MirkovicVilonenPolytopes}, Naito and Sagaki give a combinatorial description of some Mirkovi\'c-Vilonen polytopes corresponding to extremal vectors of irreducible representations of $\mathfrak{g}$. Given an irreducible representation $L(\lambda)$ with crystal basis $B(\lambda)$ seen as a subset of $B(\infty)$, the extremal weight vectors are the elements of the orbit of the highest weight $W\cdot\lambda$ by the Weyl group. To the datum $(\lambda,w)$ we can associate an extremal weight vector corresponding to a Mirkovi\'c-Vilonen polytope that Naito and Sagaki compute explicitly. To the same datum, Geiss, Leclerc and Schröer associate an indecomposable summand $V_k\in V_{\bar{w}}$ where $\lambda$ is the fundamental weight corresponding to the socle of $V_k$ and $w=s_{i_k}\cdots s_{i_1}$ by \cite[Proposition 9.1]{geiss2011KacMoodyGroups}. Examples and concrete computations can be found in \cite[Sections 2.2 \& 2.9]{menard2021AlgebresAmasseesAssociees}
		
		We will now give an explicit interpretation of this reasoning restricting the combinatorial description of \cite[Theorem 4.15]{naito2009MirkovicVilonenPolytopes} to only two paths along edges of the Mirkovi\'c-Vilonen polytope.

	\subsection{Explicitation of the isomorphism}\label{sectionExplicitation}
		
		In the following we will not only work with two reduced representatives of the same Weyl group element but with two reduced representatives of distincts elements. In order to do so we will in fact see any of these elements as "left parts" of reduced representatives of $w_0$.
		\begin{definition}
			Given a representative $\bar{w}=[i_{\ell(w)},\dots,i_1]$ of $w\in W$, we call \emph{left-completion of} $\bar{w}$ in $w_0$ a reduced representative $\dot{w}=[j_r,\dots,j_1]$ of $w_0$ such that $i_k=j_k$ for all $1\leq k\leq\ell(w)$.
		\end{definition}
		\begin{remark}
			This completion exists due to the lattice structure of $W$ for the weak left Bruhat order.
			
			The completion $\dot{w}$ is not unique in general, and we just pick one of them and fix it.
		\end{remark}
		\begin{definition}
			Given $\bar{w}$ and $\dot{w}=[i_r,\dots,i_1]$, let $1\leq k\leq\ell(w)$. We denote by $u_k\in W$ the element defined by:
			\[
				u_k=\left\{\begin{array}{cc}s_{i_r}\cdots s_{i_{k+1}}&\text{if }1\leq k\leq\ell(w)\\e& k=r\end{array}\right.
			\]
			
			Given $\dot{w}_0=[j_r,\dots,j_1]$ another representative of $w_0$, we define $[j_{q_{r-k}},\dots,j_{q_1}]$ to be the leftmost representative of $u_k$ in $\dot{w}_0$ and we write 
			\[
				[j_{r_k},\dots,j_{r_1}]=[j_r,\dots,j_1]\setminus[j_{q_{r-k}},\dots,j_{q_1}]
			\]
			for the set of letters of $\dot{w}_0$ not taking part in the writing of the leftmost representative of $u_k$ in $\dot{w}_0$.
		\end{definition}
		\begin{ex}
			Let $W$ be of type $A_3$, $\dot{w}=[2,1,2,3,2,1]$, $\dot{w}_0=[1,2,3,1,2,1]$ and $k=4$. Then $u_4=s_2s_1$, the leftmost representative of $u_4$ in $\dot{w}_0$ uses the letters of indices $q_1=3$ and $q_2=5$. Then $r_1=1, r_2=2,r_3=4$ and $r_4=6$.
		\end{ex}
		\begin{proposition}[{\cite[Section 1.7]{humphreys1990ReflectionGroupsCoxeter}}]
			Let $\bar{w_0}=[i_r,\dots,i_1]$ be a reduced representative of $w_0\in W$. Then, given the simple roots $(\alpha_i)_{i\in I}$ and the simple reflections $(s_j)_{j\in I}$ in the root system $\Phi$ associated to $W$, the roots of the sequence $(\beta_i)$:
			\[
				\beta_i=\left\{\begin{array}{lc}
					\beta_1=\alpha_{i_1}&\\
					\beta_k=s_{i_1}\cdots s_{i_{k-1}}(\alpha_{i_k})&2\leq k\leq r
				\end{array}\right.
			\]
			are pairwise distincts and form the set of all positive roots of $\Phi$
		\end{proposition}
		
			Let $\varpi_i$ be the $i$th fundamental weight associated to the coroot $\alpha_i^\ast$ of the underlying root system $\Phi$ of $W$.
			
			\begin{definition}[{\cite[Section 4.1]{naito2009MirkovicVilonenPolytopes}}]\label{defXi}
			With the previous notations, we define a sequence of weights associated to a representative $\dot{w}_0$ of $w_0$ $(\xi_i^{\dot{w}_0})_i$ in the following way. Let $\xi_0^{\dot{w}_0}=\varpi_{i_k}$ and $r_0=0$. We have
			\[
				\left\{\begin{array}{lc}\xi_{r_c}^{\dot{w}_0}:=s_{\beta_{r_c}^{\dot{w}_0}}s_{\beta_{r_{c-1}}^{\dot{w}_0}}\cdots s_{\beta_{r_1}^{\dot{w}_0}}(\varpi_{i_k})&1\leq c\leq k\\\xi_{q_d}^{\dot{w}_0}:=\xi_{r_c}^{\dot{w}_0}&\text{if }r_{c+1}>q_d>r_c\end{array}\right.
			\]
		\end{definition}
		We can now express the $\Delta$-coordinates of any direct summand $V_{k,\bar{w}}$ of $V_{\bar{w}}$, seen as a direct summand of the module $V_{\dot{w}}\supset V_{\bar{w}}$.
		
		\begin{proposition}[{\cite[Lemma 4.1.3]{naito2009MirkovicVilonenPolytopes}\cite[Example 5.14 \& Proposition 5.24]{baumann2014AffineMirkovicVilonenPolytopes}}]\label{propositionBKT}
			With the previous notations we denote $n_i^{\dot{w}_0}$, $1\leq i\leq \ell(w)$ the coefficient:
			\[
				\xi_{i-1}^{\dot{w}_0}-\xi_{i}^{\dot{w}_0}=n_i^{\dot{w}_0}\beta_i^{\dot{w}_0}
			\]
			and we have $\Delta_{\dot{w}_0,i}(V_{\bar{w},k})=n_i^{\dot{w}_0}$.
		\end{proposition}
		\begin{ex}
			Given $W$ of type $A_3$, we will look at two representatives of $w_0$: $\dot{w}=[2,1,2,3,2,1]$, $\dot{w}_0=[1,2,3,1,2,1]$. We have the quivers given at Figure \ref{figureQuiverTwoRepresentatives} with the modules given by the socle decomposition of Table \ref{tableSocleDecompositionTwoModules}.
			
			\begin{figure}[ht]
				\begin{center}
					\begin{tabular}{cc}
						$
						\begin{tikzcd}[row sep=5pt]
							&5\ar[dr]&&&&1\ar[llll]\\
							6\ar[ur]\ar[drrr]&&4\ar[urrr]\ar[ll]&&2\ar[ll]&\\
							&&&3\ar[ur]&&
						\end{tikzcd}
						$
						&
						$
						\begin{tikzcd}[row sep=5pt]
						6\ar[dr]&&&3\ar[dr]\ar[lll]&&1\ar[ll]\\
						&5\ar[urr]\ar[dr]&&&2\ar[lll]\ar[ur]&\\
						&&4\ar[urr]&&&
						\end{tikzcd}
						$\\
						a)&b)
					\end{tabular}
					\caption{a) $\Gamma_{\dot{w}}$, b) $\Gamma_{\dot{w}_0}$}
					\label{figureQuiverTwoRepresentatives}
				\end{center}
			\end{figure}
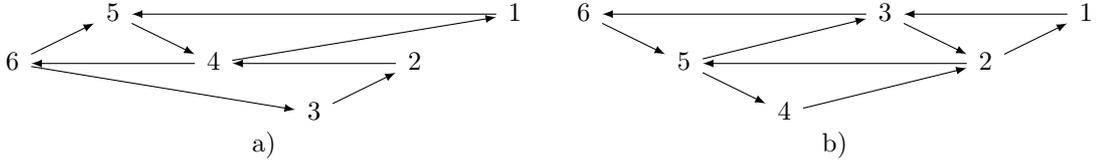
			
			\begin{table}[ht]
				\begin{center}
				$
					\begin{array}{|c||c|c||c||c|c|}
					\hline
						k&V_{\dot{w},k}&V_{\dot{w}_0,k}&k&V_{\dot{w},k}&V_{\dot{w}_0,k}\\
					\hline
						1&\textcolor{red}{1}&\textcolor{red}{1}&4&\begin{tikzcd}[row sep=-3pt, column sep=-3pt]1\ar[dr,dash]&&\textcolor{red}{3}\ar[dl,dash]\\&2&\end{tikzcd}&\begin{tikzcd}[red,row sep=-3pt, column sep=-3pt]1\ar[dr,dash]&&\\&2\ar[dr,dash]&\\&&3\end{tikzcd}\\
					\hline
						2&\begin{tikzcd}[red,row sep=-3pt, column sep=-3pt]1\ar[dr,dash]&\\&2\end{tikzcd}&\begin{tikzcd}[red,row sep=-3pt, column sep=-3pt]1\ar[dr,dash]&\\&2\end{tikzcd}&5&\begin{tikzcd}[row sep=-3pt, column sep=-3pt]&&\textcolor{red}{3}\ar[red,dl,dash]\\&\textcolor{red}{2}\ar[dl,dash]&\\1&&\end{tikzcd}&\begin{tikzcd}[row sep=-3pt, column sep=-3pt]&\textcolor{red}{2}\ar[dl,dash]\ar[dr,dash,red]&\\1\ar[dr,dash]&&\textcolor{red}{3}\ar[dl,dash]\\&2&\end{tikzcd}\\
					\hline
						3&\begin{tikzcd}[red,row sep=-3pt, column sep=-3pt]1\ar[dr,dash]&&\\&2\ar[dr,dash]&\\&&3\end{tikzcd}&\begin{tikzcd}[row sep=-3pt, column sep=-3pt]&\textcolor{red}{2}\ar[dl,dash]\\1&\end{tikzcd}&6&\begin{tikzcd}[row sep=-3pt, column sep=-3pt]&\textcolor{red}{2}\ar[dl,dash]\ar[dr,dash]&\\1\ar[dr,dash]&&3\ar[dl,dash]\\&2&\end{tikzcd}&\begin{tikzcd}[row sep=-3pt, column sep=-3pt]&&\textcolor{red}{3}\ar[dl,dash]\\&2\ar[dl,dash]&\\1&&\end{tikzcd}\\
					\hline
					\end{array}
				$
				\end{center}
				
				\caption{Socle decomposition of indecomposable summands of  $V_{\dot{w}}$ and $V_{\dot{w}_0}$}
				\label{tableSocleDecompositionTwoModules}
			\end{table}
			
			Thanks to Proposition \ref{propositionIntervalModules} we have the $\Delta$ vectors of left hand side of Table \ref{tableInitialDeltaVectors} (seen as a linear combination of the canonical basis $(e_i)_i$ or $(f_j)_j$ of $\Z^r$).
			
			Using the socle decomposition of $M$-modules given by red parts of summands of Table \ref{tableSocleDecompositionTwoModules} we can find the $\Delta_{\dot{w}_0}$-vectors of the indecomposable direct summands of $V_{\dot{w}}$ and, conversely, the $\Delta_{\dot{w}}$-vectors of the indecomposable direct summands of $V_{\dot{w}_0}$.
			
			\begin{table}[h!t]
				\begin{center}
				\begin{tabular}{|c||c|c||c|c|}
					\hline
						$k$&$\Delta_{\dot{w}}(V_{\dot{w},k})$&$\Delta_{\dot{w}_0}(V_{\dot{w}_0,k})$&$\Delta_{\dot{w}_0}(V_{\dot{w},k})$&$\Delta_{\dot{w}}(V_{\dot{w}_0,k})$\\
					\hline
						1&$e_1$&$f_1$&$f_1$&$e_1$\\
					\hline
						2&$e_2$&$f_2$&$f_2$&$e_2$\\
					\hline
						3&$e_3$&$f_1+f_3$&$f_4$&$e_1+e_6$\\
					\hline
						4&$e_2+e_4$&$f_4$&$f_2+f_6$&$e_3$\\
					\hline
						5&$e_1+e_5$&$f_2+f_5$&$f_1+f_3+f_6$&$e_2+e_4+e_6$\\
					\hline
						6&$e_2+e_4+e_6$&$f_1+f_3+f_6$&$f_2+f_5$&$e_1+e_5$\\
					\hline
				\end{tabular}
				\end{center}
				\caption{$\Delta$-vectors directly computable (first two) and computed by hand (last two)}
				\label{tableInitialDeltaVectors}
			\end{table}

			We will now illustrate Proposition \ref{propositionBKT} by computing $\Delta_{\dot{w}}(V_{\dot{w}_0,5})$. Given $\dot{w}_0$, we have $u_5=s_1$ then the leftmost representative of $u_5$ in $\dot{w}$ is given by the letter of index $q_1=5$. We then have $(r_i)=(1,2,3,4,6)$. We can compute the roots $\beta_i^{\dot{w}}$:
			\[
				\beta_1=\alpha_1,\quad \beta_2=s_1(\alpha_2)=\alpha_1+\alpha_2,\quad \beta_3=s_1(s_2(\alpha_3))=\alpha_1+\alpha_2+\alpha_3
			\]
			\[
				\beta_4=s_1s_2s_3(\alpha_2)=\alpha_3,\quad \beta_5=s_1s_2s_3s_2(\alpha_1)=\alpha_2+\alpha_3,\quad \beta_6=s_1s_2s_3s_2s_1(\alpha_2)=\alpha_2
			\]
			
			We can now compute the $(\xi_i^{\dot{w}_0})_i$ sequence. $\xi_0=\varpi_{i_5}=\varpi_2$. We have:
			\[
				\begin{array}{ccccc}
				\xi_1&:=&s_{\alpha_1}(\varpi_2)&=&\varpi_2\\
				\xi_2&:=&s_{\alpha_1+\alpha_2}(\xi_1)&=&\varpi_2-\alpha_1-\alpha_2\\
				\xi_3&:=&s_{\alpha_1+\alpha_2+\alpha_3}(\xi_2)&=&\varpi_2-\alpha_1-\alpha_2\\
				\xi_4&:=&s_{\alpha_3}(\xi_3)&=&\varpi_2-\alpha_1-\alpha_2-\alpha_3\\
				\xi_5&:=&\xi_4&=&\varpi_2-\alpha_1-\alpha_2-\alpha_3\\
				\xi_6&:=&s_{\alpha_2}(\xi_5)&=&\varpi_2-\alpha_1-2\alpha_2-\alpha_3
				\end{array}\quad
				\begin{array}{ccccc}
					n_1\beta_1&=&\xi_0-\xi_1&=&0\alpha_1\\
					n_2\beta_2&=&\xi_1-\xi_2&=&1(\alpha_1+\alpha_2)\\
					n_3\beta_3&=&\xi_2-\xi_3&=&0(\alpha_1+\alpha_2+\alpha_3)\\
					n_4\beta_4&=&\xi_3-\xi_4&=&1\alpha_3\\
					n_5\beta_5&=&\xi_4-\xi_5&=&0(\alpha_2+\alpha_3)\\
					n_6\beta_6&=&\xi_5-\xi_6&=&1\alpha_2
				\end{array}
			\]
			
			and finally we get that $\Delta_{\dot{w}}(V_{\dot{w}_0,5})=(0,1,0,1,0,1)$, as expected.
		\end{ex}
	\subsection{\texorpdfstring{$\Delta$}{Delta}-vectors of \texorpdfstring{$V_{\bar{w}}$}{Vw}}
		We just saw how to compute $\Delta_{\dot{v}}(V_{\bar{w}})$. Even if it is sufficient to use it algorithmically and to do some computation, we will now show a combinatorial description of a part of these coordinates.
		\subsubsection{Specific notations}\label{sectionSpecificNotations}
			We will first introduce some specific notations in order to give a more concise formulation of the main result.
			
			\begin{definition}\label{definitionCombinatorialNumbers}
				Given a reduced representative $\bar{w}=[i_{\ell(w)},\dots,i_1]$ of $w\in W$ and $\bar{v}=[i_{p_{\ell(v)}},\dots,i_{p_1}]$ the rightmost representative of $v\leq w$ in $\bar{w}$, $1\leq k\leq \ell(w)$ and $1\leq m\leq\ell(v)$ we define
				\[
					f_\Min(k):=\min(\{1\leq j\leq \ell(v)\mid i_{p_j}=i_k\}\cup\{0\}),
				\]
				\[
					f(k):=\max(\{1\leq j\leq \ell(v)\mid p_j\leq k\text{ and }i_{p_j}=i_k\}\cup\{0\}),
				\]
				\[
					m^{\oplus}:=\min(\{1\leq j\leq \ell(v)\mid p_j>p_m\text{ et } i_{p_j}=i_{p_m}\}\cup\{\ell(v)+1\}),
				\]
				\[
					\alpha(k,m)=\#\{1\leq j\leq m\mid i_{p_j}=i_k\},\gamma_m:=\alpha(p_m,m)\quad\text{and}
				\]
				\[
					\beta_m=\#\{1\leq j\leq p_m\mid \forall 1\leq k\leq m, p_k\neq j\text{ and }i_j=i_{p_m}\}
				\]
				
				We fix the convention $\alpha(k,0)=0$.
			\end{definition}		
			Here we deal with two systems of indices. $1\leq k\leq\ell(w)$ is the usual indexation (from right to left) of the letters of $\bar{w}$. $1\leq m\leq\ell(v)$ is the indexation of the letters of $\bar{v}$. As $\bar{v}$ is the rightmost subword of $\bar{w}$, any $1\leq m\leq \ell(v)$ correspond to an index $1\leq k\leq \ell(w)$, namely $m\leadsto p_m$. In the following we will speak about $\bar{w}$-indices for $1\leq k\leq \ell(w)$ and $\bar{v}$-indices for $1\leq m\leq\ell(v)$.
			
			$f_\Min(k)$ is the lowest $\bar{v}$-index of color $i_k$. $f(k)$ is the highest $\bar{v}$-index of the letters of color $i_k$ in the subword of $\bar{w}$ consisting of all the letter on the right of the $k$-th ($k$-th included).
			
			$m^\oplus$ is the $\bar{v}$-successor of the letter of $\bar{v}$-index $m$.
			
			$\alpha(k,m)$ is the number of $\bar{v}$-indices of color $i_k$ lower or equal to $m$.
			
			$\beta_m$ is the number of $\bar{w}$-indices of color $i_k$ lower than $p_m$ which have no $\bar{v}$-index. It counts the number of letter of $\bar{w}$ on the right of the $p_m$-th, of color $i_{p_m}$, not taking part in the writing of $\bar{v}$.
			
		\begin{ex}
			Let $W=D_5$ of Dynkin diagram $\begin{tikzcd}[row sep=-2pt,column sep=-2pt]&&&4\\1\ar[r,dash]&2\ar[r,dash]&3\ar[ur,dash]\ar[dr,dash]&\\&&&5\end{tikzcd}$. We have 
			\[
				\bar{w}=[\textcolor{blue}{2}, \underline{\textcolor{Green}{3}}, \underline{\textcolor{orange}{4}}, \underline{\textcolor{red}{1}}, \underline{\textcolor{blue}{2}}, \underline{\textcolor{Green}{3}}, \underline{\textcolor{violet}{5}}, \textcolor{red}{1}, \textcolor{blue}{2}, \underline{\textcolor{Green}{3}}, \underline{\textcolor{orange}{4}}, \textcolor{red}{1}, \textcolor{blue}{2}, \underline{\textcolor{Green}{3}}, \textcolor{red}{1}, \underline{\textcolor{blue}{2}}, \textcolor{red}{1}]
			\]
			where underlined letters give $\bar{v}$.
			
			The notations previously introduced give us Table \ref{tableExNotations}.
			
			\begin{table}[ht]
			\begin{center}
				\begin{tabular}{|c|c||c|c|c|c|c|c||c|c||c|c|c|c|c|c|}
					\hline
					$k$&$m$&$i_k$&$f_\Min(k)$&$f(k)$&$m^\oplus$&$\beta_m$&$\gamma_m$&$k$&$m$&$i_k$&$f_\Min(k)$&$f(k)$&$m^\oplus$&$\beta_m$&$\gamma_m$\\
					\hline
					1&\cellcolor{gray!50}&\textcolor{red}{1}&8&0&\cellcolor{gray!50}&\cellcolor{gray!50}&\cellcolor{gray!50}&10&\cellcolor{gray!50}&\textcolor{red}{1}&8&0&\cellcolor{gray!50}&\cellcolor{gray!50}&\cellcolor{gray!50}\\
					\hline
					2&1&\textcolor{blue}{2}&1&1&7&0&1&11&5&\textcolor{violet}{5}&5&5&18&0&1\\
					\hline
					3&\cellcolor{gray!50}&\textcolor{red}{1}&8&0&\cellcolor{gray!50}&\cellcolor{gray!50}&\cellcolor{gray!50}&12&6&\textcolor{Green}{3}&2&6&10&0&3\\
					\hline
					4&2&\textcolor{Green}{3}&2&2&4&0&1&13&7&\textcolor{blue}{2}&1&7&18&2&2\\
					\hline
					5&\cellcolor{gray!50}&\textcolor{blue}{2}&1&1&\cellcolor{gray!50}&\cellcolor{gray!50}&\cellcolor{gray!50}&14&8&\textcolor{red}{1}&8&8&18&4&1\\
					\hline
					6&\cellcolor{gray!50}&\textcolor{red}{1}&8&0&\cellcolor{gray!50}&\cellcolor{gray!50}&\cellcolor{gray!50}&15&9&\textcolor{orange}{4}&3&9&18&0&2\\
					\hline
					7&3&\textcolor{orange}{4}&3&3&9&0&1&16&10&\textcolor{Green}{3}&2&10&18&0&4\\
					\hline
					8&4&\textcolor{Green}{3}&2&4&6&0&2&17&\cellcolor{gray!50}&\textcolor{blue}{2}&1&7&\cellcolor{gray!50}&\cellcolor{gray!50}&\cellcolor{gray!50}\\
					\hline
					9&\cellcolor{gray!50}&\textcolor{blue}{2}&1&1&\cellcolor{gray!50}&\cellcolor{gray!50}&\cellcolor{gray!50}&\cellcolor{gray!50}&\cellcolor{gray!50}&\cellcolor{gray!50}&\cellcolor{gray!50}&\cellcolor{gray!50}&\cellcolor{gray!50}&\cellcolor{gray!50}&\cellcolor{gray!50}\\
					\hline
				\end{tabular}
			\end{center}
				\caption{Example of the introduced notations}
				\label{tableExNotations}
			\end{table}
		\end{ex}
	\subsubsection{Combinatorial description}
		In the following we will consider $\bar{w}=[i_{\ell(w)},\dots,i_1]$ a reduced representative of $w\in W$ and $\bar{v}=[i_{p_{\ell(v)}},\dots,i_{p_1}]$ the rightmost subword of $\bar{w}$ representing $v\leq w$ for the Bruhat order in $W$. We denote $\dot{w}$ the completion of $\bar{w}$ in a representative of $w_0$ and $\dot{v}$ the one of $\bar{v}$.
		
		By construction we have the following result:
		\begin{proposition}
			The module $V_{\bar{w}}$ is a submodule of $V_{\dot{w}}$ obtained by taking the $\ell(w)$ first summands in the indexation given by Definition \ref{definitionVk}.
			
			The quiver $\Gamma_{\bar{w}}$ is a subquiver of the quiver $\Gamma_{\dot{w}}$ obtained by taking the $\ell(w)$ first vertices and the arrows between them.
		\end{proposition}
		In the following we will identify $V_{k,\bar{w}}$ and $V_{k,\dot{w}}$ for $1\leq k\leq \ell(w)$ and $V_{m,\bar{v}}$ and $V_{m,\dot{v}}$ for $1\leq m\leq\ell(v)$.
		
		Thanks to Proposition \ref{propositionIntervalModules} we already know a formulation for $\Delta_{\bar{w}}(V_{k,\bar{w}})$. We will now establish a combinatorical description of the $\ell(v)$ first coordinates of $\Delta_{\dot{v}}(V_{k,\bar{w}})$.
		\begin{nota}
			With the above notations, we will denote $\widetilde{\Delta}_{\dot{v}}(X):=(\Delta_{\dot{v},i}(X))_{i=1}^{\ell(v)}$ the $\ell(v)$ first $\Delta_{\dot{v}}$-coordinates of a module $X$.
		
		In the following we will use the notations introduced in Section \ref{sectionExplicitation}. Precisely we will focus on the $\Delta_{\dot{v}}$-vector of $V_{k,\bar{w}}$ seen as $V_{k,\dot{w}}$ (thus $1\leq k\leq\ell(w)$). As $V_{k,\bar{w}}$ is only defined by the $k$ first letters of $\dot{w}=[i_r,\dots,\underbrace{i_{\ell(w)},\dots,i_1}_{\bar{w}}]$ we define
		\[
			\bar{w}_k=[i_k,\dots,i_1],\quad w_k=s_{i_k}\cdots s_{i_1},\quad \bar{u}_k=[i_r,\dots,i_{k+1}]
		\]
		and as $u_k=s_{i_r}\cdots s_{i_{k+1}}=w_0w_k^{-1}$ we take $[j_{q_{r-k}},\dots,j_{q_1}]$ as the leftmost representative of $u_k$ in $\dot{v}=[j_r,\dots, j_1]$ (where $r=\ell(w_0)$). We denote $1\leq r_1<\cdots<r_k\leq r$ the elements of $\{1,\dots,r\}\setminus \{q_1,\dots,q_{r-k}\}$.
		\end{nota}
		\begin{proposition}[Relation between indices of representatives of $u_k$]\label{propositionRelationIndices}
		With the above notations
			\begin{enumerate}
				\item $q_1> \ell(v)$ iff $p_{\ell(v)}<k+1$
				\item there exists an integer $0\leq t\leq \ell(v)-1$ such that $q_1=\ell(v)-t$ iff $p_{\ell(v)-t}\geq k+1$ and $p_{\ell(v)-t-1}\leq k$.
			\end{enumerate}
		\end{proposition}
		\begin{proof}
			We begin by proving the first part of the statement. Suppose that we have $p_{\ell(v)}\leq k$. Then $\bar{v}=[i_{p_{\ell(v)}},\dots,i_{p_1}]$ is a subword of $[i_k,\dots,i_1]$, representative of $w_k$ and we have $v\leq u_k^{-1}w_0=w_k$.
			
			Conversely if $v\leq w_k$, there is at least one subword of $[i_k,\dots,i_1]$ representing $v$. The rightmost subword representing $v$ is the smallest for the lexicographic order among any representative of $v$ in $\bar{w}_k$ and thus a subword of $w_k$ (see \cite[Lemma 2.1.38]{menard2021AlgebresAmasseesAssociees}). In particular we have $p_{\ell(w)}\leq k$. We then have shown that $p_{\ell(v)}\leq k\Leftrightarrow v\leq u_a^{-1}w_0$.
			
			We now establish the equivalence $v\leq u_a^{-1}w_0\Leftrightarrow q_1>\ell(v)$. Suppose that $q_1>\ell(v)$ then $[j_{q_{r-k}},\dots,j_{q_1}]$ is a subword of $[j_r,\dots,j_{\ell(v)+1}]$. We thus have $u_k\leq w_0v^{-1}$ which is equivalent to $u_k^{-1}\leq vw_0$ (as $w_0^{-1}=w_0$ and by inversion of the words) and finaly $u_k^{-1}w_0\geq v$ (by multiplication by $w_0$).
			
			Conversely, suppose that $u_k\leq w_0v^{-1}$ then $[j_r,\dots,j_{\ell(v)+1}]$ is a representative of $w_0v^{-1}$ and there exists at least one representative of $u_k$ in this representative. As we are looking at the left part of $\dot{v}$, the leftmost representative of $u_k$ in $\dot{v}$ is among them and eventually $[j_{q_{r-k}},\dots,j_{q_1}]$ is a subword of $[j_{r},\dots,j_{\ell(v)+1}]$. Then it is now obvious that $q_1\geq \ell(v)+1$. 
			
			We now prove the second statement by introducing again a third statement and proceed via double equivalence.
			
			Let $t$ be such that $q_1=\ell(v)-t$, we want to show that it is equivalent to 
			\begin{equation}	(u\leq w_0s_{j_1}\cdots s_{j_{\ell(v)-t-1}})\quad\text{and}\quad(u\not\leq w_0s_{j_1}\cdots s_{j_{\ell(v)-t}})\label{eqInegalité}
			\end{equation}
			
			As $q_1=\ell(v)-t$, $[j_{q_{r-k}},\dots,j_{q_1}]$ is a subword of $[j_r,\dots,j_{\ell(v)-t}]$ and we have $u_k\leq w_0s_{j_1}\cdots s_{j_{\ell(v)-t-1}}$.
			
			On the other hand, as $[j_{q_{r-k}},\dots,j_{q_1}]$ is the leftmost representative of $u_k$ in $\dot{v}$, there is no representative of $u_k$ $[j_{n_{r-k}},\dots j_{n_1}]$ such that $n_1>q_1$. Then $u_k$ has no representative in $[j_r,\dots,j_{q_1+1}]$ and we have $u_k\not\leq w_0s_{j_1}\dots s_{j_{\ell(v)-t}}.$
			
			Conversely, by hypothesis, we know that there is a representative of $u_k$ among the subwords of $[j_r,\dots,j_{\ell(v)-t}]$ and none among the subwords of $[j_r,\dots,j_{\ell(v)-t+1}]$. In particular the leftmost representative has as a first letter $j_{\ell(v)-t-1}$ and $\ell(v)-t=q_1$.
			
			We now show the other equivalence. If $p_{\ell(v)-t}\geq k+1$ and $p_{\ell(v)-t-1}\leq k$, $s_{i_{p_{\ell(v)-t-1}}}\cdots s_{i_{p_1}}$ has at least one representative among the subwords of $[i_k,\dots,i_1]$. One of them is $[i_{p_{\ell(v)-t-1}},\dots,i_{p_1}]$ itself and we have $s_{i_{p_{\ell(v)-t-1}}}\cdots s_{i_{p_1}}\leq u_k^{-1}w_0$.
			Conversely, by right minimality, $s_{i_{p_{\ell(v)-t}}}\cdots s_{i_{p_1}}$ has no representative among the subwords of $[i_k,\dots,i_1]$ and $s_{i_{p_{\ell(v)-t}}}\cdots s_{i_{p_1}}\not\leq u_k^{-1}w_0$ and we get the wanted conjunction of inequalities by inverting and multiplication by $w_0$.
			
			On the other hand, if we have the two above inequalities there is no representative of $s_{i_{p_{\ell(v)-t}}}\cdots s_{i_{p_1}}$ in $\bar{u}_k$ and there is at least one of $s_{i_{p_{\ell(v)-t-1}}}\cdots s_{i_{p_1}}$. By minimality of the rightmost representative of $u_k$ we have $p_{\ell(v)-t}>k$ and $p_{\ell(v)-t-1}\leq k$.
			
			The inequalities used in this equivalence are the same as the one of Equation \ref{eqInegalité} as, by definition, $(j_i)_{1\leq i\leq \ell(v)}=(i_{p_m})_{1\leq m\leq \ell(v)}$. 
		\end{proof}
		
		\begin{proposition}\label{propositionSequenceBKT}
			For $m<q_1$, $\dot{x}=[l_m,\dots,l_1]$ a reduced representative of $\dot{w}$ (possibly equal to $\dot{w}$) we have $\xi_m^{\dot{x}}=s_{l_1}\cdots s_{l_m}(\varpi_{i_{k}})$ and
			\[
				\Delta_{\dot{x},m}(V_k)=\left\{\begin{array}{cc}1&\text{if } l_m=i_k\\0&\text{otherwise}\end{array}\right.
			\]
			where $\xi$ is the sequence defined in Definition \ref{defXi} defining the $\Delta_{\dot{x}}$-coordinates of $V_{\bar{w},k}$.
		\end{proposition}
		\begin{proof}
			We have $\beta_m^{\dot{x}}=s_{l_1}\cdots s_{l_{m-1}}(\alpha_{l_m})$ and so $s_{\beta_m^{\dot{x}}}=(s_{l_1}\cdots s_{l_{m}-1})s_{l_m}(s_{l_1}\cdots s_{l_{m}-1})$ then 
			\[
				s_{\beta_m^{\dot{x}}}\cdots s_{\beta_1^{\dot{v}}}=s_{l_1}\cdots s_{l_m}.
			\]
			
			As we suppose that $m<q_1$, we have $\xi_m^{\dot{x}}=s_{\beta_m^{\dot{x}}}\cdots s_{\beta_1^{\dot{x}}}(\varpi_{i_k})$ and then 
			\[
				\xi_{m-1}^{\dot{x}}-\xi_{m}^{\dot{x}}=s_{l_1}\cdots s_{l_{m-1}}(\varpi_{i_k}-s_{l_m}(\varpi_{i_k})).
			\]
			
			We have $s_{l_m}(\varpi_{i_k})= \varpi_{i_k}$ if $i_k\neq j_m$ and $\varpi_{i_k}-\alpha_{l_m}$ if $i_k=l_m$. In this last case $\xi_{m-1}^{\dot{x}}-\xi_m^{\dot{x}}=\beta_m^{\dot{x}}$ and then:
			\[
				n_m^{\dot{x}}=\left\{\begin{array}{cc}1&\text{if } l_m=i_k\\0&\text{otherwise}\end{array}\right.
			\]
		\end{proof}
		\begin{lemma}\label{lemmaNoHole}
			If there is $0\leq t\leq \ell(v)-1$ such that $q_1=\ell(v)-t$, we have 
			\[
				[q_{t+1},\dots,q_1]=[\ell(v),\dots,\ell(v)-t].
			\]
		\end{lemma}
		\begin{proof}
			We prove this lemma for a given $\dot{w}$ and $\dot{v}$, we will proceed by decreasing induction on the value of $k$. The sequences $(q_i)_{1\leq i\leq r-k}$ depending on the value of $k$ we add an index indicating to which value of $k$ they refer to.
			
			If $k>p_{\ell(v)}-1$, we have $q_{1,k}\geq\ell(v)$ by Proposition \ref{propositionRelationIndices} and thus we are not in the hypothesis of the theorem.
			
			If $k=p_{\ell(v)}-1$ then $p_{\ell(v)}= k+1$ and as $p_{\ell(v)-1}\leq p_{\ell(v)}-1$, $p_{\ell(v)-1}\leq k$. By the same proposition $q_{1,k}=\ell(v)$ and so $[q_{1,k}]=[\ell(v)]$ and the lemma is verified. 
			
			Suppose that $1<k\leq p_{\ell(v)}$, that there exists $0\leq t_k\leq\ell(v)-1$ such that $q_{1,k}=\ell(v)-t_k$ and assume that $[q_{t_k+1,k},\dots,q_{1,k}]=[\ell(v),\dots,\ell(v)-t_k]$. We want to show that there exists $0\leq t_{k-1}\leq \ell(v)-1$ such that $q_{1,k-1}=\ell(v)-t_{k-1}$ and such that $[q_{t_{k-1}+1,k-1},\dots,q_{1,k-1}]=[\ell(v),\dots,\ell(v)-t_{k-1}]$.
			
			First we show that $t_{k-1}$ exists. As $q_{1,k}=\ell(v)-t_k$, we have $p_{\ell(v)-t_k}\geq k+1\text{ and } p_{\ell(v)-t_{k}-1}\leq k$ which can be rewritten as $p_{\ell(v)-t_{k}}\geq (k-1)+2\text{ and }p_{\ell(v)-t_k-1}\leq (k-1)+1$.
			
			We have two possibilities: either $p_{\ell(v)-t_k-1}=(k-1)+1=k$ or $p_{\ell(v)-t_k-1}\leq k-1$.
			
			In the first case we then have $
				p_{\ell(v)-t_k-1}\geq (k-1)+1\text{ and }p_{\ell(v)-t_k-2}<(k-1)+1$ so $p_{\ell(v)-t_k-2}\leq k-1$
			then, according to Proposition \ref{propositionRelationIndices} we have $q_{1,k-1}=\ell(v)-t_k-1$ and so $t_{k-1}=t_k+1$
			
			In the second case we have $p_{\ell(v)-t_k}\geq (k-1)+2\geq (k-1)+1\text{ and }p_{\ell(v)-q_k-1}\leq k-1$ and so $q_{1,k-1}=\ell(v)-t_k$, eventually $t_k=t_{k-1}$.
			
			In both cases, such a $t_{k-1}$ exists. We now show that any of the two cases imply the equality.
			
			In the first case $t_{k-1}=t_k+1$ and $q_{1,k-1}=\ell-t_{k-1}$. We first show that $[q_{r-(k-1),k-1},\dots,q_{2,k-1}]$ is a representative of $u_k$. By definition $u_{k-1}=s_{j_{q_{r-(k-1),k-1}}}\cdots s_{j_{q_{1,k-1}}}=s_{i_r}\cdots s_{i_{(k-1)+1}}$.
			
			Then $u_{k-1}s_{i_{(k-1)+1}}=s_{i_r}\cdots s_{i_{k+1}}=u_{k}$. However, $i_{p_m}=j_m$ for any $1\leq m\leq\ell(v)$ by definition of $\dot{v}$. As $q_{1,k-1}=\ell(v)-t_{k-1}$ we have $j_{q_{1,k-1}}=j_{\ell(v)-t_{k-1}}=i_{p_{\ell(v)-t_{k-1}}}=i_{p_{\ell(v)-t_k-1}}$ and, as we are in the first case, $p_{\ell(v)-t_k-1}=(k-1)+1=k$, $j_{q_{1,k-1}}=i_{k}$. Then $u_k=u_{k-1}s_{i_k}=u_ks_{j_{q_{1,k-1}}}=s_{j_{q_{r-(k-1),k-1}}}\cdots s_{j_{q_{2,k-1}}}$ and $[q_{r-(k-1),k-1},\dots,q_{2,k-1}]$ is a representative of $u_{k}$ in $\dot{v}$. As the leftmost representative of $u_k$ in $\dot{v}$ is $[q_{r-k,k},\dots,q_{1,k}]$ for any $2\leq \alpha\leq r-(k-1)$ we have $q_{\alpha-1,k}\geq q_{\alpha,k-1}$. 
			
			We will show that, in fact, $q_{\alpha-1,k}=q_{\alpha,k-1}$ for any $2\leq \alpha\leq t_k+2$.
			
			As $q_{2,k-1}>q_{1,k-1}$ we have $\ell(v)-t_{k}=q_{1,k}\geq q_{2,k-1}>q_{1,k-1}=\ell(v)-t_{k-1}=\ell(v)-t_k-1$ so $q_{1,k}=q_{2,k-1}$. As for any $2\leq\alpha\leq t_{k}+2$ we have $q_{\alpha-1,k}=\ell(v)-t_k+(\alpha-2)$ by induction hypothesis, we can repeat this proof for any value of $\alpha\leq t_{k}+2$ and get:
			\[
				\forall 2\leq \alpha\leq t_{k}+2,\ q_{\alpha,k-1}=\ell(v)-t_{k-1}+(\alpha-1)=\ell(v)-t_k+(\alpha-2)=q_{\alpha-1,k}
			\]
			and so $[q_{t_{k-1}+1,k-1},\dots,q_{2,k-1}]=[q_{t_{k}+1,k},\dots, q_{1,k}]=[\ell(v),\dots,\ell(v)-t_{k-1}+1].$
			
			As we have already shown that $q_{1,k-1}=\ell(v)-t_{k-1}$ we finally obtain
			\[
				[q_{t_{k-1}+1,k-1},\dots, q_{1,k-1}]=[\ell(v),\dots,\ell(v)-t_{k-1}]
			\]
			
			In the second case, $\ell(v)-t_{k}=q_{1,k}=q_{1,k-1}=\ell(v)-t_{k-1}$. As previously, by definition, $u_{k-1}=u_ks_{i_{k}}$ and so $u_{k-1}=s_{j_{q_{r-k,k}}}\cdots s_{j_{q_{1,k}}}s_{i_{k}}$.
			
			However $\ell(u_{k-1})=r-{k-1}$ and $\ell(u_ks_{j_{q_{1,k-1}}})\leq r-(k-1)-1$ as it admits $s_{j_{q_{r-(k-1),k-1}}}\cdots s_{j_{q_{2,k-1}}}$ as representative. We then have $\ell(u_{k-1})>\ell(u_{k-1}s_{j_{q_{1,k-1}}})$.
			
			According to the exchange relation there is $r-(k-1)\geq\alpha\geq q_{1,k-1}$ such that 
			\[
				u_k=u_{k-1}s_{i_{k}}=s_{j_{q_{r-(k-1),k-1}}}\cdots\hat{s}_{j_{q_{\alpha,k-1}}}\cdots s_{j_{q_{1,k-1}}}.
			\]
			As $[j_{q_{r-k,k}},\dots,j_{1,k}]$ is the leftmost representative of $u_k$ we have:
			\[
				\forall 1\leq \beta\leq \alpha,\  q_{\beta,k-1}\leq q_{\beta,k}\text{ and }\forall \alpha +1\leq \beta\leq r-(k-1),\ q_{\beta,k-1}\leq q_{\beta-1,k}
			\]
			
			We will show by induction that for all $2\leq \beta\leq \min(\alpha,t_{k-1}+1)$, we have $q_{\beta,k-1}=q_{\beta,k}.$
			
			By hypothesis, we have $q_{1,k-1}=q_{1,k}=\ell(v)-t_{k-1}$ and so for $\beta=2$ we have 
			\[
				\ell(v)-t_{k-1}=q_{1,k-1}<q_{2,k-1}\leq q_{2,k}=\ell(v)-t_k+1=\ell(v)-t_{k-1}+1
			\]
			and thus $q_{2,k-1}=\ell-t_{k-1}+1=q_{2,k}$. We then have the initialization.
			
			By induction, we suppose that we have $q_{\beta-1,k-1}=\ell(v)-t_{k-1}+\beta-2=q_{\beta-1,k}$. Then we have: $\ell(v)-t_{k-1}+\beta-2=q_{\beta-1,k-1}<q_{\beta,k-1}$ 
			and, by minimality we have \[q_{\beta,k-1}\leq q_{\beta,k}=\ell(v)-t_{k-1}+\beta-1\] and so we have the following sequence of inequalities.
			\[
				\ell(v)-t_{k-1}+\beta -2=q_{\beta-1,k-1}<q_{\beta,k-1}\leq q_{\beta,k}=\ell(v)-t_{k-1}+\beta-1
			\]
			and so $q_{\beta,k-1}=q_{\beta,k}$.
			
			We know want to determine relative positions of $\alpha$ and $t_{k-1}+1$. Indeed, if $\alpha> t_{k-1}+1$ we can deduce the lemma from the above induction.
			
			By contradiction, let us suppose $\min(\alpha,t_{k-1}+1)=\alpha\leq t_{k-1}+1$ then we have:
			\[
				\ell(v)-t_{k-1}+\alpha-2=q_{\alpha-1,k-1}<q_{\alpha,k-1}<q_{\alpha+1,k-1}\leq q_{\alpha,k}=\ell(v)-t_{k-1}+\alpha-1
			\]
			which is absurd by a cardinality argument. Then $\alpha>t_{k-1}+1$ and $q_{\beta,k}=q_{\beta,k-1}$ for all $2\leq \beta\leq t_{k-1}+1$ and as we previously show that $q_{1,k-1}=q_{1,k}$, we can eventually write:
			\[
				[q_{t_{k-1}+1,k-1},\dots,q_{1,k-1}]=[\ell(v),\dots,\ell(v)-t_{k-1}]
			\]
			and we have the heredity and thus the lemma is proved in both cases.
		\end{proof}
		\begin{theorem}\label{thmCoordinateTranslation}
			We have the following correspondance between $\Delta$-vector coordinates:
			\[
				\Delta_{\dot{v},m}(V_k)=\Delta_{\dot{w},i_{p_m}}(V_k),\quad \forall 1\leq m\leq\ell(v)\text{ and }1\leq k\leq\ell(w).
			\]
		\end{theorem}
		\begin{proof}
			First, if $k=\ell(w_0)$ (and thus $w=w_0$), $u_k=e$ and we can always apply Proposition \ref{propositionSequenceBKT} and we have the desired equality.
			
			In any other case we will prove the identities, for any coordinate of index $1\leq m\leq\ell(v)$.
			
			If $m<q_1$, thanks to Proposition \ref{propositionSequenceBKT} we take $\dot{x}=\dot{v}$ and then $\Delta_{\dot{v},m}(V_k)=\left\{\begin{array}{cc}1&\text{if }j_m=i_k\\0&\text{otherwise}\end{array}\right..$
			
			Now we will show we can use the same result but with $\dot{w}=\dot{x}$ and $p_m$ instead of $m$ to get
			\[
				\Delta_{\dot{w},m}(V_k)=\left\{\begin{array}{cc}1&\text{if }i_{p_m}=i_k\\0&\text{otherwise}\end{array}\right.
			\]
			
			As $m<q_1$ we have two possibilities. Either we have $m<m_1$ for any $1\leq m\leq \ell(v)$ (first case of Proposition \ref{propositionRelationIndices}) or not. If we do, $p_{\ell(v)}\leq k$. As $p_m\leq p_{\ell(v)}$, $q_1$ is the first index of $u_k$ in $\dot{w}$, $k+1$ by definition. We then have $p_m\leq q_1$ and we can apply Proposition \ref{propositionSequenceBKT}.
			
			In the other case, there exists $0\leq t\leq\ell(v)-1$ such that $q_1=\ell(v)-t$. However, in any case we have $m<q_1$ so $p_m\leq p_{\ell(v)-t-1}\leq k$ and we can once again apply Proposition \ref{propositionSequenceBKT}. We then have 
			\[
				\Delta_{\dot{w},p_m}(V_k)=\left\{\begin{array}{cc}1&\text{if }i_{p_m}=i_k\\0&\text{otherwise}\end{array}\right.
			\]
			hence the result.
			
			Now we consider the case where $m\geq q_1$. In this case, thanks to Proposition \ref{propositionRelationIndices} we now that $m\geq \ell(v)-t$ and thus $p_m\geq p_{\ell(v)-t}\geq k+1$. By definition of $u_k$ in $\dot{w}$, we know that $\xi_{p_m}^{\dot{w}}=\xi^{\dot{w}}_{p_{m}-1}$ and so $n_{p_m}^{\dot{w}}=0$ and finally $\Delta_{\dot{w},p_m}(V_k)=0$. We then have to prove that, in this case, $\Delta_{\dot{v},m}(V_k)=0$. 
			
			We use Lemma \ref{lemmaNoHole}. We have that $[\ell(v),\dots,\ell(v)-t]=[q_{t+1},\dots,q_1]$ and thus, for all $\ell(v)\geq m\geq\ell(v)-t$, we have $\Delta_{\dot{v},m}(V_k)=0$.
			
			Finally, for all $1\leq m\leq\ell(v)$ we have: $\Delta_{\dot{v},m}(V_k)=\Delta_{\dot{w},p_m}(V_k)$.
		\end{proof}
		\begin{ex}
			In order to illustrate the previous notions let us consider the following case. Here $W=D_6$ 
			
			We consider $\bar{w}=[5, 3, 4, 2, 3, 6, 4, 2, 1, 5, 2, 3, 2, 4, 3, 6, 4, 5, 3, 4, 1, 3]$ and complete it in:
			\[
				\dot{w}=[2, 3, 4, 6, 1, 2, 3, 4, \underline{5}, 3, 4, 2, 3, 6, \underline{4}, 2,\underline{ 1, 5, 2, 3, 2, 4, 3, 6, 4, 5}, 3, \underline{4}, 1, \underline{3}].
			\]
			
			We chose $v=s_4s_5s_1s_2s_3s_4s_6s_2s_3s_4s_5s_4s_2s_3$ (so $\ell(v)=14$) and we have 
			\[
				\dot{v}= [4, 6, 2, 3, 4, 5, 1, 2, 3, 4, 6, 5, 4, 2, 3, 2, \underbrace{5, 4, 1, 5, 2, 3, 2, 4, 3, 6, 4, 5, 4, 3}_{\bar{v}}].
			\]
			
			The indices of $\bar{w}$ letters forming $\bar{v}$ are $(p_m)_{1\leq m\leq 14}=(1,3,5,6,7,8,9,10,11,12,13,14,16,22).$
			
			Let us take $k=10$ we then have
			\[
				\dot{w}=[\underbrace{2, 3, 4, 6, 1, 2, 3, 4, 5, 3, 4, 2, 3, 6, 4, 2, 1, 5, 2, 3}_{\bar{u}_{10}}, \underbrace{2, 4, 3, 6, 4, 5, 3, 4, 1, 3}_{\bar{w}_{10}}].
			\]
			and now we look for the leftmost representative of $u_{10}$ in $\dot{v}$.
			
			We find $[6, 2, 3, 4, 5, 1, 2, 3, 4, 6, 4, 2, 3, 2, 5, 4, 1, 5, 2, 3]$ which uses the letters in positions
				\[
					(q_i)_{1\leq i\leq \ell(w_0)-k}=(9,10,11,12,13,14,15,16,17,18,20,21,22,23,24,25,26,27,28,29).
				\]
				
				In particular we then have $q_1=9$. $p_9=11\geq k+1$ et $p_8=10\leq k$, as in the second case of Proposition \ref{propositionRelationIndices}.
				
				To sum up, we have
				\[
					\dot{w}=[\rouge{2},\rouge{3},\rouge{4},\rouge{6},\rouge{1},\rouge{2},\rouge{3},\rouge{4},\rouge{5},\rouge{3},\rouge{4},\rouge{2},\rouge{3},\rouge{6},\rouge{4},\rouge{2},\rouge{1},\rouge{5},\rouge{2},\rouge{3},2,4,3,6,4,5,3,4,1,3]
				\]
				\[
					\dot{w}=[2,3,4,6,1,2,3,4,\bleu{5},3,4,2,3,6,\bleu{4},2,\bleu{1},\bleu{5},\bleu{2},\bleu{3},\bleu{2},\bleu{4},\bleu{3},\bleu{6},\bleu{4},\bleu{5},3,\bleu{4},1,\bleu{3}]
				\]
				\[
					\dot{v}=[4,\rouge{6},\rouge{2},\rouge{3},\rouge{4},\rouge{5},\rouge{1},\rouge{2},\rouge{3},\rouge{4},\rouge{6},5,\rouge{4},\rouge{2},\rouge{3},\rouge{2},\rouge{5},\rouge{4},\rouge{1},\rouge{5},\rouge{2},\rouge{3},2,4,3,6,4,5,4,3]
				\]
				\[
					\dot{v}=[4,6,2,3,4,5,1,2,3,4,6,5,4,2,3,2,\bleu{5},\bleu{4},\bleu{1},\bleu{5},\bleu{2},\bleu{3},\bleu{2},\bleu{4},\bleu{3},\bleu{6},\bleu{4},\bleu{5},\bleu{4},\bleu{3}]
				\]
				where the overlined red letters represent positions of the letters of the leftmost representative of $u_{10}$ in $\dot{w}$ and $\dot{v}$ and the underlined blue ones, the one of $v$. We can see Lemma \ref{lemmaNoHole} in the fact that, for indices $9=q_1\leq m\leq \ell(v)=14$, all letters are (the right) part of the representative of $u_{10}$.
				
				We now place ourselves in a close case. We keep $\dot{w}$, and take $v=s_4s_5s_1s_2s_3s_4s_6s_2s_3s_4s_5s_4s_2s_3$. We have $\bar{v}=[4, 1, 2, 3, 2, 4, 3, 6, 4, 5, 4, 3]$, so, in particular, $p_{12}=15$ and
				\[
					\dot{v}=[4, 6, 2, 3, 4, 5, 1, 2, 3, 4, 6, 2, 3, 4, 5, 4, 3, 2, 4, 1, 2, 3, 2, 4, 3, 6, 4, 5, 4, 3].
				\] 
				Let $k=16$ be. We have $p_{12}=15<17=k+1$. Thus we have
				\[
					\dot{w}=[\rouge{2},\rouge{3},\rouge{4},\rouge{6},\rouge{1},\rouge{2},\rouge{3},\rouge{4},\rouge{5},\rouge{3},\rouge{4},\rouge{2},\rouge{3},\rouge{6},4,2,1,5,2,3,2,4,3,6,4,5,3,4,1,3]
				\]
				\[
					\dot{w}=[2,3,4,6,1,2,3,4,5,3,4,2,3,6,4,\bleu{2},\bleu{1},5,\bleu{2},\bleu{3},\bleu{2},\bleu{4},\bleu{3},\bleu{6},\bleu{4},\bleu{5},3,\bleu{4},1,\bleu{3}]
				\]
				\[
					\dot{v}=[4, \rouge{6}, \rouge{2}, \rouge{3}, \rouge{4}, \rouge{5}, \rouge{1}, \rouge{2}, \rouge{3}, \rouge{4}, \rouge{6}, \rouge{2}, 3, \rouge{4}, \rouge{5}, 4, \rouge{3}, 2, 4, 1, 2, 3, 2, 4, 3, 6, 4, 5, 4, 3]
				\]
				\[
					\dot{v}=[4, 6, 2, 3, 4, 5, 1, 2, 3, 4, 6, 2, 3, 4, 5, 4, 3, 2, \bleu{4}, \bleu{1}, \bleu{2}, \bleu{3}, \bleu{2}, \bleu{4}, \bleu{3}, \bleu{6}, \bleu{4}, \bleu{5}, \bleu{4}, \bleu{3}]
				\]
				
				The representative of $u_{16}$ is $\bar{u}_{16}=[2,3,4,6,1,2,3,4,5,3,4,2,3,6]$. The leftmost representative of $u_{16}$ in $\dot{v}$ is then: $[4, 6, 2, 3, 4, 5, 1, 2, 3, 4, 6, 2, 4, 5, 3]$. In particular $q_1=14>12=\ell(v)$, we are in the first case of Proposition \ref{propositionRelationIndices}.
		\end{ex}
		
		The previous results give us now a combinatorial description of the first coordinates of the $\Delta_{\dot{v}}$-vectors of indecomposable summands of $V_{\bar{w}}$.
		
		\begin{theorem}[Structure of initial $\Delta_{\dot{v}}$-vectors]\label{theoremStructureInitialDeltaVectors}
			The $\ell(v)$-first coordinates of $\Delta_{\dot{v}}$-vectors of direct indecomposable summands of $V_{\bar{w}}$ are given by:
			\[
				\widetilde{\Delta}_{\dot{v}}(V_k)=\sum_{j=1}^{f(k)}\delta_{i_{p_j},i_k}f_j\quad 1\leq k\leq\ell(w)
			\]
			where $(f_j)_{1\leq j\leq \ell(v)}$ is the canonical basis.
			
			We can also say that the indices of coordinates equal to 1 of $\Delta_{\dot{v}}(V_k)$ are the integers of the set
			\[
				\{1\leq j\leq\ell(v)\mid f_{\Min}(k)\leq j\leq f(k)\mid i_{p_j}=i_k\},
			\]
			the other $\ell(v)$ first being zero.
		\end{theorem}
		\begin{proof}
			By definition, $f(k)=\max\{1\leq j\leq\ell(v)\mid p_j\leq k \text{ et }i_{p_j}=i_k\}$. We will use Theorem \ref{thmCoordinateTranslation} and look at any index $1\leq m\leq \ell(v)$ of $\Delta_{\dot{v}}$. 
			
			First suppose that we have $m$ such that $p_m\leq k$. For these coordinates, thanks to Proposition \ref{propositionRelationIndices}, we have that $m<q_1$ and so we can directly use Proposition \ref{propositionBKT}.
			
			In the other case, by Lemma \ref{lemmaNoHole} and the proof of Theorem \ref{thmCoordinateTranslation}, we now that these coordinates are zero and here they are not part of the sum.
			
			For the second formulation, it is equivalent as $\{1\leq j<f_{\Min}(k)\mid i_{p_j}i_k\}$ is empty by minimality of $f_{\Min}(k)$.
		\end{proof}
	\subsection{Criterion for being in \texorpdfstring{$\C^v$}{Cv}}
		We now want to use $\Delta$-vector to be able to tell if a module is in $\C^v$ or in $\C_w$.
		
		\begin{proposition}\label{propositionCwbelonging}
			Let $\bar{w}$ be a  reduced representative of $w\in W$ and $\dot{w}$ a reduced representative of $w_0$ having $\bar{w}$ as a right factor. A module $M\in\mod(\Lambda)$ is in $\C_w$ iff its $\Delta_{\dot{w}}$-vector has zero coordinates for indices $k\geq \ell(w)$.
		\end{proposition}
		\begin{proof}
			Thanks to Lemma \ref{lemmaStratificationCw}, any module of $\C_w$ admits a stratification by modules $(M_{i,\dot{w}})_{i=1}^{\ell(w)}$ and conversely, any module admitting such a stratification lies in the category. 
		\end{proof}
		
		In order to prove an analoguous result for $\C^v$, we first need some preparatory lemmas.
		\begin{lemma}
			Let $M\in\mod(\Lambda)$ rigid, $\bar{v}$ a representative of $v$. The maximal submodule of $M$ in $\C_v$, denoted $t_v(M)$, has as $\Delta_{\dot{v}}$-vector the one coming from the $\ell(v)$ first coordinates of $\Delta_{\dot{v}}(M)$.
		\end{lemma}
		\begin{proof}
			Let $N\subset M$ of $\Delta_{\dot{v}}$-vector given by the $\ell(v)$ first coordinates of $M$. As $M$ is rigid, so is $N$ and it is completely determined by its $\Delta_{\dot{v}}$-vector. We will show that it is the maximal submodule of $M$ being in $\C_v$.
			
			$N\subset M\in\C_v$. By maximality, we have $N\subset t_v(M)$ and then $M/N\supseteq M/t_v(M)$. By definition of $\Delta_{\dot{v}}$-vectors, any $M$ submodule containing $N$ has as $\ell(v)$ first coordinates the one of $N$. So if there exists $t_v(M)\supsetneq N$ in $\C_v$, then $t_v(M)$ has at least one coordinate in its $\Delta_{\dot{v}}$-vector nonzero and of index more than $\ell(v)$. Thanks to Proposition \ref{propositionCwbelonging}, such a module does not belong to $\C_v$ which is absurd. So $t_v(M)=N$.
		\end{proof}
		\begin{proposition}[Criterion to be in $\C^v$]\label{propositionCvbelonging}
			Let $v\in W$ and $M$ be a rigid indecomposable $\Lambda$-module then $M\in\C^v$ iff $\widetilde{\Delta}_{\dot{v}}(M)=0$.
		\end{proposition}
		\begin{proof}
			Thanks to \cite[Proposition 3.12]{leclerc2016ClusterStructuresStrata}, if $M$ is a rigid indecomposable module of $\C_w$ then $M/t_v(M)$ is a rigid indecomposable module of $\C_{v,w}$. Then $M/t_v(M)\in\C^v$.
			
			Thus $M=M/t_v(M)$ or, equivalently, $t_v(M)=0$ iff $M\in\C^v$. According to the previous lemma, $t_v(M)=0$ iff $\widetilde{\Delta}_{\dot{v}}(M)=0$.
		\end{proof}

\section{Seed computation algorithm}
	Thanks to Section \ref{sectionInitialSeedCw}, we have an explicit way to build a seed for the cluster structure on $\C_w$ and thanks to Proposition \ref{propositionCvbelonging} we have a way to check if the underlying module of the seed is in $\C^v$. We must now combine both in order to get a $\C_{v,w}$-rigid maximal basic module and its quiver. Doing so we will then have an initial seed for the cluster structure on $\C_{v,w}$.
	\subsection{Concept}
		As it would be very difficult to build the quiver of a seed $G$ of $\C_{v,w}$ by looking at the endormorphism algebra of the module, we want to use the fact that we already know the quiver of $V_{\bar{w}}$ and we can compute, following the mutation, the quiver of any seed mutation-equivalent to $G$. Our idea is then to apply a well-chosen sequence of mutation to $G_{\bar{w}}=(V_{\bar{w}},\Gamma_{\bar{w}})$ in order to get a seed $G'=(V',\Gamma')$ such that a subseed is a seed for the cluster structure on $\C_{v,w}$. That is to say that we want $V'$ to be a $\C_{v,w}$-rigid maximal basic module and $\Gamma'$ the quiver of its endomorphism algebra.
		
		The rigidity of the seed will not be an issue as rigidity is preserved through mutation and cluster reduction as well as the basicity. Regarding the $\C_{v,w}$-maximal rigidity we already saw in Theorem \ref{thmNumberOfIndecomposables} that to check it, it is enough to check that the module has $\ell(w)-\ell(v)$ indecomposable summands. Thus the cluster reduction will have to remove $\ell(v)$ summands to the seed.
		
		Eventually the main issue in this method will be to ensure that the obtained module does lie in $\C_{v,w}$. We will use the criterion of Proposition \ref{propositionCvbelonging} and reach it thanks to a very specific choice of mutations. We will now introduce two different formulations of the sequence of mutations to perform, each one coming from a different idea. We will show in Section \ref{sectionMutationSequenceEquivalence} that these two formulations are in fact equivalent.
	\subsection{First formulation}\label{sectionSchroer}
		The idea of this formulation, due to Jan Schröer, is to mimic the sequence of mutation in \cite[Section 13]{geiss2011KacMoodyGroups}. In this sequence, the authors look at the modules $V_k$ of $V_{\bar{w}}$ which have the $\Delta_{\dot{w}}$-vectors as described in Proposition \ref{propositionIntervalModules}, making them interval modules. The sequence of mutation aims to "shift" these $\Delta_{\dot{w}}$- coordinates in the following sense: initial $V_k$ modules have nonzero (and equal to 1) coordinates for the indices $\leq k$ of letters in $\bar{w}$ of color $i_k$; final $T_k$ modules then have nonzero (equal to 1) $\Delta_{\dot{w}}$-coordinates for the last $k$ indices of color $i_k$.
		
		Here we want to do the same but with $\Delta_{\dot{v}}$-coordinates of modules originally equal to $V_{\bar{w}}$, shifting the coordinates far enough from the start, such that first $\ell(v)$ coordinates are all zero and thus the corresponding modules are in $\C^v$. As we are now dealing with $\Delta_{\dot{v}}$ coordinates of $V_{\bar{w}}$ the sequence of mutation will be slightly different and we will only find the same sequence if $\bar{v}$ is a right factor of $\bar{w}$.
		
		The advantage of this formulation is that we can directly read the sequence of mutation to perform from the positions of the letters of $\bar{v}$ in $\bar{w}$. This formulation will not be very easy to use in a proof but we will see in the next section a more helpful one. 
		
		We will express the mutation sequence using the combinatorial numbers defined in Definition \ref{definitionCombinatorialNumbers}
		
		\begin{definition}
			Given a reduced representative $\bar{w}$ of $w\in W$ and the rightmost representative $\bar{v}$ of $v\leq w$ in $\bar{w}$ we define for any letter $1\leq m\leq \ell(v)$ of $\bar{v}$ the following sequence of mutations:
			\[
				\tilde{\mu}_m:=\left\{\begin{array}{cc}\mu_{(k_\Max)^{\gamma_m-}}\circ\mu_{(k_\Max)^{(\gamma_m+1)-}}\circ\cdots\circ\mu_{(k_\Min)^{(\beta_m+1)+}}\circ\mu_{(k_\Min)^{\beta_m+}}&\text{if } (k_\Max)^{\gamma_m-}\geq (k_\Min)^{\beta +}\\\mathrm{id}&\text{otherwise}\end{array}\right.
			\]
			where $k=p_m$. Then we combine all the $\tilde{\mu}_m$ to form the sequence: $\widetilde{M}=\tilde{\mu}_{\ell(v)}\circ\cdots\circ\tilde{\mu}_1$
		\end{definition}
		
		A way to rephrase this algorithm is to say that given the index $1\leq m\leq \ell(v)$, the sequence $\widetilde{\mu}_m$ consists in mutating all the summands whose color is $i_{p_m}$ by ascending order of index (following the indexation given in Definition \ref{definitionVk}) apart from a number of initial summands (depending on the number of letters of color $i_{p_m}$ among the letters of $\bar{w}$ of index $<p_m$ not used in the writing of $\bar{v}$) and a number of final summands (depending on the number of times we have already mutated this line).
		\begin{ex}\label{exempleAlgo1}
				Here $W$ is of type $A_5$ with $\bar{w}=[\textcolor{red}{1},\textcolor{Green}{3},\textcolor{blue}{2},\textcolor{orange}{4},\textcolor{Green}{3},\underline{\textcolor{blue}{2}},\underline{\textcolor{orange}{4}},\underline{\textcolor{violet}{5}},\textcolor{orange}{4},\underline{\textcolor{Green}{3}},\textcolor{blue}{2},\underline{\textcolor{red}{1}},\underline{\textcolor{blue}{2}}]$ and $ v=s_2s_4s_5s_3s_1s_2.$
				
				The integers $p_m$, $\beta$ and $\gamma$ and the colors $i_{p_m}$ are written for any $1\leq m\leq 6=\ell(v)$ in Table \ref{tableBetaGamma}.
				
				\begin{table}[ht]
					\[
						\begin{array}{|c|c|c|c|c|c|c|}
						\hline
							m&1&2&3&4&5&6\\
						\hline
							p_m&1&2&4&6&7&8\\
						\hline
							i_{p_m}&\textcolor{blue}{2}&\textcolor{red}{1}&\textcolor{Green}{3}&\textcolor{violet}{5}&\textcolor{orange}{4}&\textcolor{blue}{2}\\
						\hline
							\beta_m&0&0&0&0&1&1\\
						\hline
							\gamma_m&1&1&1&1&1&2\\
						\hline
						\end{array}
					\]
					\caption{Values of $p_m, i_{p_m},\beta_m$ and $\gamma_m$}
					\label{tableBetaGamma}
				\end{table}
				
				We then have $\textcolor{blue}{\widetilde{\mu}_1=\mu_8\circ\mu_3\circ\mu_1}$, $\textcolor{red}{\widetilde{\mu}_2=\mu_2}$, $ \textcolor{Green}{\widetilde{\mu}_3=\mu_9\circ\mu_4}$, $\textcolor{violet}{\widetilde{\mu}_4=id}$, $\textcolor{orange}{\widetilde{\mu}_5=\mu_7}$, $\textcolor{blue}{\widetilde{\mu}_6=\mu_3}.$ Note that whenever $\beta_m=0$ we get the same sequence of mutations as in \cite[Section 13]{geiss2011KacMoodyGroups}.			
			\end{ex}
	\subsection{Recursive formulation}\label{sectionRecursive}
		We have another formulation of the algorithm, allowing us to see it as a kind of  greedy algorithm. The idea is to first cancel the first $\Delta_{\dot{v}}$ coordinates of (almost) all the summands, then the second coordinates etc. until the $\ell(v)$-th. Then, removing all the summands which still have nonzero coordinates among the $\ell(v)$ first, we will obtain a module in $\C^v$, thus in $\C_{v,w}$.
		
		As this algorithm is defined by induction we need first to introduce some notations before giving its formulation.
		\begin{nota}
			In the following we will define a sequence of sequences of mutations $(\hat{\mu}_m)_{m=1}^{\ell(v)}$. In order to keep track of the successive modules we will call $V_{\bar{w}}=R_0$ (respectively $V_{k,\bar{w}}=R_{k,0}$ for $1\leq k\leq\ell(w_0)$) and for $0\leq m\leq \ell(v)-1$ $R_{m+1}=\hat{\mu}_{m+1}(R_m)$ (respectively $R_{k,m+1}$ is the $k$-th summand of $R_{m+1}$).
		\end{nota}
		\begin{definition}[Index set $A_m(R)$]
			Let $1\leq m\leq \ell(v)$ and $R_{m-1}=\hat{\mu}_{m-1}\circ\cdots\circ\hat{\mu}_1(R_0)$ with summands indexed as they were in $V_{\bar{w}}=R_0$. Let $b_m=((p_m)_\Max)^{\gamma_m -}$. 
			
			We define: $A_m(R_{m-1})=\{1\leq i\leq b_m\mid \Delta_{\dot{v},m}(R_{i,m-1})\neq 0\}$. We order the elements of $A_m$ by ascending order, denote $(A_m(R_{m-1}))_1$ the smallest one, the $j$-th smaller $(A_m(R_{m-1}))_j$ and $(A_m(R_{m-1}))_\Max$ the largest.
		\end{definition}
		
		We can now define the sequence of mutation $\hat{\mu}_m$ for any $m$ by induction.
		\begin{definition}\label{definitionHatMu}
			Let $m\geq 1$. Suppose we have already defined: $R_{m-1}=\hat{\mu}_{m-1}\circ\cdots\circ\hat{\mu}_1(R_0)$. We write $A_m=A_m(R_{m-1})$. Then 
			\[
				\hat{\mu}_m:=\left\{\begin{array}{cc}\mu_{A_{m,\Max}}\circ\mu_{A_{m,\Max-1}}\circ\cdots\circ \mu_{A_{m,1}}&\text{if }A_m\neq\varnothing\\\mathrm{id}&\text{otherwise}\end{array}\right.
			\]
			
			We can summarize it by saying that we will mutate in increasing order all the summands of indices smaller than a given bound, which have a nonzero $m$-th coordinate in their $\Delta_{\dot{v}}$-vector.
		\end{definition}
		\begin{definition}
			We define $\mu_\bullet(V_{\bar{w}})=S(\hat{\mu}_{\ell(v)}\circ\cdots\circ\hat{\mu}_1(V_{\bar{w}}))$ where $S$ is the deletion of all the direct summands $R_{k,\ell(v)}$ such that $k>(k_\Max)^{\alpha(k,\ell(v))-}$.
		\end{definition}
		\begin{ex}
			Please see Example \ref{exampleAlgo2} p. \pageref{exampleAlgo2}.
		\end{ex}
	\subsection{Main theorem}
		We will now state the main theorem of this article:
		\begin{theorem}\label{theoremMain}
			$\mu_\bullet(V_{\bar{w}})$ is a $\C_{v,w}$-rigid maximal basic module. The quiver of $\mu_\bullet(V_{\bar{w}})$ is obtained from $\hat{\mu}_{\ell(v)}\circ\cdots\circ\hat{\mu}_1(V_{\bar{w}})$ by removing all the arrows adjacent to a vertex deleted by $S$ and only these ones.
		\end{theorem}
		
		The next section will be dedicated to the proof of this theorem.
\section{Proof}
	As we said, the main issue here is to prove that $\mu_{\bullet}(V_{\bar{w}})$  is in $\C_{v,w}$. In order to prove it we will do an induction which will study some properties of a part of the successive seeds. In the next section we define this subseed.
	
	We recall that $R_0=V_{\bar{w}}$ and for $1\leq m\leq\ell(v)$, $R_m=\hat{\mu}_m\circ\cdots\circ\hat{\mu}_1(R_0)$. Accordingly, $\Gamma_0=\Gamma_{\bar{w}}$ and $\Gamma_m=\hat{\mu}_m\circ\cdots\circ\hat{\mu}_1(\Gamma_0)$.
	\subsection{Cut seed and evicted summands}\label{sectionCutSeed}
		\begin{definition}\label{definitionCutSeed}
			Given a seed $G_m=\hat{\mu}_m\circ\cdots\circ\hat{\mu}_1((V_{\bar{w}},\Gamma_{\bar{w}}))$, we define the cut seed $C(R_m)$ as the couple $(C(R_m),C(\Gamma_m))$ where:
			\begin{itemize}
				\item the module $C(R_m)$ is obtained from $R_m$ by removing the direct indecomposable summands $R_{i,m}$ such that $\widetilde{\Delta}_{\dot{v}}(R_{i,m})=0$. We call them \emph{evicted modules}. We also remove the direct indecomposable summands $R_{i,m}$ such that $i>(i_\Max)^{\alpha(i,m)-}$. We call them \emph{deleted modules}.

				\item the quiver $C(\Gamma_{\bar{w}})$ obtained from $\Gamma_{\bar{w}}$ by removing all deleted or evicted vertices and the arrows iff they are adjacent to at least one of these vertices.
			\end{itemize}
			The other summands keep the same index as in $R_m$.
		\end{definition}
		
		We can note that, for the final step $m=\ell(v)$, the deleted summands correspond to the summands erased by $S$. Provided our algorithm works as intended, at this point, all the other summands must be in $\C_{v,w}$ and thus evicted and so the final cut seed must be empty.
		
		We now state some additional properties of the evicted summands. We will place ourselves under the following hypothesis using the notations introduced in Section \ref{sectionSpecificNotations}. This hypothesis generalizes the description given in Theorem \ref{theoremStructureInitialDeltaVectors} and will be a part of the induction hypothesis for the proof of the algorithm.
		\begin{hypothesis}\label{hypothesisDeltaVector}
			The non-zero coordinates of $\widetilde{\Delta}_{\dot{v}}(R_k)$ are the elements of the subset:
			\[
				\{1\leq j\leq \ell(v)\mid f_{\min}(k)^{\alpha(k,m)\oplus}\leq j\leq f(k^{\alpha(k,m)+})\mid i_{p_j}=i_k\}
			\]
		\end{hypothesis}
		\begin{ex}
			Example \ref{exampleAlgo2} is an example where this hypothesis is always verified.
		\end{ex}
		\begin{lemma}\label{lemmaXiTrivial}
			If $i_k=i_{p_m}$ then $f_\Min(k)^{\alpha(k,m-1)\oplus}=f(p_m)=m$
		\end{lemma}
		\begin{proof}
			$f_\Min(k)^{\alpha(k,m-1)\oplus}$ is the $\alpha(k,m-1)$-th successor (among the letters of $\bar{v}$) of the first letter in $\bar{v}$ of color $i_k$. As the $m$-th letter of $\bar{v}$ is of color $i_k$, $f_\Min(k)=f_\Min(p_m)$ and $f_\Min(p_m)^{\alpha(k,m-1)\oplus}=m=f(p_m)$ by definition.
		\end{proof}
		In order to have such a result for any index in $\bar{w}$, we introduce the following notation.
		\begin{definition}
			The application $\xi(k,m)$, $1\leq k\leq\ell(w)$ and $0\leq m\leq\ell(v)$ is defined by
			\[
				\xi(k,m)=\left\{\begin{array}{cc}\min\{p_m<j\leq\ell(w)\mid \exists 1\leq q\leq\ell(v)\mid j=p_q\text{ et }i_j=i_k\}&\text{if non-empty}\\
				\ell(w)+1&\text{otherwise}\end{array}\right.
			\]
			where $p_0=0$.
			
			$\xi(k,m)$ is the $\bar{w}$-coordinate of the first letter of $\bar{v}$ of color $i_k$ and of index strictly greater than $p_m$.
		\end{definition}
		\begin{proposition}
			$f(\xi(k,m))=f_\Min(k)^{\alpha(k,m)\oplus}$.
		\end{proposition}
		\begin{proof}
			If $m=0$, $\xi(k,0)=f_\Min(k)$ by definition.
			
			If $m\neq 0$ and $i_k=i_{p_m}$, $f(\xi(k,m))=p_m^\oplus=f_\Min(k)^{(\alpha(k,m-1)+1)\oplus}$ thanks to Lemma \ref{lemmaXiTrivial}. As $\alpha(k,m)=\alpha(k,m-1)+1$ here as $i_k=i_{p_m}$ we deduce the result.
			
			If $i_k\neq i_{p_m}$ we prove by induction on $m$. We already have the initialization when $m=0$. Suppose that $f(\xi(k,m-1))=f_\Min(k)^{\alpha(k,m-1)\oplus}$. As in this case $\alpha(k,m-1)=\alpha(k,m)$, we have to show that $\xi(k,m-1)=\xi(k,m)$. 
			
			If these two numbers were different, there would be a letter of $\bar{v}$ of color $i_k$ in the interval $\llbracket p_{m-1}+1,p_m-1\rrbracket$. This is impossible, as $\bar{v}$ is the rightmost representative. Thus $\xi(k,m-1)=\xi(k,m)$ and the heredity follows.
		\end{proof}
		\begin{corollary}\label{corollaryGeneralEvictions}
			Let $R_k$ a module verifying Hypothesis \ref{hypothesisDeltaVector} then $R_k\in\C^v$ iff $k^{\alpha(k,m)+}<\xi(k,m)$.
		\end{corollary}
		\begin{proof}
			In this case, as we have $f(k^{\alpha(k,m)+})<f(\xi(k,m))=f_\Min(k)^{\alpha(k,m)\oplus}$, the set of $\widetilde{\Delta}_{\dot{v}}$-coordinates not equal to $0$ is empty.
			
			The inequality is strict. Even if $f$ is not strictly increasing, $\xi(k,m)=p_j$ (for a $1\leq j\leq\ell(v)$) is the $\bar{w}$-index of a letter of $\bar{v}$ thus, as $k^{\alpha(k,m)+}<\xi(k,m)$, $f(k^{\alpha(k,m)+})$ cannot be equal to $f(\xi(k,m))=j$.
		\end{proof}
		\begin{corollary}[Initial evictions]\label{corollaryInitialEvictions}
			The indecomposable summands $V_k$ of $V_{\bar{w}}$ such that $f(k)<f_\Min(k)$ are in $\C^v$ and thus are evicted summands.
		\end{corollary}
		\begin{proof}
			By using Theorem \ref{theoremStructureInitialDeltaVectors} we are in Hypothesis \ref{hypothesisDeltaVector} for $m=0$ and thus the Corollary \ref{corollaryGeneralEvictions} holds.
		\end{proof}
		\begin{corollary}\label{corollaryPredecessorEviction}
			If $R_k$ and $R_{k^-}$ are direct indecomposable summands of $R$ verifying Hypothesis \ref{hypothesisDeltaVector} and $R_k\in\C^v$ then $R_{k^-}\in \C^v$.
		\end{corollary}
		\begin{proof}
			$(k^-)^{\alpha(k,m)+}<k^{\alpha(k,m)+}<\xi(k,m)=\xi(k^-,m)$. By Corollary \ref{corollaryGeneralEvictions} $R_k^-\in\C^v$.
		\end{proof}

	\subsection{Induction theorem}
		The goal of the greedy algorithm is to annihilate all $\Delta_{\dot{v}}$-coordinates of index $\leq \ell(v)$. This will translate on the cut seed as a greedy induction where the goal is to remove all summands. In order for the heredity to be fulfilled we will have to introduce some technical conditions. 
		\begin{theorem}[Induction Theorem]\label{theoremInduction}
			With the previously defined notations, let $1\leq k\leq \ell(w)$ and $0\leq m\leq\ell(v)$, $G_{m}=C(R_m)$ and $\widetilde{\Gamma}_m=C(\Gamma_m)$. We have that:
			\begin{enumerate}
				\item\label{pointCoordinateAnnulation} any summand $G_{k,m}$ has $\Delta_{\dot{v},j}(G_{k,m})=0$ for any $1\leq j\leq m$.
				\item\label{pointHorizontalArrow} For any vertex $G_{k,m}$ such that $G_{k^+,m}$ is a vertex of $\widetilde{\Gamma}_m$, there is an arrow $G_{k^+,m}\leftarrow G_{k,m}$ in $\widetilde{\Gamma}_m$. Such arrows are called horizontal arrows.
				\item\label{pointOrdinaryArrow} The other arrows of $\widetilde{\Gamma}_m$ are between vertices of adjacent colors . The vertices are called ordinary arrows.
				\item\label{pointSawTeeth} For any couple of adjacent colors $(i_k,i_j)$, the bicolor subquiver of $\widetilde{\Gamma}_m$ according to $(i_k,i_j)$ has a saw-teeth configuration. In particular for a given vertex $k$ and a given adjacent color to $i_k$ there is at most two ordinary arrows between $k$ and vertices of color $i_k$.
				\item\label{pointPureSawTeeth} for any color $i_j$ adjacent to $i_{p_{m+1}}$, the bicolor subquiver of $\widetilde{\Gamma}_m$ according to $(i_{p_{m+1}},i_j)$ has a pure saw teeth configuration.
				\item\label{pointDeltavcoordinates} the coordinate $1\leq j\leq \ell(v)$ of  $\widetilde{\Delta}_{\dot{v}}(G_{k,m})$ is:
				\begin{itemize}
					\item $1$ if $i_{p_j}=i_k$ and $j$ satisfies $f_\Min(k)^{\alpha(k,m)\oplus}\leq j\leq f(k^{\alpha(k,m)+})$
					\item $0$ in any other case.
				\end{itemize}
				If $m=\ell(v)$ the point \ref{pointPureSawTeeth} will be considered as trivially true.
			\end{enumerate}
		\end{theorem}
		\begin{nota}
			We recall that the notion of horizontal/ordinary arrow is introduced in Section \ref{sectionInitialQuiverDefinition}, as well as the notion of bicolor subquiver and (pure) saw teeth configuration, the notion of adjacent colors in Definition \ref{definitionkplus} and the notations of point \ref{pointDeltavcoordinates} in Definition \ref{definitionCombinatorialNumbers}.
		\end{nota}
		\begin{remark}
			Note that point \ref{pointDeltavcoordinates} is the same hypothesis as Hypothesis \ref{hypothesisDeltaVector}.
		\end{remark}
	\subsection{Initialization}\label{sectionInitialization}
		
		For $m=0$ we have $R_0=V_{\bar{w}}$.
		
		We first need to look at vertices evicted or deleted in the cut seed. There is no deleted summands as $i_{\Max}^{\alpha(i,m)-}=i_\Max$ and, by definition, there exists no index $i$ such that $i>i_\Max$.
		
		The point \ref{pointCoordinateAnnulation} is an empty condition. 
		
		For the points \ref{pointHorizontalArrow} and \ref{pointOrdinaryArrow}  we remove only arrows adjacent to vertices also removed from $\Gamma_0$ which satisfy the condition by definition in Section \ref{sectionInitialQuiverDefinition}. Then $\widetilde{\Gamma}$ verifies also the condition.
		
		For the point \ref{pointSawTeeth}, thanks to Corollary \ref{corollaryInitialEvictions} we know that, on the one hand, thanks to Proposition \ref{propositionSawTeethStructureVwbar} the quiver has a saw teeth structure. On the other hand, $\widetilde{\Gamma}_0$ contains only vertices of index $k$ such that $f(k)\geq f_\Min(k)$. In particular, $\widetilde{\Gamma}_0$ is a subquiver of the quiver $\Gamma_{\bar{w}_{\geq p_1}}$ where $\bar{w}=[i_{\ell(w)},\cdots,i_{p_1+1},i_{p_1}]$. The quiver $\Gamma_{\bar{w}_{\geq p_1}}$ is obtained by removing consecutive initial vertices of $\Gamma_0$. By Proposition \ref{propositionStabilitySawTeethStructure}, the saw teeth structure is preserved by such a removal of initial vertices. Now, in order to get $\widetilde{\Gamma}_0$ from $\Gamma_{\bar{w}_{\geq p_1}}$ we need to remove vertices on lines other than $i_{p_1}$. However, thanks to Corollary \ref{corollaryPredecessorEviction} we remove consecutively the first vertices of given lines. Then, using once again Proposition \ref{propositionSawTeethStructureVwbar} we have that the saw teeth structure is preserved.
		
		We will now prove that $V_{\bar{w}}$ verifies the point \ref{pointPureSawTeeth} by contradiction.
		
		\begin{proof}
			First, note that, due to the rightmostness of $\bar{v}$, $(p_1)_\Min=p_1$. Thus any vertex of color $i_{p_1}$ is in $\widetilde{\Gamma}_0$. Now suppose that for one line of color $i_j$ adjacent to the line of color $i_{p_1}$ the ordinary arrow of minimal target index is an arrow $V_{(p_1)^{\gamma+}}\rightarrow V_j$ for $\gamma\geq 0$ and $1\leq j\leq \ell(w_0)$ of color $i_j$. We will show that there exists an arrow $V_j\rightarrow V_{(p_1)^{\zeta +}}$ where $\zeta<\gamma$. This implies that the inequality: $p_1^{(\gamma+1)+}\geq j^+>p_1^{\gamma+}>j$ holds and we want to show that there exists $\zeta <\gamma$ such that $j^+\geq p_1^{(\zeta+1)+}>j>p_1^{\zeta +}.$
			
			We have $\gamma>0$: if $\gamma=0$ then we have $j$ such that $j<p_1$ which is impossible as the quiver is a subquiver of $\Gamma_{\bar{w}_{\geq p_1}}$. So $p_1^{\gamma+}>j>1$ and there existe $\zeta<\gamma$ such that $p_1^{(\zeta +1)+}>j>p_1^{\zeta +}$. As, in addition we have $j^+>p_1^{\gamma +}$ and $p_1^{\gamma +}>1^{\zeta +}$, eventually we have: $j^+\geq p_1^{(\zeta+1)+}>j>p_1^{\zeta +}$ and, as $j$ and $p_1^{\zeta +}$ are vertices of $\widetilde{\Gamma}_0$ subquiver of $\Gamma_{\bar{w}}$ there is an arrow $V_j\rightarrow V_{p_1^{\zeta +}}$ and so the arrow $V_j\rightarrow V_{p_1^{\gamma+}}$ was not the one of minimal target index. Then the ordinary arrow of minimal target index between both lines has as target a vertex of color $i_{p_1}$. So the bicolor subquiver has a pure saw teeth structure.
		\end{proof}
		
		For point \ref{pointDeltavcoordinates}, the considered modules are a subset of the indecomposable summands of $V_{\bar{w}}$. Thanks to Theorem \ref{theoremStructureInitialDeltaVectors} we know that each of these summands has a $\widetilde{\Delta}$-vector as described. 
		
		Then we have proved that $R_0$ satisfies all the properties of Theorem \ref{theoremInduction}.
	\subsection{Induction lemma}
		We now want to prove that, given the seed $(G_m,\widetilde{\Gamma}_m)$ satisfying all properties of Theorem \ref{theoremInduction}, the next seed $(G_{m+1},\widetilde{\Gamma}_{m+1})$ satisfies all properties of Theorem \ref{theoremInduction} for any $0\leq m\leq\ell(v)$.
		
		We know how to get a cut seed from $(R_m,\Gamma_m)$ and how to mutate it in $(R_{m+1},\Gamma_{m+1})$ via $\hat{\mu}_m$ but obtaining $(G_{m+1},\widetilde{\Gamma}_{m+1})$ from $(G_{m},\widetilde{\Gamma}_{m})$ is more subtle and will be proven by another induction. We can summarize the situation by Figure \ref{figureCutSeedMutation}.
		\begin{figure}[h!t]
		\[
			\begin{tikzcd}(R_m,\Gamma_m)\ar[r,"\hat{\mu}_{m+1}"]\ar[d,"C"]&(R_{m+1},\Gamma_{m+1})\ar[d,"C"]\\(G_m,\widetilde{\Gamma}_m)\ar[r,dashed,"?"]&(G_{m+1},\widetilde{\Gamma}_{m+1})\end{tikzcd}
		\]
		\caption{Cut seeds and mutated seeds}
		\label{figureCutSeedMutation}
		\end{figure}
		
		\begin{nota}
			In this section we will consider the seed
			\[
				(R,\Gamma)=\mu_{j^{\gamma +}}\circ\cdots\circ\cdots \mu_j\circ\hat{\mu}_m\circ\cdots\circ\hat{\mu}_1(R_0,\Gamma_0)=\mu_{j^{\gamma+}}\circ\cdots\circ\mu_j(R_{m},\Gamma_m)
			\]
			where $\hat{\mu}_{m+1}=\mu_{j^{\zeta +}}\circ\cdots \circ \mu_{j^{\gamma +}}\circ\cdots\circ\mu_j$ and $\zeta\geq \gamma\geq 0$.
			
			We denote by $R_{i,m+1}$ the $i$-th summand of $R$ if $i\in\{j,j^+,\dots,j^{\gamma+}\}$ and $R_{i,m}$ otherwise. 
			
			We define the cut seed $C(R,\Gamma)=:(G,\widetilde{\Gamma})$ as the seed where
			\begin{itemize}
				\item the module $G$ is obtained from $R$ by removing, for $p\in\{m,m+1\}$, the direct indecomposable summands $G_{k,p}$ such that $\Delta_{\dot{v}}(R_{k,p})=0$ which we call evicted and the direct indecomposable summands $G_{k,p}$ such that $k>(k_\Max)^{\alpha(k,p)-}$ which we call deleted.
				\item the quiver $\widetilde{\Gamma}$ is obtained from $\Gamma$ by removing all vertices corresponding to evicted or deleted summand and arrows adjacents to these vertices.
			\end{itemize}
		\end{nota}
		We will now prove by induction the following lemma 
		\begin{lemma}[Induction lemma]\label{lemmaInduction}
			In the cut seed $(G,\widetilde{\Gamma})$:
			\begin{enumerate}
				\item\label{pointCoordinateAnnulationlemma} The summands $G_{k,m}$, provided $i_k\neq i_j$, and the summands $G_{j,m+1},\dots, G_{j^{\gamma+},m+1}$ have all their $\Delta_{\dot{v}}$-coordinates of index $\leq m+1$ equal to $0$.
				\item\label{pointHorizontalArrowlemma} The horizontal arrows are as in Theorem \ref{theoremInduction}, apart from the horizontal arrow $G_{j^{(\gamma+1)+},m}\rightarrow G_{j^{\gamma+},m+1}$ if $(G,\widetilde{\Gamma})\neq(G_m,\widetilde{\Gamma}_m)$.
				\item\label{pointOrdinaryArrowlemma} All non-horizontal arrows are between lines with adjacent colors.
				\item\label{pointSawTeethlemma} For any couple of color $(i_k,i_l)$ with $i_k,i_l\neq i_j$, the bicolor subquiver $(i_k,i_l)$ has a saw teeth configuration. The bicolor subquivers $(i_j,i_l)$ with $i_l$ adjacent to $i_j$, where we remove all vertices of color $i_j$ and of index $>j^{\gamma +}$, have a saw teeth configuration.
				\item\label{pointPureSawTeethlemma} The bicolor subquivers $(i_j,i_l)$ where we remove the vertices of color $i_j$ and of index $<j^{\gamma+}$ have a pure saw teeth structure.
				\item\label{pointDeltavcoordinateslemma} $\widetilde{\Delta}_{\dot{v}}(G_{k,p})$ has all coordinates equal to 0 except from the one of color $i_j$ and of index in the interval:
				\[
					f_{\min}(k)^{\alpha(k,p)\oplus}\leq l\leq f(k^{\alpha(k,p)+})
				\]
				which are equal to 1.
			\end{enumerate}
		\end{lemma}
		\begin{ex}
			There are examples of these properties in Example \ref{exampleAlgo2}, with the notation $V_k^\ast=R_{k,1}$ and $R_{k,m}^\ast=R_{k,m+1}$. See especially Figure \ref{figureMutation1} \& \ref{figureMutation2}.
		\end{ex}
		We will first prove the initialization of Lemma \ref{lemmaInduction}, i.e., the case where $(G,\widetilde{\Gamma})=(G_m,\widetilde{\Gamma}_m)$. It consists in translating the points of Theorem \ref{theoremInduction} into the points of Lemma \ref{lemmaInduction}. 
		
		The point \ref{pointCoordinateAnnulationlemma} of Lemma \ref{lemmaInduction} comes from the points \ref{pointCoordinateAnnulation} and \ref{pointDeltavcoordinates} of Theorem \ref{theoremInduction}. The first point ensures that all coordinates of index $\leq m$ are equal to $0$. As the $m+1$ summand is of color $i_{p_{m+1}}=i_j$, for summands of color different from $i_j$, the $m+1$th coordinate is equal to zero thanks to point \ref{pointDeltavcoordinates} of the Theorem. The list of summands $G_{j,m+1},\dots, G_{j^{\gamma+},m+1}$ is empty. The point is proved.
		
		The points \ref{pointHorizontalArrowlemma} and \ref{pointOrdinaryArrowlemma} of the Lemma are equivalent to the points \ref{pointHorizontalArrow} and \ref{pointOrdinaryArrow} of the Theorem.
		
		The point \ref{pointSawTeethlemma} of the Lemma consists in looking only at bicolor subquivers of colors $(i_k,i_l)$ with $i_k,i_l\neq i_j$, a subset of the bicolor subquivers known to have a saw teeth structure thanks to the point \ref{pointSawTeeth} of the Theorem.
		
		The points \ref{pointPureSawTeethlemma} and \ref{pointDeltavcoordinateslemma} of the lemma are exactly the points \ref{pointPureSawTeeth} and \ref{pointDeltavcoordinates} of the Theorem.
		
		Then the seed $(G_m,\widetilde{\Gamma}_m)$ satisfies all conditions of Lemma \ref{lemmaInduction}.
	\subsection{Heredity of the induction lemma}
		We will now establish the heredity of the hypotheses of Lemma \ref{lemmaInduction}.
		
		We will deduce point \ref{pointCoordinateAnnulationlemma} of the lemma from point \ref{pointDeltavcoordinateslemma}.
		
		For the point \ref{pointHorizontalArrowlemma}, we have two cases. If the vertex  we are mutating, $G_{j^{\gamma +},m}$, is the first of its line, it has only one horizontal arrow $G_{j^{(\gamma+1) +},m}\leftarrow G_{j^{\gamma +},m}$. As, from point \ref{pointPureSawTeethlemma} we know that there is no ordinary arrow in $\widetilde{\Gamma}$ having $G_{j^{\gamma+},m}$ as source, we know that there will not be any disturbance of other horizontal arrow on other lines. 
		
		Then the only horizontal arrow to be changed will be the one having $G_{j^{\gamma +},m}$ as source which will become $G_{j^{(\gamma+1) +},m}\rightarrow G_{j^{\gamma +},m+1}$, which is precisely the only derogation we allowed, now considering $G_{j^{(\gamma+1)+},m}$. Now, in order to get the next $\widetilde{\Gamma}$, either we evict $G_{j^{\gamma +},m+1}$ and we are once again in the initial case with no derogation, or not and we still validate the hypothesis.
		
		If we are not in an initial case then we have $G_{j^{(\gamma+1)+},m}\leftarrow G_{j^{\gamma+},m}\rightarrow G_{j^{(\gamma-1)+},m+1}$ before mutation. As before, thanks to point \ref{pointPureSawTeethlemma}, we know that the only horizontal arrow to be changed will be the one adjacent to $G_{j^{\gamma +},m}$ and after mutation we have $G_{j^{(\gamma+1)+},m}\rightarrow G_{j^{\gamma+},m+1}\leftarrow G_{j^{(\gamma-1)+},m+1}$ so we still validate the hypothesis when looking at $G_{j^{(\gamma+1)+},m}$. We cannot have eviction here thanks to Corollary \ref{corollaryPredecessorEviction}.
		
		The points \ref{pointOrdinaryArrowlemma}, \ref{pointSawTeethlemma} and \ref{pointPureSawTeethlemma} will result from an enumeration of cases, which we postpone to Section \ref{sectionConfigurationsInheritageTeethShift} in the appendix.
		
		Indeed, we initially verify the condition of the case enumeration and then with the evictions and mutation we stay in one of the studied case.
		
		As all the bicolor subquivers of first color $i_j$ are in pure saw teeth structure, all the ordinary arrow adjacent to a vertex to mutate will have it as target and thus it is impossible, in the cut quiver, to create arrow between two lines of non-adjacent colors. 
		
		By the study of configurations, looking only at what happens at the already mutated vertices, we see that we keep the saw teeth structure. By the same point, when only looking at the vertices still to mutate, we see that we still are in a pure saw teeth structure. 
		
		We now need to compute the $\Delta$-vectors after mutation. Using Proposition \ref{propositionMutationDeltaVectors} we have that $\Delta$-vectors can be computed by exactly one of the two computations, one corresponding to taking into account neighbours of $G_j$ via arrow of source $G_j$, the other one via arrow of target $G_j$. 
		
		Here we will only focus on the $\widetilde{\Delta}$-vector, only the $\ell(v)$ first coordinates being relevant to our study and then we allow ourself to only consider non-evicted vertices, the evicted one having zero $\widetilde{\Delta}_{\dot{v}}$-vectors. 
		
		We have a more precise criterion to know which summands to take into account on the computation:
		\begin{proposition}\label{propositionUnidirectionArrowQuiver}
			The neighbouring of $G_j$ in $\widetilde{\Gamma}$ just before mutating is composed of the ordinary arrow having $G_j$ as target, the horizontal arrows having $G_j$ as source, and no other arrows.
		\end{proposition}
		
		It follows from the study of all cases. Thus the computation of the $\widetilde{\Delta}_{\dot{v}}$-vectors amounts to taking into account only the ordinary arrows or only the horizontal arrows.
		
		We now should use the computational criterion of \cite{geiss2011KacMoodyGroups} to know which one to use for the computation but we will show, by contradiction, that in fact we can only consider the horizontal arrows.
		\begin{lemma}[Evolution of $\Delta_{\dot{v}}$-vectors coordinates by mutation]\label{lemmaComputationDeltavVector}
			In the cut quiver $\widetilde{\Gamma}$, if the vertex $G_j$ we mutate has a configuration described by points \ref{pointSawTeethlemma} and \ref{pointPureSawTeethlemma} of the lemma \ref{lemmaInduction}, then
			\[
				\widetilde{\Delta}_{\dot{v}}(G_{j,m+1}=\widetilde{\Delta}_{\dot{v}}(G_{j^+,m})+\widetilde{\Delta}_{\dot{v}}(G_{j^-,m+1})-\widetilde{\Delta}_{\dot{v}}(G_{j,m})
			\]
			where we write $\Delta_{\dot{v}}(G_{j^-,m+1})=0$ if there is no summand/vertex of index $j^-$ in the cut module/quiver.
		\end{lemma}
		\begin{proof}
			First we will show that we have to consider summands of the same color as they are linked to $G_{j,m}$ by horizontal arrows.
			
			If we look at ordinary arrows, we will have the computation 
			\[
				\sum_{G_{i,m}\rightarrow G_{j,m}}\widetilde{\Delta}_{\dot{v}}(G_{i,m})-\widetilde{\Delta}_{\dot{v}}(G_{j,m})
			\]
			and if we look at the $m+1$st coordinate, then $\Delta_{\dot{v},m+1}(G_{i,m})=0$ for any $G_i$ adjacent to $G_j$ via an ordinary arrow and $\Delta_{\dot{v},m+1}(G_{j,m})=1$ by the induction hypothesis \ref{pointDeltavcoordinateslemma} of Lemma \ref{lemmaInduction}. Thus we would have $\Delta_{\dot{v},m+1}(G_{j,m+1})=0-1<0$. This is absurd by the definition of $\Delta$-vectors and then we have to consider horizontal arrows.
			
			Then there are only two neighbours of $G_{j,m}$: $G_{j^-,m+1}$ (if it exists) and $G_{j^+,m}$ (which always exists). Thus the computation is the one given in the lemma. 
		\end{proof}
		
		Now we suppose that, according to point \ref{pointDeltavcoordinateslemma} of Lemma \ref{lemmaInduction}, we have that the non-zero and equal to 1 coordinates of $\widetilde{\Delta}_{\dot{v}}(G_{k,p})$ are the one of index $1\leq l\leq\ell(v)$ having the color $i_k$ and being in the interval:
		\[
			f_\Min(k)^{\alpha(k,p)\oplus}\leq l\leq f(k^{\alpha(k,p)+})
		\]
		
		In particular we have:
		\[
			f_\Min(j)^{(\alpha(j,m)+1)\oplus}\leq l\leq f(j^{\alpha(j,m)+})\quad\text{ for }G_{j^-,m+1}
		\]
		\[
			f_\Min(j)^{\alpha(j,m)\oplus}\leq l\leq f(j^{\alpha(j,m)+})\quad\text{ for }G_{j,m}
		\]
		\[
			f_\Min(j)^{\alpha(j,m)\oplus}\leq l\leq f(j^{(\alpha(j^+,m)+1)+})\quad\text{ for }G_{j^+,m}
		\]
		using the facts that $f_\Min(j)=f_\Min(j^+)=f_\Min(j^-)$, that $\alpha(j,m)=\alpha(j^+,m)=\alpha(j^-,m)$ and that $\alpha(j,m)+1=\alpha(j,m+1)$.
		
		We want to show that 
		\[
			f_\Min(j)^{\alpha(j,m+1)\oplus}\leq l\leq f(j^{\alpha(j,m+1)+})\quad\text{ for }G_{j,m+1}
		\]

		Then we can cut these intervals in three intervals:
		\[
			f_\Min(j)^{\alpha(j,m)\oplus}\leq l< f_\Min(j)^{(\alpha(j,m)+1)\oplus}
		\]
		\[
			f_\Min(j)^{(\alpha(j,m)+1)\oplus}\leq l\leq f(j^{\alpha(j,m)+})
		\]
		\[
			f(j^{\alpha(j,m)+})<l\leq f(j^{(\alpha(j,m)+1)+})
		\]
		
		One these intervals, we have respectively :
		\[
			(\Delta_{\dot{v},l}(G_{j^-,m+1}),\Delta_{\dot{v},l}(G_{j,m}),\Delta_{\dot{v},l}(G_{j^+,m}))=(0,1,1), \text{on the first},
		\]
		\[
			(\Delta_{\dot{v},l}(G_{j^-,m+1}),\Delta_{\dot{v},l}(G_{j,m}),\Delta_{\dot{v},l}(G_{j^+,m}))=(1,1,1), \text{on the second},
		\]
		\[
			(\Delta_{\dot{v},l}(G_{j^-,m+1}),\Delta_{\dot{v},l}(G_{j,m}),\Delta_{\dot{v},l}(G_{j^+,m}))=(0,0,1), \text{on the third.}
		\]
		So we have as coordinate for $\widetilde{\Delta}_{\dot{v},l}(G_{j,m+1})$, on the first interval $0-1+1=0$, on the second: $1-1+1=1$, on the third $0-0+1=1$. 
		
		Finally, in $\widetilde{\Delta}_{\dot{v}}(G_{j,m+1})$, the non-empty coordinates are equal to 1 and are the ones of index $l$ of color $i_j$ and in the interval $f_\Min(j)^{\alpha(j,m+1)\oplus}\leq l\leq f(j^{\alpha(j,m+1)+})$ whence the heredity of point \ref{pointDeltavcoordinateslemma} of Lemma \ref{lemmaInduction} and the point \ref{pointCoordinateAnnulationlemma} coming from point \ref{pointDeltavcoordinateslemma}. 
		
		Then we have proved that Lemma \ref{lemmaInduction} holds by induction.
	\subsection{Toward heredity proof of the induction theorem}
		As Lemma \ref{lemmaInduction} is proved, when taking the result for the seed $(R,\Gamma)=\hat{\mu}_{m+1}\circ\cdots\circ\hat{\mu}_1(V_{\bar{w}},\Gamma_{\bar{w}})$ we have:
		\begin{enumerate}
			\item The summands $G_{k,m+1}$ belonging to $C(R_{m+1})$ have all their coordinates of $\Delta_{\dot{v}}$-vector of index $\leq m+1$ equal to $0$.
			\item In the cut quiver $\widetilde{\Gamma}_{m+1}$ (obtained from $\Gamma$ by removing the last vertex of line $i_j$ and the evicted vertices), all the horizontal arrows are conform to \ref{pointHorizontalArrow} of Theorem \ref{theoremInduction}.
			\item All ordinary arrows are between lines of adjacent colors in $\widetilde{\Gamma}_{m+1}$.
			\item The quiver $\widetilde{\Gamma}_{m+1}$ has a saw teeth structure.
			\item This condition is empty.
			\item The description of $\widetilde{\Delta}_{\dot{v}}$-coordinates of the summands is conform to point \ref{pointDeltavcoordinates} of Theorem \ref{theoremInduction}.
		\end{enumerate}
		
		Then, up to point \ref{pointPureSawTeeth} we have the heredity of the hypothesis. It remains to show that the next line to mutate will have a pure saw teeth structure.
		
		To show that final point we need a new lemma, introduced and proved in the next section.
	\subsection{Heredity: inclusion lemma}
		Our goal here is to show that in the seed $\widetilde{\Gamma}_{m+1}$ the line of color $i_{p_{m+1}}$ has a pure saw teeth structure for any bicolor subquiver having $i_{p_{m+1}}$ as first color. In order to do so, we will use (and prove) the fact that the cut seed $\widetilde{\Gamma}_{m+1}$ has the same quiver as the left (maximal indices) part of the quiver of $\Gamma_0$.
		
		\begin{definition}[Graph morphism]
			Let $\Gamma=(\Gamma_0,\Gamma_1)$ and $\Gamma'=(\Gamma_0',\Gamma_1')$ be two graphs. A map $f:\Gamma_0\rightarrow\Gamma_0'$ is a graph morphism if, for any vertex $\alpha_{j_s,j_t}\in\Gamma_1$ of source the vertex $V_{j_s}\in\Gamma_0$ and of target $V_{j_t}$, the image vertex $f(\alpha_{j_s,j_t})$ of source $f(V{j_s})$ and of target $f(V_{j_t})$ belongs to $\Gamma'_1$. If $f$ is injective, we say that $\Gamma$ is included in $\Gamma'$.
		\end{definition}
		\begin{lemma}[Inclusion lemma]
			Let $0\leq m\leq\ell(v)$, $\Gamma_{W_{m}}=(W_{i,m},\beta_{j_s,j_t,m})$ and $\widetilde{\Gamma}_m=(G_{i,m},\alpha_{j_s,j_t,m})$ be the quivers of the modules (respectively) $V_{\bar{w}_{\geq p_{m+1}}}\text{ and } C(\hat{\mu}_m\circ\cdots\circ\hat{\mu}_1(V_{\bar{w}}))$,using the convention that $s(\gamma_{j_s,j_t,m})=j_s$ and $t(\gamma_{j_s,j_t,m})=j_t$ for any vertex.
			
			We define the family of maps $(\rho_m)_{m=0}^{\ell(v)}$ by: $\rho_m:\left\{\begin{array}{ccc}(G_{i,m},\alpha_{j_s,j_t,m})&\rightarrow&(W_{i,m+1},\beta_{j_s,j_t,m+1})\\G_{i,m}&\mapsto&W_{i^{\alpha(i,m)+},m+1}\end{array}\right..$
			
			These maps are injective graph morphisms and we have a partial inverse:
			\[
				\rho_m^\ast:\left\{\begin{array}{ccl}(W_{i,m+1},\beta_{j_s,j_t,m+1})&\rightarrow &(G_{i_m},\alpha_{j_s,j_t,m})\\W_{i,m+1}&\mapsto&\left\{\begin{array}{ccc}G_{i^{\alpha(i,m)-},m}&\text{if}&i^{\alpha(i,m)-}\geq \xi(i,m)\\\notin D_{\rho_m^\ast}&\text{otherwise}&\end{array}\right.\end{array}\right.
			\]
		\end{lemma}
		\begin{proof}
			We need two more functions (generally not graph morphisms) to prove the result by induction.
			
			We define $\iota_m:\Gamma_{W_m}\rightarrow \Gamma_{W_{m+1}}$ by the following identifications:
			\[
				\iota_m:\left\{\begin{array}{ccl}W_{j,m}&\mapsto&\left\{\begin{array}{ccl}W_{j^+,m+1}&\text{if}&i_j=i_{p_m}\text{ and }(p_m)_\Max\geq j^+\geq p_{m+1}\\W_{j,m+1}&\text{ otherwise if }&i_j\neq i_{p_m} \text{ and } j\geq p_{m+1}\\\notin D_{\iota_m}&\text{otherwise,}& \end{array}\right.\\
				\beta_{j_s,j_t,m}&\mapsto&\left\{\begin{array}{ccc}\notin D_{\iota_m}&\text{if}&W_{j_s,m}\notin D_{\iota_m}\text{ or }W_{j_t}\notin D_{\iota_m}\\\beta_{j_s^+,j_t^+,m+1}&\text{otherwise if}& i_{j_s}=i_{j_t}=i_{p_m}\\\beta_{j_s,j_t,m+1}&\text{otherwise,}&\end{array}\right.\end{array}\right.
			\]
			where $D_{\iota_m}$ is the domain of $\iota_m$.
			
			We define $\bar{\mu}_m:\widetilde{\Gamma}_{m-1}\rightarrow \widetilde{\Gamma}_m$ as:
			\[
				\bar{\mu}_m:\left\{\begin{array}{ccl}G_{j,m-1}&\mapsto&\left\{\begin{array}{ccl}G_{j,m}&\text{if}&i_j\neq i_{p_m}\text{ or } j_\Max>j^{\alpha(j,m)+}\geq \xi(j,m)\\\notin D_{\bar{\mu}_m}&\text{otherwise,}&\end{array}\right.\\\alpha_{j_s,j_t,m-1}&\mapsto&\left\{\begin{array}{ccc}\\\alpha_{j_s,j_t,m}&\text{if}&i_{j_s}\neq i_{p_m}\neq i_{j_t}\text{ or }i_{j_s}=i_{p_m}=i_{j_t}\\\alpha_{j_s^-,j_t,m}&\text{if}&i_{j_s}=i_{p_m}\neq i_{j_t}\text{ and } G_{j_s^-,m-1},G_{j_t,m-1}\in D_{\bar{\mu}_m}\\\alpha_{j_s,j_t^-,m}&\text{if}&i_{j_s}\neq i_{p_m}= i_{j_t}\text{ and } G_{j_s,m-1},G_{j_t^-,m-1}\in D_{\bar{\mu}_m}\\\notin D_{\bar{\mu}_m}&\text{otherwise}&\end{array}\right.\end{array}\right.
			\]
			and for any arrow $\alpha_{j_s,j_t,m-1}$ such that there is no arrow $\alpha_{j_t^{\gamma+},j_s,m-1}$ with $\gamma>0$ and verifying that $j_t^{\alpha(j_t,m-1)+}\neq (j_t)_\Max$ and $i_{j_t}=i_{p_m}$ (when we are on a final barb shifted), we add an arrow $\alpha_{j_{t_2},j_s,m}$ where $j_{t_2}=(j_t)_\Max^{\alpha(j_t,m)-}$.
			
			First we show that $\iota_m$ is well-defined, and that $\iota_m(\Gamma_{W_m})=\Gamma_{W_{m+1}}$. Then we will show that $\bar{\mu}_m$ is well-defined, that $\bar{\mu}_m(\widetilde{\Gamma}_{m-1})=\widetilde{\Gamma}_m$ and that $\rho_m$ is a graph morphism.
			
			Let $0\leq m<\ell(v)$. By definition, $\Gamma_{W_{m+1}}$ is the quiver obtained from $\Gamma_{W_m}$ by removing all summands of index $p_{m+1}>j\geq p_m$ and the adjacent arrows.
			
			The vertices having no image by $\iota_m$ are $W_{(p_m)_\Max,m}$, the vertices $W_{j,m}$ such that $i_j\neq i_{p_m}$ and $p_{m+1}>j>p_m$, and the vertices $W_{j,m}$ such that $i_j=i_{p_m}$ and $p_{m+1}>j^+>p_m$.
			
			Thus among the vertices having an image via $\iota_m$, we have a bijection $W_{j,m}\leftrightarrow W_{j,m+1}$ if $i_j\neq i_{p_m}$ and $W_{j,m}\leftrightarrow W_{j^+,m}$ if $i_j=i_{p_m}$.
			
			In addition the vertices of $\Gamma_{W_{m+1}}$ are the same of those of $\Gamma_{W_m}$ except the ones adjacent to a vertex of index $p_{m+1}>j\geq p_m$. Among the preserved arrows, for the arrows $\alpha_{j_s,j_t,m}$ such that at least one of the extremity is not of color $i_{p_m}$ we have a bijection $\alpha_{j_s,j_t,m}\leftrightarrow\alpha_{j_s,j_t,m+1}$. If both extremities are of color $i_{p_m}$ then the arrow is an horizontal arrow $\alpha_{k,k^+,m}$ and we have the bijection $\alpha_{k,k^+,m}\leftrightarrow\alpha_{k^+,k^{2+},m+1}$. The arrow $\alpha_{k,k^{2+},m+1}$ belongs to $\Gamma_{W_{m+1}}$ because we suppose that $k^+\geq p_{m+1}$ and $k^+\neq k_\Max$ so $W_{k^{2+},m+1}\in\Gamma_{W_{m+1}}$.
			
			Thus we have a bijection between $\Gamma_{W_{m+1}}$ and $\iota_m(\Gamma_{W_m})$ for $0 \leq m< \ell(v)$.
			
			Now we consider $\widetilde{\Gamma}_{m+1}$. We will show that $\bar{\mu}_m(\widetilde{\Gamma}_{m+1})$. To fix notations we will suppose that $\hat{\mu}_m=\mu_{k^{\gamma+}}\circ\cdots\circ\mu_{k}$.
			
			Then $\bar{\mu}_m$ amounts exactly to the modifications described by Proposition \ref{propositionTeethShift} and the cut.
			
			Thus $\iota_m$ and $\bar{\mu}_m$ are well-defined and the images are the one expected.
			
			We will now show that $\rho_m$ is a morphism of graphs for any $0\leq m\leq \ell(v)$ by induction.
			
			For $\rho_0$, $\widetilde{\Gamma}_0$ is the subquiver of $\Gamma_{\bar{w}}$ containing all vertices $R_{j,0}$ such that $j\geq \xi(j,0)$ (by Corollary  \ref{corollaryInitialEvictions}), $\Gamma_{W_1}$ is the subquiver of $\Gamma_{\bar{w}}$ containing all summands such that $j\geq p_1$ and for $1\leq j\leq\ell(w)$ we have $\xi(j,0)\geq p_1$. $\widetilde{\Gamma}_0$ is then a subquiver of $\Gamma_{W_1}$. As for all $j$ we have $\alpha(j,0)=0$, there is no vertex shift nor arrows shift thus $\rho_0$ is well-defined and is just an inclusion. The partial inverse exists and is trivial.
			
			Now we will show the heredity. Suppose that $\rho_{m-1}$ is a morphism of graphs. We will show that $\rho_m$ is also a morphism of graphs. To do so, we will show that we have $\rho_m\circ\bar{\mu}_m=\iota_m\circ\rho_{m-1}$ on their definition domain, having the commutative diagram:
			\[
				\begin{tikzcd}\widetilde{\Gamma}_{m-1}\ar[r,"\rho_{m-1}"]\ar[d,"\bar{\mu}_m"']&\Gamma_{W_m}\ar[d,"\iota_m"]\\
				\widetilde{\Gamma}_m\ar[r,"\rho_m"']&\Gamma_{W_{m+1}}
				\end{tikzcd}
			\]
			
			We first show the equality on vertices. We have four cases (provided the previous one are not verified) for the index of a vertex $G_{j,m-1}$:
			\begin{enumerate}
				\item either $j^{\alpha(j,m-1)+}=(p_m)_\Max$
				\item or $j^{\alpha(j,m)+}<p_{m+1}$
				\item or $\xi(j,m)>j^{\alpha(j,m)+}\geq p_{m+1}$
				\item or $j^{\alpha(j,m)+}\geq \xi(j,m)$
			\end{enumerate}
			In the first case $G_{j,m-1}\not\in D_{\bar{\mu}_m}$. On the other side $\rho_{m-1}(G_{j,m-1})=W_{j^{\alpha(j,m-1)+},m}$ and $W_{j^{\alpha(j,m-1)+},m}\notin D_{\iota_m}$ as $j^{\alpha(j,m-1)+}=(p_m)_\Max$ and thus $(j^{\alpha(j,m-1)+})^+>(p_m)_\Max$.
			
			In the second case, if $i_j=i_{p_m}$, then $G_{j,m-1}\notin D_{\bar{\mu}_m}$ as $\xi(i,m)\geq p_{m+1}$ for any $1\leq i\leq \ell(w)$. On the other side we have $\rho_{m-1}(G_{j,m-1})=W_{j^{\alpha(j,m-1)+},m}\notin D_{\iota_m}$. If $i_j\neq i_{p_m}$ then $j^{\alpha(j,m)+}<p_{m+1}\leq\xi(j,m)=\xi(j,m-1)$ and thus this vertex is not in $\widetilde{\Gamma}_{m-1}$.
			
			In the third case, it depends on the color. If $i_j\neq i_{p_m}$ then $\xi(j,m)=\xi(j,m-1)$ and thus $G_{j,m-1}$ does not exist. Then $i_j=i_{p_m}$ and as $\xi(j,m)=p_{m+1}$, the interval is empty.
			
			In the fourth case we have $\bar{\mu}_m(G_{j,m-1})=G_{j,m}$ and $\rho_m(G_{j,m})=W_{j^{\alpha(j,m)+},m+1}$. On the other side, we have $\rho_{m-1}(G_{j,m-1})=W_{j^{\alpha(j,m-1)+},m}$ and then $\iota_m(W_{j^{\alpha(j,m-1)+},m})=_{j^{\alpha(j,m)+},m+1}$ using the fact that $\alpha(j,m-1)=\alpha(j,m)$ if $i_j\neq i_{p_m}$ and $\alpha(j,m-1)+1=\alpha(j,m)$ if $i_j=i_{p_m}$.
			
			Thus we do have $\rho_m\circ\bar{\mu}_m=\iota_m\circ\rho_{m-1}$ for any vertex of $\widetilde{\Gamma}_m$.
			
			We will now prove the equality for all arrows. We have four cases for $\alpha_{j_s,j_t,m-1}$
			\begin{enumerate}
				\item $G_{j_s,m-1}$ or $G_{j_t,m-1}$ is not in $D_{\rho_m\circ\bar{\mu}_m}=D_{\iota_m\circ\rho_{m-1}}$,
				\item $i_{j_s}\neq i_{p_m}\neq i_{j_t}$,
				\item $i_{j_s}=i_{p_m}=i_{j_t}$,
				\item $i_{j_s}=i_{p_m}$ or $i_{j_t}=i_{p_m}$ but not both.
			\end{enumerate}
			
			In the first case, the arrow will not be in $D_{\bar{\mu}_m}$ neither in $D_{\iota_m\circ\rho_{m-1}}$.
			
			In the second case $\bar{\mu}_m(\alpha_{j_s,j_t,m-1})=\alpha_{j_s,j_t,m}$ and $\rho_m(\alpha_{j_s,j_t,m})=\beta_{j_s^{\alpha(j_s,m)+},j_t^{\alpha(j_t,m)+},m+1}$. On the other hand we have $\rho_{m-1}(\alpha_{j_s,j_t,m-1})=\beta_{j_s^{\alpha(j_s,m-1)+},j_t^{\alpha(j_t,m-1)+},m}$ and
			\[
				\iota_m(\beta_{j_s^{\alpha(j_s,m-1)+},j_t^{\alpha(j_t,m-1)+},m})=\beta_{j_s^{\alpha(j_s,m-1)+},j_t^{\alpha(j_t,m-1)+},m+1}.
			\]
			 As $\alpha(j_t,m-1)=\alpha(j_t,m)$ and $\alpha(j_s,m-1)=\alpha(j_s,m)$ we have the equality.
			
			In the third case, we have $\bar{\mu}_m(\alpha_{j_s,j_t,m-1})=\alpha_{j_s,j_t,m}$ and then $\rho_m(\alpha_{j_s,j_t,m})=\beta_{j_s^{\alpha(j_s,m)+},j_t^{\alpha(j_t,m)+},m+1}$. On the other hand, we have $\rho_{m-1}(\alpha_{j_s,j_t,m-1})=\beta_{j_s^{\alpha(j_s,m-1)+},j_t^{\alpha(j_t,m-1)+},m}$ and 
			\[\iota_m(\beta_{j_s^{\alpha(j_s,m-1)+},j_t^{\alpha(j_t,m-1)+},m})=\beta_{j_s^{(\alpha(j_s,m-1)+1)+},j_t^{(\alpha(j_t,m-1)+1)+},m}\] and, as $\alpha(j_s,m)=\alpha(j_t,m)=\alpha(j_s,m-1)+1=\alpha(j_t,m-1)+1$ we have the equality.
			
			The last case is symmetrical, and we will suppose that $i_{j_s}=i_{p_m}$. Thus $\bar{\mu}_m(\alpha_{j_s,j_t,m-1})=\alpha_{j_s^-,j_t,m}$ and we have $\rho_m(\alpha_{j_s^-,j_t,m})=\beta_{j_s^{(\alpha(j_s,m)-1)+},j_t^{\alpha(j_t,m)+},m+1}$. On the other side we have $\rho_{m-1}(\alpha_{j_s,j_t,m-1})=\beta_{j_s^{\alpha(j_s,m-1)+},j_t^{\alpha(j_t,m-1)+},m}$ and $\iota_m(\beta_{j_s^{\alpha(j_s,m-1)+},j_t^{\alpha(j_t,m-1)+},m})=\beta_{j_s^{\alpha(j_s,m-1)+},j_t^{\alpha(j_t,m-1)+},m+1}$. As $\alpha(j_s,m)=\alpha(j_s,m-1)+1$ and $\alpha(j_t,m)=\alpha(j_t,m-1)$ we have the equality.
			
			Finally in all cases we have the equality, and the diagram commutes. Thus for any vertex or arrow in $\widetilde{\Gamma}_{m-1}$ being in the domain, we have $\rho_m\circ\bar{\mu}_m=\iota_m\circ\rho_{m-1}$.

			As $\iota_{m}$ is clearly invertible for vertices of index $\geq p_{m+1}$ and for adjacent arrows, we can define $\rho_m^\ast=\bar{\mu}_m\circ\rho_{m-1}^\ast\circ\iota_m^{-1}$ on any vertex of $\Gamma_{W_{m+1}}$ of index $\geq p_{m+1}$.
			
			To end our proof we need to show that all arrows created between $\widetilde{\Gamma}_{m-1}$ and $\widetilde{\Gamma}_m$ (by tooth shift of final barbs) have an image through $\rho_m$. In that case we have an arrow $\alpha_{((p_m)_\Max)^{\alpha(p_m,m)-},j,m}$ where $i_j$ is adjacent to $i_{p_m}$. In addition we know that in $\widetilde{\Gamma}_{m-1}$ the last ordinary arrow between $i_j$ and $i_{p_m}$ was the arrow $\alpha_{j,k,m-1}$ where $i_k=i_{p_m}$ and $k\leq ((p_m)_\Max)^{\alpha(p_m,m)-}$. As $\rho_{m-1}$ is a graph morphism by hypothesis, there exists an arrow $\beta_{j^{\alpha(j,m-1)+},k^{\alpha(k,m-1)+},m}$ in $\Gamma_{W_m}$. As $\Gamma_{W_m}$ is a subquiver of $V_{\bar{w}}$, its arrows satisfies the index condition of Section \ref{sectionInitialSeedCw} and then:
			\begin{equation}
				j^{(\alpha(j,m-1)+1)+}\geq k^{(\alpha(k,m-1)+1)+}>j^{\alpha(j,m-1)+}>k^{\alpha(k,m-1)+}\label{eqFinalBarb}
			\end{equation}
			which can be rewritten, using the facts that $\alpha(j,m)=\alpha(j,m-1)$ and $\alpha(k,m)=\alpha(k,m-1)+1$:
			\begin{equation}
				j^{(\alpha(j,m)+1)+}\geq k^{\alpha(k,m)+}>j^{\alpha(j,m)+}>k^{\alpha(k,m-1)+}\label{eqFinalBarb2}
			\end{equation}
			
			To show that $\rho_m$ is a graph morphism, we must show that the image $\rho_m(\alpha_{((p_m)_\Max)^{\alpha(p_m,m)-},j,m})$ is an arrow of $\Gamma_{W_{m+1}}$. In that case the image arrow is:
			\[
				\beta_{(p_m)_\Max,j^{\alpha(j,m)+},m+1}.
			\]
			This arrow exists iff the following inequality is verified:
			\[
				\ell(w)+1\geq j^{(\alpha(j,m)+1)+}> (p_m)_\Max>j^{\alpha(j,m)+}
			\]
			
			As we have $((p_m)_\Max)^{\alpha(p_m,m)-}>k$, $k^{\alpha(k,m)+}>j^{\alpha(j,m)+}$ and $\alpha(k,m)=\alpha(p_m,m)$, we obtain 
			\[
				(p_m)_\Max>j^{\alpha(j,m)+}.
			\]
			We always have $\ell(w)+1\geq j^{(\alpha(j,m)+1)+}$ with equality iff $j^{\alpha(j,m)+}=j_\Max$. In that case it is obvious that $j^{(\alpha(j,m)+1)+}>(p_m)_\Max$.
			
			Suppose that we have $(p_m)_\Max>j^{(\alpha(j,m)+1)+}$. Then there exists $\delta>0$ such that:
			\[
				\ell(w)+1\geq j^{(\delta+\alpha(j,m)+1)+}>(p_m)_\Max>j^{(\delta+\alpha(j,m))+}
			\]
			and thus we have the arrow $\beta_{(p_m)_\Max,j^{(\delta+\alpha(j,m))+},m+1}\in\Gamma_{W_{m+1}}$ and then $\beta_{(p_m)_\Max,j^{(\delta+\alpha(j,m))+},m+1}\in\Gamma_{W_{m}}$. Then, by $\rho_{m-1}^\ast$ we have that $\alpha_{((p_m)_\Max)^{\alpha(k,m-1)-},j^{\delta+},m-1}\in \widetilde{\Gamma}_{m-1}$. That is impossible by hypothesis on the final barb. 
			
			Thus $j^{(\alpha(j,m)+1)+}>(p_m)_\Max$ and $\rho_m$ is a graph morphism.
			
		\end{proof}
		\begin{lemma}\label{lemmaPoint5}
			In the cut seed $(G_m,\widetilde{\Gamma}_m)$, $0\leq m<\ell(v)$ the line of color $i_{p_{m+1}}$ has a pure saw teeth structure. 
		\end{lemma}
		\begin{proof}
			We consider the quiver $\Gamma_{W_{m+1}}$. Thanks to Section \ref{sectionInitialization} we know it has a saw teeth structure and a pure saw teeth structure for the line $i_{p_{m+1}}$. Then we know that $\widetilde{\Gamma}_{m}=\rho^\ast_m(\Gamma_{W_{m+1}})$ and so we get $\widetilde{\Gamma}_m$ from $\Gamma_{W_{m+1}}$ by removing all vertices $G_{k,m}$ such that $\xi(k,m)>k\geq p_{m+1}$ and for $i_k=i_{p_m}$ this interval is empty. Thus we remove initial vertices of some line but not of $i_{p_m}$ and, using the same reasonning as Section \ref{sectionInitialization}, we have the result. 
		\end{proof}
	\subsection{Conclusion of the induction}
		Lemma \ref{lemmaPoint5} gives us point \ref{pointPureSawTeeth} of the induction hypothesis of Theorem \ref{theoremInduction} and thus the induction is proved.
		
		If we look at point \ref{pointDeltavcoordinates} of Theorem \ref{theoremInduction} when $m=\ell(v)$ we have that $f_\Min(k)^{\alpha(k,m)\oplus}=\ell(v)+1$ and then the interval is empty and in fact $(G_{\ell(v)},\widetilde{\Gamma}_{\ell(v)})$ is the empty seed.
		
		Then $\mu_\bullet(V_{\bar{w}})$ is the sum of all evicted summands and the quiver $\mu_\bullet(\Gamma_{\bar{w}})$ is the quiver obtained by applying the mutations to $\Gamma_{\bar{w}}$ and removing all deleted vertices. We will now show that $\mu_\bullet(V_{\bar{w}})$ is a $\C_{v,w}$-rigid maximal basic module and its Gabriel quiver is the quiver $\mu_\bullet(\Gamma_{\bar{w}})$.
	\subsection{Conclusion}
		
		\begin{proposition}\label{propositionRigidity}
				The module $\mu_\bullet(V_{\bar{w}})$ is rigid.
			\end{proposition}
			\begin{proof}
				$V_{\bar{w}}$ is rigid as it is the module of a seed for the cluster structure of $\C_w$. Moreover, rigidity is preserved through mutation. Then $\hat{\mu}_{\ell(v)}\circ\cdots\circ\hat{\mu}_1(V_{\bar{w}})$ is rigid. In addition, if the module $M$ is rigid, then any $N\subset M$ is also rigid. Then $\mu_{\bullet}(V_{\bar{w}})\subset \hat{\mu}_{\ell(v)}\circ\cdots\circ\hat{\mu}_1(V_{\bar{w}})$ is rigid.
			\end{proof}
		\begin{proposition}\label{propositionBasicity}
			The module $\mu_{\bullet}(V_{\bar{w}})$ is basic.
		\end{proposition}
			\begin{proof}
				$V_{\bar{w}}$ was basic. Mutations preserves basicity and so does removing direct summands thus $\mu_{\bullet}(V_{\bar{w}})$ is basic.
			\end{proof}
		\begin{proposition}\label{propositionCvwbelonging}
			The module $\mu_{\bullet}(V_{\bar{w}})$ is in the category $\C_{v,w}$.
		\end{proposition}
			\begin{proof}
				We will show first that $\mu_{\bullet}(V_{\bar{w}})\in \C_w$ and then that $\mu_{\bullet}(V_{\bar{w}})\in\C^v$. 
				
				$V_{\bar{w}}\in\C_w$. $\C_w$ is stable by extensions so $\hat{\mu}_{\ell(v)}\circ\cdots\circ\hat{\mu}_1\in\C_w$. In addition $\C_w$ is stable by taking direct summands and direct sums so $\mu_{\bullet}(V_{\bar{w}})\in\C_w$.
				
				To show that $\mu_{\bullet}(V_{\bar{w}})\in\C^v$ we use Proposition \ref{propositionCvbelonging}. We know that every direct summands of $\mu_{\bullet}(V_{\bar{w}})$ is an evicted module of the cut module $G_{\ell(v)}$ so, if $M$ is a direct summand of $\mu_{\bullet}(V_{\bar{w}})$, $\widetilde{\Delta}_{\dot{v}}(M)=0$ and, by Proposition \ref{propositionCvbelonging}, $M\in\C^v$. As $\C_v$ is stable by taking direct summands and direct sums, $\mu_{\bullet}(V_{\bar{w}})\in\C^v$. 
				
				Then $\mu_{\bullet}(V_{\bar{w}})\in\C_{v,w}$.
			\end{proof}
		\begin{proposition}\label{propositionMaximalRigidity}
			The module $\mu_{\bullet}(V_{\bar{w}})$ is $\C_{v,w}$-maximal rigid.
		\end{proposition}
		\begin{proof}
			We know from Theorem \ref{thmNumberOfIndecomposables} that a basic rigid module of $\C_{v,w}$ is maximal if it has $\ell(w)-\ell(v)$ direct summands. We know that $V_{\bar{w}}$ had $\ell(w)$ direct summands and that we removed all summands whose indices are in the set $\bigcup_{k\in I}\{1\leq j\leq \ell(w)\mid i_j=i_k\mid j>(j_\Max)^{\alpha(k,\ell(v))-}\}$ or equivalently, as seen with the cut seed: to remove one summand for each sequence of mutation $\hat{\mu}_m$. As we perform $\ell(v)$ such sequencess, we remove $\ell(v)$ summands and the final total number of summands is $\ell(w)-\ell(v)$ and so $\mu_{\bullet}(V_{\bar{w}})$ is $\C_{v,w}$-rigid maximal.
		\end{proof}
		\begin{proposition}\label{propositionFrozenVertices}
			The $\C_{v,w}$-projectives direct indecomposable summands of $\mu_{\bullet}(V_{\bar{w}})$ are the summands corresponding to vertices of $\Gamma_{R_{\ell(v)}}$ adjacent to summands deleted by $S$ and to the vertices of $S(\Gamma_{\ell(v)})$ adjacent to no other vertex of $S(\Gamma_{\ell(v)})$.
		\end{proposition}
		\begin{proof}
			A module which is not in any of the two cases is in $\C_{v,w}$ and only adjacent to direct summands of $\C_{v,w}$. So, as $\C_{v,w}$ is extension-closed, mutating this summands will still give a $\C_{v,w}$-module. If the summand $R_k$ to be mutated is adjacent to a summand of $\C_w\setminus\C_{v,w}$ then the mutated summand will be in $\C_{w}\setminus\C_{v,w}$ so it is non $\C_{v,w}$-mutable and so it has to be $\C_{v,w}$-projective.
			
			If the summand is adjacent to no other vertex, then it is non-mutable and thus $\C_{v,w}$-projective.
		\end{proof}
		\begin{proposition}
			The quiver $\mu_\bullet(\Gamma_{\bar{w}})$ obtained from $\Gamma_{\bar{w}}$ by mutations and vertex removal is the quiver of $\mu_\bullet(V_{\bar{w}})$
		\end{proposition}
		\begin{proof}
			By Proposition \ref{propositionFrozenVertices} the deleted vertices are adjacent to only $\C_{v,w}$-frozen vertices and then deleting these summands does not change the quiver of the neighbours of mutable vertices.
			
			The quiver obtained is thus the quiver of a $\C_{v,w}$-seed.
		\end{proof}
		
		This finishes the proof of Theorem \ref{theoremMain}.
\section{Corollaries}
	We now state some consequences of the existence of this algorithm. 
	\subsection{Mutation sequence equivalence}\label{sectionMutationSequenceEquivalence}
		In Sections \ref{sectionSchroer} \& \ref{sectionRecursive} we saw two different formulations for the sequence of mutations, respectively $\tilde{\mu}_m$ and $\hat{\mu}_m$ for $1\leq m\leq\ell(v)$. We now show that these two formulations are in fact equivalent.
		\begin{proposition}
			For any $1\leq m\leq\ell(v)$ we have $\tilde{\mu}_m=\hat{\mu}_m$.
		\end{proposition}
		\begin{proof}
			Using the results of Theorem \ref{theoremInduction}, we can describe the vertices/summands mutated by $\hat{\mu}_m$. We know that the summands $G_{k,m-1}$ of the cut seed having $\widetilde{\Delta}_{\dot{v},m}(G_{k,m-1})\neq 0$ are the one mutated by $\hat{\mu}_m$. It amounts to the summands of color $i_{p_m}$ with index $f_\Min(k)^{\alpha(k,m-1)\oplus}\leq m\leq f(k^{\alpha(k,m-1)+})$ and such that $k<(k_\Max)^{\alpha(k,m)-}$.
			
			The smallest index to be mutated is $k$ such that $m=f(k^{\alpha(k,m-1)+})=f_\Min(k)^{\alpha(k,m-1)\oplus}$ and the bigest is $(k_\Max)^{\alpha(k,m)-}$. 
			
			On the other side, following the $\tilde{\mu}_m$ sequence, we mutate the indices $((p_m)_\Min)^{\beta_m+},\dots,((p_m)_\Max)^{\gamma_m-}$.
			
			We recall that $\gamma_m=\alpha(p_m,m)$ so $\widetilde{\mu}_m$ mutates all indices $((p_m)_\Min)^{\beta_m +},\dots,((p_m)_\Max)^{\alpha(p_m,m)-}$.
			
			So we have to show that $k=(k_\Max)^{\alpha(k,m)-}\Leftrightarrow k=((p_m)_\Max)^{\alpha(p_m,m)-}$ and 
			\[
				m=f(k^{\alpha(k,m-1)+})\Leftrightarrow k=((p_m)_\Min)^{\beta_m+}.
			\]
			
			The first equivalence is direct as $(p_m)_\Max=k_\Max$ and $\alpha(k,m)=\alpha(p_m,m)$ as $i_k=i_{p_m}$.
			
			For the second equivalence, by definition of $f$, the smallest $k$ such that $f(k^{\alpha(k,m-1)+})=m$ verifies $k^{\alpha(k,m-1)+}=p_m$.
			
			We can now rewrite the equivalence as $p_m=k^{\alpha(k,m-1)+}\Leftrightarrow ((p_m)_\Min)^{\beta_m+}=k$.
			
			Then, applying $(\gamma_m-1)+=(\alpha(k,m)-1)+=\alpha(k,m-1)+$ to the right side we obtain:
			\[
				((p_m)_\Min)^{(\beta_m+\gamma_m-1)+}=k^{\alpha(k,m-1)+}
			\]
			
			However we have by definition $\gamma_m+\beta_m=\#\{1\leq k\leq p_m\mid i_k=i_{p_m}\}$ that is, the number of letters of color $i_{p_m}$ and of index $\leq p_m$ in $\bar{w}$. So, the $\beta_m+\gamma_m-1$-th successor of $(p_m)_\Min$ is $p_m$ and so we have $p_m=k^{\alpha(k,m-1)+}$ (see Section \ref{sectionCutSeed}). Thus we have proved the equivalence.
		\end{proof}
		
	\subsection{Red-green sequences}
		In this section we examine the mutation sequence according to the theory of Keller's red-green sequences. We recall some definitions based on  \cite{demonet2019SurveyMaximalGreen}.
		
		Given a quiver $Q$, we define the framed quiver $\widehat{Q}$ from $Q$ by adding a vertex $i'$ for any vertex $i\in Q_0$ and an arrow $i'\rightarrow i$. In this context the vertices $i'$ are called frozen.
		
		Let $\Gamma$ be a quiver obtained by $\widehat{Q}$ by a sequence of mutations. A vertex $i$ is said to be green with respect to $\widehat{Q}$ if it is the source of no arrow $i\rightarrow j'$ where $j$ is a frozen vertex. Conversely a vertex $i$ is said to be red if it is the target of no arrow $j'\rightarrow i$.
		
		\begin{remark}
			Here we use the opposite conventions of \cite{demonet2019SurveyMaximalGreen} as we are working with the quiver of $\End_{\Lambda}(V_{\bar{w}})$. 
		\end{remark}
		
		We have the following Theorem:
		\begin{theorem}[{\cite[Theorem 1.7.]{demonet2019SurveyMaximalGreen}}]
			With respect to any quiver of the mutation class of $\Gamma$, every non-frozen vertex of $\Gamma$ is either green or red.
		\end{theorem}
		
		A sequence of vertices $(k_1,\dots,k_N)$ of an initial quiver $Q$ is said to be green if for any $1\leq i\leq N$, the vertex $k_i$ is green with respect to $Q$ in the mutated quiver $\mu_{k_{i-1}}\circ\cdots\circ\mu_{k_1}(\widehat{Q})$.
		
		A sequence is maximal green with respect to $\Gamma$ if it is green with respect to $\Gamma$ and if all the non-frozen vertices of $\mu_{k_N}\circ\cdots\mu_{k_1}(\widehat{Q})$	are red with respect to $\Gamma$.
		
		A sequence is reddening with respect to $\Gamma$ if all non-frozen vertices of $\mu_{k_N}\circ\cdots\mu_{k_1}(\widehat{Q})$ are red with respect to $\Gamma$ (the sequence does not have to be green).
		
		We have the following conjecture and proposition:
		\begin{conjecture}
			The sequence of vertices of $\Gamma_{\bar{w}}$ mutated by $\mu_\bullet$ is a green sequence with respect to $\Gamma_{\bar{w}}$.
		\end{conjecture}
		\begin{proposition}
			The sequence of mutated vertices of $\Gamma_{\dot{w}}$ by $\mu_{\bullet}$ is a green sequence with respect to $\Gamma_{\dot{v}}$.
		\end{proposition}
		\begin{proof}
			The quiver $\Gamma_{\dot{w}}$ is in the mutation class of $\Gamma_{\dot{v}}$. Indeed, thanks to Matsumoto's lemma (see Section \ref{sectionDeltaVectorRepresentatives}) we have a sequence of $2$ and $3$-moves linking $\dot{v}$ and $\dot{w}$. Moreover $2$-moves translate in a reindexation of the vertices and $3$-moves as a mutation and a reindexation of vertices.
			
			Let $F_{\dot{v}}$ denote the functor $\Hom_\Lambda(V_{\dot{v}},-)$. According to the proof of the algorithm (and particularly the proof of Lemma \ref{lemmaComputationDeltavVector}), in any of our mutation to compute $\widetilde{\Delta}_{\dot{v}}$-vectors of mutated summands, we only consider short exact sequences where the summand to mutate is the kernel. Thanks to \cite[Proposition 12.4.]{geiss2011KacMoodyGroups}, there is an equivalence between the $F_{\dot{v}}$-exactness of a short exact sequence and a condition on dimension vectors of the summands which is equivalent to taking the sequence where $R_k$ is the kernel. According to \cite{keller2020CommunicationPrivee} this is equivalent to the fact that the mutation is green with respect to $V_{\dot{v}}$.
			
			Thus the mutations of $\mu_\bullet$ are green mutations with respect to $V_{\dot{v}}$. 
		\end{proof}
		\begin{remark}
			According to \cite[Exemple 5.2.7.]{menard2021AlgebresAmasseesAssociees} the mutation sequence of $\mu_\bullet$ can be neither maximal green nor reddening.
		\end{remark}
	\section{Aknowledgements}
		The author wants to thank Bernard Leclerc for his supervision, and Pierre-Guy Plamondon for fruitful conversations on implementing the computation, leading to study the $\Delta$-vectors. He also wants to thank Bernhard Keller for the discussion about red/green sequences and the rewieving of the PhD thesis from which these results are taken.
\newpage
\appendix
	\section{Configurations, inheritance and teeth shift principle}\label{sectionConfigurationsInheritageTeethShift}
		In this section, we translate \cite[Sections 4.10 \& 4.11]{menard2021AlgebresAmasseesAssociees}.
		
		Consider the seed $(R_m,\Gamma_m)$ with the notations of Theorem \ref{theoremInduction}, verifying in particular the points \ref{pointSawTeeth} and \ref{pointPureSawTeeth}.
		
		We will consider the configurations in the cut seed $(G,\widetilde{\Gamma})$ (with the notations of Lemma \ref{lemmaInduction}) around the vertex that will be mutated in any bicolor subquiver $(i_k,i_j)$. 
		
		Due to the definition of the sequence of mutations in Definition \ref{definitionHatMu}, and the fact we are working in the cut seed, the mutation sequence amounts to mutating all vertices of the line except the last which is deleted. There can be eviction, but, according to Corollary \ref{corollaryPredecessorEviction}, the vertex $G_{k^+}$ can only be evicted if $G_{k}$ has been evicted. Conversely, if $G_k$ is not evicted, no successor of $G_k$ can be evicted. Thus we will only study evictions in initial cases here.
		
		At any time we will take configurations of Tables \ref{tableConfigurationsInitiales} \& \ref{tableConfigurationsGénérales}, adapting if necessary the indexation of vertices. We will always suppose that $G_{k^{2+}}$ exists, on the contrary, $G_{k^+}$ is the last vertex of its line so it will not be mutated and it will be erased when considered.
		\subsection{Initial configurations}
			If $G_k$ is the first of its line (in $\widetilde{\Gamma}$) then, according to hypotheses \ref{pointSawTeeth} and \ref{pointPureSawTeeth} of Theorem \ref{theoremInduction}, the only possible configurations are the one listed in Table \ref{tableConfigurationsInitiales}. Due to point \ref{pointPureSawTeeth}, the first vertex of the line can have at most one ordinary arrow, being its target.
			
			\paragraph{Case $\alpha_0$ if eviction}
				If we are in case $\alpha_0$ before mutation, after its mutation we evict $G_k$ and now $G_{k^+}$ is the first vertex of the line in the new quiver $\widetilde{\Gamma}$.If $G_{k^+}$ had no ordinary arrow we are once again in a case $\alpha_0$. If there is one, it has $G_{k^+}$ as target, being $G_j\rightarrow G_{k^+}$ and then, we are in case $\alpha_1$ if there is no arrow $G_{k^{2+}}\rightarrow G_j$ or in case $\alpha_2$ if it exists. 
				
			\paragraph{Case $\alpha_0$ without eviction}
				If there is no eviction, we continue by now considering the vertex $G_{k^+}$. If $G_{k^+}$ had no ordinary arrow, it still does not and we are in case $\beta_0$. As previously if it has an ordinary arrow, it has only one $G_j\rightarrow G_{k^+}$, being its target. Then, depending on the existence of an arrow $G_{k^+}\rightarrow G_j$ we are in case $\beta_1$ or $\beta_2$.
				
			\paragraph{Case $\alpha_1$ with eviction}
				If $G_{k}^\ast$ is evicted after mutation, $G_k^+$ becomes the first vertex of the line. However, we have an arrow $G_j\rightarrow G_k$ coming from the mutation at $G_k^\ast$. It is the only ordinary arrow adjacent to $G_{k^+}$ because either the initial arrow $G_j\rightarrow G_k$ was the only ordinary between two lines (as a final barb) or it was the beginning of a tooth $G_{k^{\gamma +}}\rightarrow G_j\rightarrow G_k$ with $\gamma>1$ as the tooth is long, the left arrow being preserved in spite of the eviction. Then we have only one ordinary arrow pointing at $G_{k^+}$ and if we were on the beginning of tooth and $\gamma=2$ the we are now in case $\alpha_2$, otherwise in case $\alpha_1$.
				
			\paragraph{Case $\alpha_1$ without eviction}
				If we have now eviction, we still know there is only one ordinary arrow adjacent to $G_{k^+}$: $G_j\rightarrow G_{k^+}$ and there is one ordinary arrow $G_j\leftarrow G_k$. Thus, depending on the length of the tooth or, equivalently, the configuration around $G_{k^{2+}}$ we are either in case $\beta_3$ or $\beta_4$.
				
			\paragraph{Case $\alpha_2$ with eviction}
			
			If we have eviction $G_{k^+}$ becomes the first vertex of its line. In addition there is no arrow between $G_{k^+}$ and $G_j$ (the original $G_{k^+}\rightarrow G_j$ being deleted by the mutation at $G_k$). Then either there exists a $\gamma>0$ such that $G_{j^{\gamma+}}\rightarrow G_{k^+}$ exists and then we are now in case $\alpha_1$ or $\alpha_2$ depending on what happens at $G_{k^{2+}}$, or not and there was only one ordinary arrow, the final barb and now we are on case $\alpha_0$ then $\beta_0$ (when evictions stop) until the end of the line.
				
			\paragraph{Case $\alpha_2$ without eviction}
				Without eviction then $G_{k^+}$ is, after mutation of $G_k$, the source of two horizontal arrows and, potentially an arrow $G_{j^{\gamma+}}\rightarrow G_{k^+}$ where $\gamma>0$. The ordinary arrow $G_{j}\rightarrow G_k$ is no longer our concern so, up to renaming $j^{\gamma +}\leadsto j$ and $k^+\leadsto k$ we are in case $\beta_0$ if there was no ordinary arrow or $\beta_1$ or $\beta_2$ if there was one.
				
			Finally, for any of the 3 initial possible configurations, after mutation, we are in only 8 possible different configurations for $G_{k^+}$.				
			
			\begin{landscape}
					\begin{table}[h]
						\begin{center}
						\begin{tabular}{|c||c|c|c|}
							\hline Config.&$\alpha_0$&$\alpha_1$&$\alpha_2$\\
							\hline Nickname&Initial simple&Initial long tooth&Initial short tooth\\
							\hline
							$G$&{\begin{tikzcd}[column sep=tiny]\rien&G_{k^+}\ar[l,dashed,gray]&G_k\ar[l]\end{tikzcd}}&
							{\begin{tikzcd}[column sep=tiny,row sep=small]G_{j^+}&&G_j\ar[ll,dashed,gray]\ar[dr]\ar[dl,phantom,"\emptyset"]&\ar[l,dashed,gray]\\\rien&G_{k^+}\ar[l,dashed,gray]&&G_k\ar[ll]\end{tikzcd}}&
							{\begin{tikzcd}[column sep=tiny,row sep=small]G_{j^+}&&G_j\ar[ll,dashed,gray]\ar[dr]&\ar[l,dashed,gray]\\\rien&G_{k^+}\ar[l,dashed,gray]\ar[ur]&&G_k\ar[ll]\end{tikzcd}}\\
							\hline
							$\mu_k(G)$&
							{\begin{tikzcd}[column sep=tiny]\rien&G_{k^+}\ar[l,dashed,gray]\ar[r]&G_k^\ast\end{tikzcd}}&
							{\begin{tikzcd}[column sep=tiny,row sep=small]G_{j^+}&&G_j\ar[ll,dashed,gray]\ar[dl]&\ar[l,dashed,gray]\\\rien&G_{k^+}\ar[rr]\ar[l,dashed,gray]&&G_k^\ast\ar[ul]\end{tikzcd}}&
							{\begin{tikzcd}[column sep=tiny,row sep=small]G_{j^+}&&G_j\ar[ll,dashed,gray]&\ar[l,dashed,gray]\\\rien&G_{k^+}\ar[l,dashed,gray]\ar[ur,phantom,"\emptyset"]\ar[rr]&&G_k^\ast\ar[ul]\end{tikzcd}}\\
							\hline
							Next if no eviction&$\beta_0,\beta_1,\beta_2$&$\beta_3,\beta_4$&$\beta_0,\beta_1,\beta_2$\\
							Next if eviction&$\alpha_0,\alpha_1,\alpha_2$&$\alpha_1,\alpha_2$&$\alpha_0,\alpha_1,\alpha_2$\\
							\hline
						\end{tabular}
						\end{center}
						\caption{Initial possible configurations and evolutions}
						\label{tableConfigurationsInitiales}
					\end{table}
					
				\begin{table}[h]
					\begin{center}
					\begin{tabular}{|c||c|c|c|c|c|}
						\hline 
						Config.&$\beta_0$&$\beta_1$&$\beta_2$&$\beta_3$&$\beta_4$\\
						\hline
						Nickname&Start and end&Entrance long tooth&Entance short tooth&Middle (long) tooth&End tooth\\
						\hline
						$G$&\begin{tikzcd}[column sep=tiny]G_{k^+}&G_k\ar[l]\ar[r]&G_{k^-}^\ast\end{tikzcd}&
						\begin{tikzcd}[column sep=tiny,row sep =small]&G_j\ar[dl,phantom,"\emptyset"]\ar[dr,phantom,"\emptyset"]\ar[d]&\\G_{k^+}&G_k\ar[l]\ar[r]&G_{k^-}^\ast\end{tikzcd}&
						\begin{tikzcd}[column sep=tiny,row sep =small]&G_j\ar[d]\ar[dr,phantom,"\emptyset"]&\\G_{k^+}\ar[ur]&G_k\ar[l]\ar[r]&G_{k^-}^\ast\end{tikzcd}
						&\begin{tikzcd}[column sep=tiny,row sep =small]&G_j\ar[d]\ar[dl,phantom,"\emptyset"]&\\G_{k^+}&G_k\ar[l]\ar[r]&G_{k^-}^\ast\ar[ul]\end{tikzcd}&
						\begin{tikzcd}[column sep=tiny,row sep =small]&G_j\ar[d]&\\G_{k^+}\ar[ur]&G_k\ar[l]\ar[r]&G_{k^-}^\ast\ar[ul]\end{tikzcd}\\
						\hline 
						$\mu_k(G)$&\begin{tikzcd}[column sep=tiny,row sep =small]G_{k^+}\ar[r]&G_k^\ast&G_{k^-}^\ast\ar[l]\end{tikzcd}&
						\begin{tikzcd}[column sep=tiny,row sep =small]&G_j\ar[dl]\ar[dr]&\\G_{k^+}\ar[r]&G_k^\ast\ar[u]&G_{k^-}^\ast\ar[l]\end{tikzcd}&
						\begin{tikzcd}[column sep=tiny,row sep =small]&G_j\ar[dr]\ar[dl,phantom,"\emptyset"]&\\G_{k^+}\ar[r]&G_k^\ast\ar[u]&G_{k^-}^\ast\ar[l]\end{tikzcd}&
						\begin{tikzcd}[column sep=tiny,row sep =small]&G_j\ar[dl]\ar[dr,phantom,"\emptyset"]&\\G_{k^+}\ar[r]&G_k^\ast\ar[u]&G_{k^-}^\ast\ar[l]\end{tikzcd}&
						\begin{tikzcd}[column sep=tiny,row sep =small]&G_j\ar[dl,phantom,"\emptyset"]\ar[dr,phantom,"\emptyset"]&\\G_{k^+}\ar[r]&G_k^\ast\ar[u]&G_{k^-}^\ast\ar[l]\end{tikzcd}\\
						\hline
						Next&$\beta_0,\beta_1,\beta_2$&$\beta_3,\beta_4$&$\beta_0,\beta_1,\beta_2$&$\beta_3,\beta_4$&$\beta_0,\beta_1,\beta_2$\\
						\hline
					\end{tabular}
					\end{center}
					\caption{General configuration possible and evolutions}
					\label{tableConfigurationsGénérales}
				\end{table}
				\end{landscape}
				
				\begin{ex}
				We illustrate all this reasoning by Figure \ref{figureConfigurationsInitiales}.
					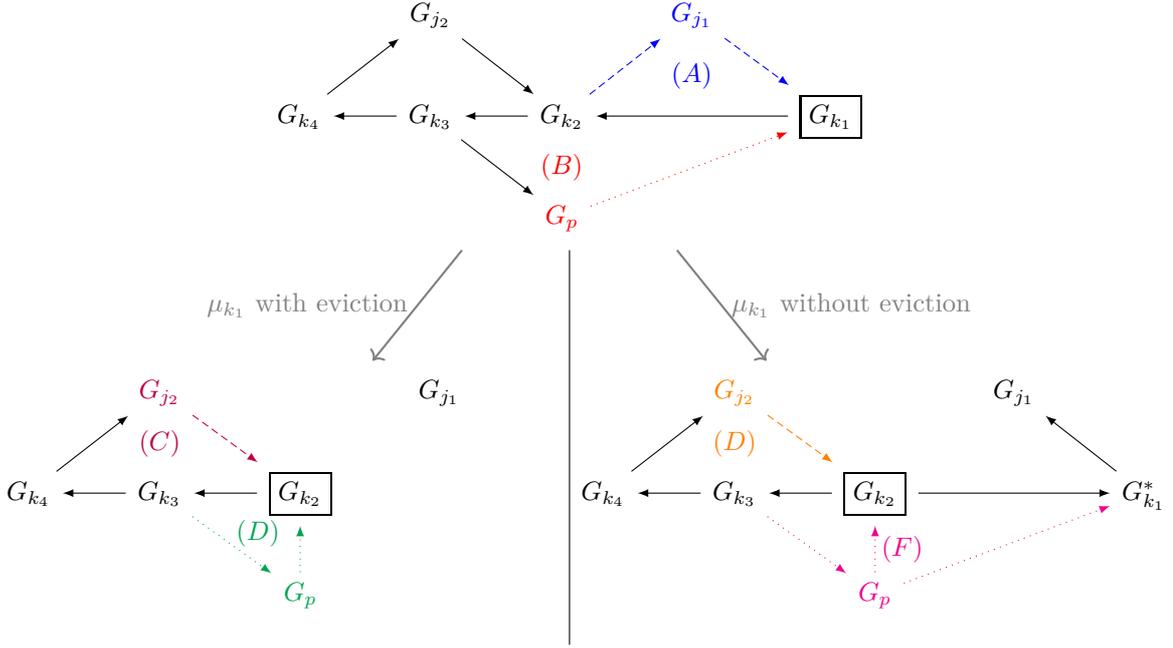
\begin{figure}[ht]
						\begin{center}
						\begin{tikzpicture}[scale=2]
							\coordinate (A) at (0,0);
							\coordinate (B) at (-2,-2.5);
							\coordinate (C) at (2,-2.5);
							\node (carquoisInitial) at (A) {
							$
							\begin{tikzcd}
							&G_{j_2}\ar[dr]&&\textcolor{blue}{G_{j_1}}\ar[dr,blue,dashed]&\\
							G_{k_4}\ar[ur]&G_{k_3}\ar[dr]\ar[l]&G_{k_2}\ar[l]\ar[blue,dashed,ur]\ar[red,d,phantom,"(B)"]&\rien\ar[blue,u,phantom,"(A)"]&\boxed{G_{k_1}}\ar[ll]\\
							&&\textcolor{red}{G_p}\ar[red,dotted,urr]&&
							\end{tikzcd}
							$};
							\node (carquoisEvince) at (B) {
							$
							\begin{tikzcd}
							&\textcolor{purple}{G_{j_2}}\ar[purple,dashed,dr]\ar[d,phantom,"(C)",purple]&&G_{j_1}&\\
							G_{k_4}\ar[ur]&G_{k_3}\ar[dr,Green,dotted]\ar[l]&\boxed{G_{k_2}}\ar[l]&&\\
							&&\textcolor{Green}{G_p}\ar[Green,dotted,u]\ar[uul,phantom,Green,"(D)" near start]&&
							\end{tikzcd}
							$
							};
							\node (carquoisMute) at (C) {
							$
							\begin{tikzcd}
							&\textcolor{orange}{G_{j_2}}\ar[dashed,orange,dr]&&G_{j_1}&\\
							G_{k_4}\ar[ur]&G_{k_3}\ar[dr,magenta,dotted]\ar[u,orange,phantom,"(D)"]\ar[l]&\boxed{G_{k_2}}\ar[l]\ar[magenta,d,phantom,"(F)" right]\ar[rr]&&G_{k_1}^\ast\ar[ul]\\
							&&\textcolor{magenta}{G_p}\ar[magenta,dotted,urr]\ar[dotted,magenta,u]&&
							\end{tikzcd}
							$
							};
							\draw[gray,->,thick] (carquoisInitial)--(carquoisEvince) node[midway, left]{$\mu_{k_1}$ with eviction};
							\draw[gray,->,thick] (carquoisInitial)--(carquoisMute) node[midway, right]{$\mu_{k_1}$ without eviction};
							\draw (carquoisInitial)--(0,-3.5);
						\end{tikzpicture}
						\end{center}
						\caption{Evolutions of initial configurations}
						\label{figureConfigurationsInitiales}
					\end{figure}
					
					Initially we have a quiver whose vertices $j_2$ and $j_1$ are of color $i_j$, $k_i$ of color $i_k$ and $p$ of color $i_p$.
					
					The bicolor subquiver $(i_k,i_j)$ is in configuration $\alpha_2$ for $G_{k_1}$ (arrows in blue dashes and horizontal arrows around (A)). The bicolor subquiver $(i_k,i_p)$ is in configuration $\alpha_1$ around $G_{k_1}$ (dotted red arrows around (B) and horizontal arrows).
					
					Then, after mutation a $G_{k_1}$, if it is evicted, we observe that the bicolor subquiver $(i_k,i_j)$ is now in configuration $\alpha_1$ around $G_{k_2}$ (violet dashes, letter (C)). On the other side, the bicolor subquiver $(i_k,i_p)$ is now in configuration $\alpha_2$ (green dots, (D)). We remark that if the initial teeth was $G_{k_4}\rightarrow G_p\rightarrow G_{k_2}$, (D)  would still be in configuration $\alpha_1$.
					
					If there is no eviction, then the bicolor subquiver $(i_k,i_j)$ configuration around $G_{k_2}$ is of type $\beta_1$ (orange dashes, (D)). If the teeth $G_{k_4}\rightarrow G_{j_2}\rightarrow G_{k_2}$ was shorter $(G_{k_3}\rightarrow G_{j_2}\rightarrow G_{k_2})$, we would have a configuration of type $\beta_2$.
					
					The configuration of the bicolor subquiver $(i_k,i_p)$ is of type $\beta_4$ (3 magenta dotted ordinary arrows, (F)). We remark, as before that if the tooth was longer and if we had originally $G_{k_4}\rightarrow G_p\rightarrow G_{k_1}$ we would be in configuration $\beta_3$.
			\end{ex}
			
	\subsection{General configurations}
		At this point we will now encounter at least 5 general configurations. We will show that under our hypothesis, there is no other configuration emerging. Recall that thanks to Corollary \ref{corollaryPredecessorEviction}, as here $G_{k^-}$ is not evicted, $G_k$ cannot be evicted either.
			
		\paragraph{Case $\beta_0$}
			After mutating $G_k$, $G_{k^+}$ can have at most one ordinary arrow (whom it is the target). If there is not such an arrow we are again in a case $\beta_0$, if there is one, we are in case $\beta_1$ or $\beta_2$.
			
		\paragraph{Case $\beta_1$}
			After mutation at $G_k$, there is a triangle $G_k^\ast\rightarrow G_j\rightarrow G_{k^+}$. Moreover, as in case $\alpha_1$, there cannnot be an arrow $G_{j^{\gamma+}}\rightarrow G_{k^+}$ for any $\gamma>0$. Thus we are in the case $\beta_3$ or $\beta_4$
		
		\paragraph{Case $\beta_2$}
			In that case, as in $\alpha_2$, we will potentially have an arrow $G_{j^{\gamma +}}\rightarrow G_{k^+}$ and, as there is no arrow $G_{k^+}\ \varnothing\  G_j$, we can rename the vertices. If the arrow $G_{j^{\gamma +}}\rightarrow G_{k^+}$ does not exist, we are in case $\beta_0$ and that was the last tooth with no final barb and we will only have $\beta_0$ cases until the end. If this arrow exists, we are now in case $\beta_1$ or $\beta_2$.
		
		\paragraph{Case $\beta_3$}
			As in $\beta_1$, we have only one ordinary arrow $G_{j}\rightarrow G_{k^+}$. Thus, depending on $G_{k^{2+}}$ we are in case $\beta_3$ or $\beta_4$
			
		\paragraph{Case $\beta_4$}
			As in $\beta_2$ there is no connection between $G_j$ and $G_{k^+}$ and, if there is an ordinary arrow it is $G_{j^{\gamma+}}\rightarrow G_{k^+}$. If there is no such arrow we are in case $\beta_0$ until the end, if there is one we are in case $\beta_1$ or $\beta_2$.
			
			\begin{ex}\label{exConfigurationsGénérales}
					We will illustrate this on Figure \ref{figureConfigurationsGénérales}.
					\begin{figure}[ht]
						\begin{center}
						\begin{tikzpicture}[scale=1.7]
							\coordinate (A) at (0,0);
							\coordinate (B) at (-2.5,-2.5);
							\coordinate (C) at (2.5,-2.5);
							\node (carquoisInitial) at (A) {
							$
							\begin{tikzcd}[column sep= 15pt]
							&G_{j_2}\ar[dr]&\rien&\textcolor{Green}{G_{j_1}}\ar[Green,dashed,d]&\\
							G_{k_4}\ar[dr]&G_{k_3}\ar[u]\ar[l]&G_{k_2}\ar[l]\ar[ur,dashed,Green]&\boxed{G_{k_1}}\ar[r]\ar[l]\ar[u,phantom,"(A)" right,Green]&G_{k_0}^\ast\ar[Green,ul,dashed]\\
							&\textcolor{red}{G_p}\ar[red,dotted,urr]\ar[phantom,ur,red,"(B)"]&&&
							\end{tikzcd}
							$};
							\node (carquoisEvince) at (B) {
							$
							\begin{tikzcd}[column sep= 15pt]
							&\textcolor{purple}{G_{j_2}}\ar[purple,dashed,dr]&\rien&G_{j_1}&\\
							G_{k_4}\ar[dr]&G_{k_3}\ar[dashed,purple,u]\ar[l]\ar[ur,phantom,purple,"(C)" near start]&\boxed{G_{k_2}}\ar[r]\ar[l]&G_{k_1}^\ast\ar[u]\ar[dotted,blue,dll]&G_{k_0}^\ast\ar[l]\\
							&\textcolor{blue}{G_p}\ar[urrr]\ar[blue,dotted,ur]\ar[u,phantom,blue,"(D)"]&&&
							\end{tikzcd}
							$
							};
							\node (carquoisMute) at (C) {
							$
							\begin{tikzcd}[column sep= 15pt]
							&G_{j_2}\ar[drr]&\rien&G_{j_1}&\\
							G_{k_4}\ar[dr,dotted,magenta]&\boxed{G_{k_3}}\ar[r]\ar[l]\ar[u,phantom,"(E)"]&G_{k_2}^\ast\ar[ul]\ar[dotted,magenta,dl]&G_{k_1}^\ast\ar[u]\ar[l]&G_{k_0}^\ast\ar[l]\\
							&G_p\ar[urrr]\ar[magenta,dotted,u]\ar[u,magenta,phantom,"(F)" right]&&&
							\end{tikzcd}
							$
							};
							\draw[gray,->,thick] (carquoisInitial)--(carquoisEvince) node[midway, left]{$\mu_{k_1}$};
							\draw[gray,->,thick] (carquoisEvince)--(carquoisMute) node[midway,above]{$\mu_{k_2}$};
						\end{tikzpicture}
						\end{center}
						\caption{Evolutions of general configurations}
						\label{figureConfigurationsGénérales}
					\end{figure}
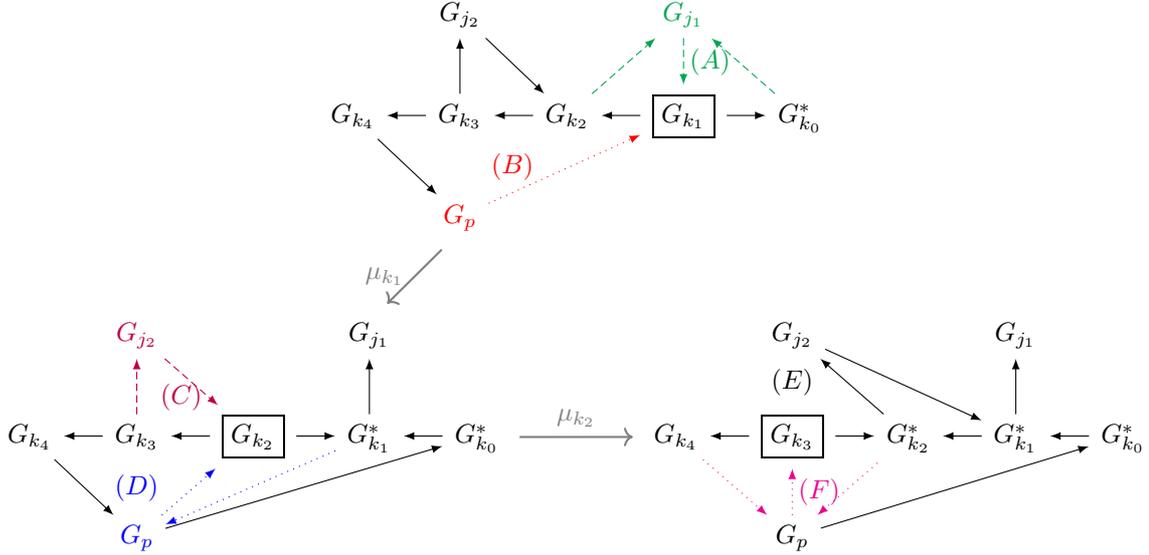
					
					Initially in the bicolor subquiver $(i_k,i_j)$ we have a configuration $\beta_4$ around $G_{k_1}$ (green dashes, (A)). In the bicolor subquiver  $(i_k,i_p)$ we have a configuration $\beta_1$ (red dots, (B)).
					
					After mutation $\mu_{k_1}$, in the subquiver $(i_k,i_j)$, we have around $G_{k_2}$ a configuration $\beta_2$ (violet dashed, (C)). In the subquiver$(i_k,i_p)$, we have a configuration $\beta_3$ (blue dots, (D)).
					
					After mutation $\mu_{k_2}$, in the subquiver $(i_k,i_j)$ we have a configuration $\beta_0$ around $G_{k_3}$ ((E), no ordinary arrow). In the subquiver $(i_k,i_p)$, we have a configuration $\beta_4$ once again (magenta dots).
					
					\end{ex}
		\subsection{Teeth shift}
			Using the previous knowledge of configuration, we will describe a more "macro" behaviour on the bicolor subquiver $(i_k,i_j)$ after performing the whole sequence of mutation $\hat{\mu}$.
			\begin{proposition}[Teeth shift]\label{propositionTeethShift}
				Suppose that a line $i_k$ of the subquiver $(i_k,i_j)$ is in pure saw teeth configuration. After mutating all the vertices of the line except the last one that we delete, we get a saw teeth structure. Precisely, any tooth $G_{k^{\alpha +}}\rightarrow G_j\rightarrow G_k$ becomes a tooth $G_{k^{(\alpha-1) +}}^\ast\rightarrow G_j\rightarrow G_{k^-}^\ast$ with the following derogation:
				\begin{itemize}
					\item if $k$ is the first of its line, the tooth $G_{k^{\alpha+}}\rightarrow G_j\rightarrow G_k$ becomes the arrow $G_{k^{(\alpha-1)+}}^\ast\rightarrow G_j$.
					\item If the last ordinary arrow is in the orientation $G_j\rightarrow G_k$ and that $k$ is not the last vertex of the line, then there appears a tooth $G_{k^{\alpha +}}^\ast\rightarrow G_j\rightarrow G_{k^-}^\ast$ (with the previous correction if $k$ is the first of its line) where $G_{k^{\alpha +}}^\ast$ is the last vertex of the line after deletion of the previous one.
				\end{itemize}
			\end{proposition}
			\begin{proof}
				We will rely on the previous study of all cases.
				
				We will start the reasoning at the first teeth non evicted.
				
				If the line starts with a sequence of vertices linked by horizontal arrows without any ordinary arrow (in the sense of Definition \ref{definitionSawTeethStructure}), we do the mutations until reaching the first tooth (there is no initial barb due to point \ref{pointPureSawTeeth} of Theorem \ref{theoremInduction}). For the sequence of initial vertices, the arrow $G_{k^+}\leftarrow G_k$ is reversed by mutation at $G_k$ but reversed back by mutation at $G_{k^+}$ so, globally there is no change on it.
				
				When reaching the right end of the teeth sequence, we have 4 possible cases. The vertex is either the initial one or not and the initial tooth has length one or more. We then have the configuration shown in Table \ref{tableConfigInitialeDécalage}.
				
				\begin{table}[ht]
					 		\begin{center}
					 			\begin{tabular}{|c|c|c|}
					 				\cline{2-3}
					 				\multicolumn{1}{c|}{}&Long tooth&Short tooth\\
					 				\hline
					 				Initial vertex&$\alpha_1$&$\alpha_2$\\
					 				\hline
					 				Non-initial vertex.&$\beta_1$&$\beta_2$\\
					 				\hline
					 			\end{tabular}
					 			\caption{Initial configurations}
					 			\label{tableConfigInitialeDécalage}
					 		\end{center}
					 	\end{table}
				
				We first consider short cases. If we are in case $\alpha_2$, after mutation, we have an only arrow $G_k^\ast\rightarrow G_j$ as remnant of the initial tooth as instructed by the first derogation. Then there is no arrow between $G_j$ and another vertex of the line of $G_k$ and thus we will be in a case $\beta$ for the next summuit, considering a new tooth of the sequence.
				
				If we are in case $\beta_2$, we directly have the tooth switch: $G_{k^+}\rightarrow G_j\rightarrow G_k$ becoming $G_{k}\rightarrow G_j\rightarrow G_{k^-}$. By the same reasoning we will then be in case $\beta_1$ or $\beta_2$ for the next vertex, right end of the next tooth.
				
				We now study the long cases. The case $\alpha_1$ will be almost the same reasoning as for the case $\beta_1$ provided we consider that the arrow $G_j\rightarrow G_k$ is deleted as in case $\alpha_2$. We will see this case through the $\beta_1$ case (but there $\alpha_1$ leads to the apparition of an inital barb).
				
				For the general case of a long tooth, we start with configuration $\beta_1$. The arrow $G_{k^{\gamma+}}\rightarrow G_j$ forming the left side of the tooth is not represented on the drawing, as it is too remote from $G_k$ ($\gamma>1$). Mutation of $G_k$ gives us the arrow $G_j\rightarrow G_{k^-}^\ast$ (provided we are not looking at the $\alpha_1$ case), the tilting of the ordinary arrow in $G_k^\ast\rightarrow G_j$ and the appearance of an arrow $G_j\rightarrow G_{k^+}$.  The first ordinary arrow targeting $G_{k^-}$ will now stay untouched by following mutations and will form the right end of the shifted tooth.
				
				We now consider the next vertex, being in case $\beta_3$ if $\gamma>2$ or $\beta_4$ if $\gamma=2$.
				
				If we are now in case $\beta_3$ around the new vertex $G_{k^+}$, by mutating it, the arrow $G_k\rightarrow G_j$ is deleted, the arrow $G_{j}\rightarrow G_{k^+}$ is tilted and there is a new arrow $G_j\rightarrow G_{k^{2+}}$. Then the following vertex will be in configuration $\beta_3$ if the left end of the initial tooth is not $G_{k^{3+}}$ and then we use the same reasoning as the current one. 
				
				On the contrary if we reach a configuration $\beta_4$ we will mutate $G_{k^{\gamma-1}+}$ and from the initially considered tooth, we have:
				\begin{itemize}
					\item the shifted arrow $G_j\rightarrow G_{k^-}^\ast$ (not existing if $G_k$ was the initial vertex of the line)
					\item the arrow $G_{k^{(\gamma-2)+}}\rightarrow G_j$
					\item the arrow $G_j\rightarrow G_{k^{(\gamma-1)+}}$
					\item the arrow $G_{k^{\gamma +}}\rightarrow G_j$ which is the initial left side of the tooth
				\end{itemize}
				
				After mutation the first arrow in unchanged, the second one is deleted, the third is tilted in, $G_{k^{(\gamma-1)+}}\rightarrow G_j$, the left side of the new tooth and the fourth arrow is deleted. In total the initial tooth $G_{k^{\gamma+}}\rightarrow G_j\rightarrow G_{k}$ has been shifted in $G_{k^{(\gamma -1)+}}\rightarrow G_j\rightarrow G_{k^-}$.
				
				After the tooth shift, the next vertex will be on configuration $\beta_1$ or $\beta_2$ (if there is another tooth or a final barb) or $\beta_0$ (if we are at the end of the tooth sequence). In that final case we use the same reasoning as the first $\beta_0$ until the end of the line.
				
				If we have a tooth we repeat the reasoning as before.
				
				If we have a final barb, then we will have a configuration $\beta_1$ followed by configurations $\beta_3$ but with no $\beta_4$ configuration to end the tooth. If $G_{k^{\zeta+}}$ is the last vertex of the line, the last mutated vertex will be $G_{k^{(\zeta-1)+}}$. After mutation  we will have an arrow $G_{k^{(\zeta-1)+}}\rightarrow G_j$ closing the end having the shift of the final barb as right side and an arrow $G_j\rightarrow G_{k^{\zeta+}}$ which will disappear immediatly with the deletion of the vertex $G_{k^{\zeta +}}$. The new appeared left side of the tooth will be a minor inconvenience as we will never more mutate on $G_{k^{(\zeta-1)+}}$, the new last vertex of the line.
			\end{proof}
			\begin{ex}
					 	We will take again Example \ref{exConfigurationsGénérales}
					 	
					 	We have initially the tooth $G_{k_3}\rightarrow G_{j_2} \rightarrow G_{k_2}$ shifted in $G_{k_2}^\ast\rightarrow G_{j_2}\rightarrow G_{k_1}^\ast$ after $\mu_{k_2}\circ\mu_{k_1}$, this is the non-initial short tooth case.
					 	
					 	Let us look at the configuration before the mutation $\mu_{k_0}$. It is:
					 	
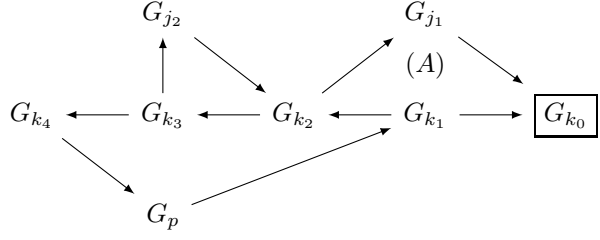
\[
					 		\begin{tikzcd}
							&G_{j_2}\ar[dr]&\rien&{G_{j_1}}\ar[dr]&\\
							G_{k_4}\ar[dr]&G_{k_3}\ar[u]\ar[l]&G_{k_2}\ar[l]\ar[ur]&G_{k_1}\ar[r]\ar[l]\ar[u,phantom,"(A)"]&\boxed{G_{k_0}}\\
							&{G_p}\ar[urr]&&&
							\end{tikzcd}
					 	\]
					 	
					 	Then we see a tooth (A) $G_{k_2}\rightarrow G_{j_1}\rightarrow G_{k_0}$ that, after $\mu_{k_1}\circ\mu_{k_0}$ reduces to the initial barb $G_{k_1}^\ast\rightarrow G{j_1}$. That is the case of an initial long teeth.
					 	
					 	To illustrate the second part of Proposition \ref{propositionTeethShift}, we consider the quiver:
					 	
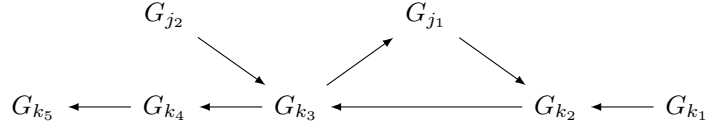
\[
					 		\begin{tikzcd}
					 			&G_{j_2}\ar[dr]&&G_{j_1}\ar[dr]&&\\
					 			G_{k_5}&G_{k_4}\ar[l]&G_{k_3}\ar[ur]\ar[l]&&G_{k_2}\ar[ll]&G_{k_1}\ar[l]
					 		\end{tikzcd}.
					 	\]
					 
					 In that case, after $\mu_{k_4}\circ\cdots\circ\mu_{k_1}$ we have:
						
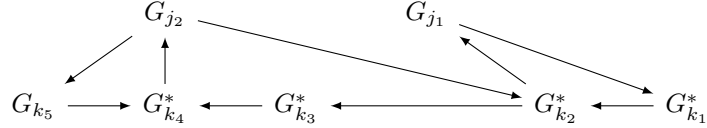
\[
					 		\begin{tikzcd}
					 			&G_{j_2}\ar[drrr]\ar[dl]&&G_{j_1}\ar[drr]&&\\
					 			G_{k_5}\ar[r]&G_{k_4}^\ast\ar[u]&G_{k_3}^\ast\ar[l]&&G_{k_2}^\ast\ar[ll]\ar[ul]&G_{k_1}^\ast\ar[l]
					 		\end{tikzcd},
					 	\]
					 	which becomes, after deletion of $G_{k_5}$:
					 	
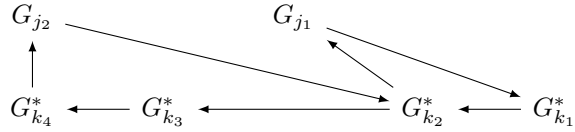
\[
					 		\begin{tikzcd}
					 			&G_{j_2}\ar[drrr]&&G_{j_1}\ar[drr]&&\\
					 			&G_{k_4}^\ast\ar[u]&G_{k_3}^\ast\ar[l]&&G_{k_2}^\ast\ar[ll]\ar[ul]&G_{k_1}^\ast\ar[l]
					 		\end{tikzcd}.
					 	\]
					 	
					 	We have here the appearance of a tooth $G_{k_4}^\ast\rightarrow G_{j_2}\rightarrow G_{k_2}^\ast$ when we initially only have a final barb $G_{j_2}\rightarrow G_{k_2}$.
					 	
					 	Nevertheless, we remark that it does not perturbate the saw teeth structure neither upcoming mutations, the vertex $R_{k_4}$ will not be mutated anymore and then the arrow will not be taken into account in the computation of $\Delta$-vectors in upcoming mutations.
		\end{ex}
			
	\section{Examples}
			\begin{ex}\label{exampleAlgo2}
				We take again the data of Example \ref{exempleAlgo1} p.\pageref{exempleAlgo1}. Recall that we are in the case $W$ of type $A_5$ with $\bar{w}=[1,3,2,4,3,\underline{2},\underline{4},\underline{5},4,\underline{3},2,\underline{1},\underline{2}]$ and $v=s_2s_4s_5s_3s_1s_2$. We represent the initial seed on Figure \ref{figureModuleVbarW} with the quiver $\Gamma_{\bar{w}}$ labelled by socle decomposition of the direct summands of $V_{\bar{w}}$. On Example \ref{exempleAlgo1} we saw that we had the mutation sequences:
				
				\[
					\tilde{\mu}_1=\mu_8\circ\mu_3\circ\mu_1,\quad\tilde{\mu}_2=\mu_2,\quad\tilde{\mu}_3=\mu_9\circ\mu_4,\quad \tilde{\mu}_4=\text{id},\quad \tilde{\mu}_5=\mu_7,\quad \tilde{\mu}_6=\mu_3
				\]
				
				We recall on Table \ref{tableSousQuotientsExemple} the modules defining the $\Delta_{\dot{v}}$-vectors. This gives the $\Delta$-vectors of the Table \ref{tableDeltaVecteurs1}. On this table, the color of the canonical vector $f_m$ relates to the color $i_{p_m}$. The indices $m>6$ are not colored as their refer to an arbitrary completion of $\bar{v}$ in $\dot{v}$.
				
				In order to define the first sequence of mutation $\hat{\mu}_1(V_{\bar{w}})$ we determine the set:
			
			\[
				A_1(V_{\bar{w}})=\{1\leq j\leq 8\mid \Delta_{\dot{v},1}(V_{\bar{w}})\neq 0\}=\{1,3,8\}
			\]
			We then have $\hat{\mu}_1(V_{\bar{w}})=\mu_8\circ\mu_3\circ\mu_1$ that we apply. Note that, as expected, we have $\hat{\mu}_1=\widetilde{\mu}_1$ of Example \ref{exempleAlgo1}.

			\begin{figure}[h!t]
				\begin{center}\resizebox{!}{7cm}{
				\begin{tikzpicture}[>=stealth, yscale=0.9]
					\node[blue] (V1) at (-1,-4) {$V_1=2$};
					\node[red] (V2) at (-2,-2) {$V_2=\xymatrix@=-5pt{&2\\1&}$};
					\node[blue] (V3) at (-3,-4) {$V_3=\xymatrix@=-5pt{1&\\&2}$};
					\node[Green] (V4) at (-4,-6) {$V_4=\xymatrix@=-5pt{1&&\\&2&\\&&3}$};
					\node[Orange] (V5) at (-4.5,-8) {$V_5=\xymatrix@=-5pt{1&&&\\&2&&\\&&3&\\&&&4}$};
					\node[violet] (V6) at (-6,-10) {$V_6=\xymatrix@=-5pt{1&&&&\\&2&&&\\&&3&&\\&&&4&\\&&&&5}$};
					\node[Orange] (V7) at (-7,-8) {$V_7=\xymatrix@=-5pt{1&&&&\\&2&&&\\&&3&&5\\&&&4&}$};
					\node[blue] (V8) at (-8,-4) {$V_8=\xymatrix@=-5pt{&2&\\1&&3\\&2&}$};
					\node[Green] (V9) at (-9,-6) {$V_9=\xymatrix@=-5pt{&2&&&\\1&&3&&5\\&2&&4&\\&&3&&}$};
					\node[Orange] (V10) at (-10,-8) {$V_{10}=\xymatrix@=-5pt{&2&&&\\1&&3&&\\&2&&4&\\&&3&&5\\&&&4&}$};
					\node[blue] (V11) at (-11,-4) {$V_{11}=\xymatrix@=-5pt{&&&&5\\&2&&4&\\1&&3&&\\&2&&&}$};
					\node[Green] (V12) at (-12,-6) {$V_{12}=\xymatrix@=-5pt{&2&&4&\\1&&3&&5\\&2&&4&\\&&3&&}$};
					\node[red] (V13) at (-13,-2) {$V_{13}=\xymatrix@=-5pt{&&&&5\\&&&4&\\&&3&&\\&2&&&\\1&&&&}$};

					\draw (V1);
					\draw (V2);
					\draw (V3);
					\draw (V4);
					\draw (V5);
					\draw (V6);
					\draw (V7);
					\draw (V8);
					\draw (V9);
					\draw (V10);
					\draw (V11);
					\draw (V12);
					\draw (V13);
					
					\draw[->] (V1)--(V3);
					\draw[->] (V3)--(V8);
					\draw[->] (V8)--(V11);
					
					\draw[->] (V2)--(V13);
					
					\draw[->] (V4)--(V9);
					\draw[->] (V9)--(V12);
					
					\draw[->] (V5)--(V7);
					\draw[->] (V7)--(V10);
					
					\draw[->] (V2)--(V1);
					\draw[->] (V11)--(V2);
					\draw[->] (V4)--(V3);
					\draw[->] (V7)--(V4);
					\draw[->] (V8)--(V4);
					\draw[->] (V6)--(V5);
					\draw[->] (V10)--(V6);
					\draw[->] (V9)--(V7);
					\draw[->] (V9)--(V8);
					\draw[->] (V10)--(V9);
					\draw[->] (V11)--(V9);
					\draw[->] (V12)--(V10);
					\draw[->] (V12)--(V11);
					\draw[->] (V13)--(V11);
					
				\end{tikzpicture}}
				\end{center}
				\caption{Seed $(V_{\bar{w}},\Gamma_{\bar{w}})$}
				\label{figureModuleVbarW}
			\end{figure}
			
			\begin{table}[h!t]
				\[
					\begin{array}{|c|c|c|c|c|c|}
						\hline
						k&1&2&3&4&5\\
						\hline
						M_{k,\dot{v}}&\xymatrix@=-5pt{2}&\xymatrix@=-5pt{&2\\1&}&\xymatrix@=-5pt{2&\\&3}&\xymatrix@=-5pt{5}&\xymatrix@=-5pt{2&&&\\&3&&5\\&&4&}\\
						\hline
						k&6&7&8&9&10\\
						\hline
						M_{k,\dot{v}}&\xymatrix@=-5pt{&2&\\1&&3}&\xymatrix@=-5pt{3}&\xymatrix@=-5pt{&2&&&\\1&&3&&5\\&&&4&}&\xymatrix@=-5pt{3&&5\\&4&}&\xymatrix@=-5pt{&5\\4&}\\
						\hline
						k&11&12&13&14&15\\
						\hline
						M_{k,\dot{v}}&\xymatrix@=-5pt{1}&\xymatrix@=-5pt{&2&&\\&&3&\\&&&4}&\xymatrix@=-5pt{&2&&\\1&&3&\\&&&4}&\xymatrix@=-5pt{3&\\&4}&\xymatrix@=-5pt{4}\\
						\hline
					\end{array}
				\]
				\caption{Strata defining the $\Delta_{\dot{v}}$-stratification}
				\label{tableSousQuotientsExemple}
			\end{table}
			
			\begin{table}[h!t]
				\[
					\begin{array}{|c|c|c|c|c|c|c|c|}
						\hline
						k&1&2&3&4&5&6&7\\
						\hline
						\Delta_{\dot{v}}(V_k)&\textcolor{blue}{f_1}&\textcolor{red}{f_2}&\textcolor{blue}{f_1}+f_{11}&\textcolor{Green}{f_3}+f_{11}&f_{12}&\textcolor{violet}{f_4}+f_{12}&\textcolor{orange}{f_5}+f_{11}\\
						\hline
						k&8&9&10&11&12&13&\\
						\hline
						\Delta_{\dot{v}}(V_k)&\textcolor{blue}{f_1}+\textcolor{blue}{f_6}&\textcolor{Green}{f_3}+f_8&\textcolor{orange}{f_5}+f_{11}+f_{13}&\textcolor{blue}{f_1}+\textcolor{blue}{f_6}+f_{10}&\textcolor{Green}{f_3}+f_8+f_{15}&\textcolor{red}{f_2}+f_7+f_{10}&\\
						\hline
					\end{array}
				\]
				\caption{$\Delta_{\dot{v}}$-vectors of $V_{\bar{w}}$ direct indecomposable summands}
				\label{tableDeltaVecteurs1}
			\end{table}
			
			We apply Proposition \ref{propositionMutationDeltaVectors} to compute the $\Delta_{\dot{v}}$-vector of $V_1^\ast$. Here we do not do the auxilliary computation with $d_\Delta$ as, thanks to the proof of Lemma \ref{lemmaComputationDeltavVector}, only one of the two computations gives a vector of $\N^{\ell(w_0)}$. 
			
			We either have
			\[
				\Delta_{\dot{v}}(V_1^\ast)\overset{?}{=}\Delta_{\dot{v}}(V_3)-\Delta_{\dot{v}}(V_1)=f_1+f_{11}-f_1=f_{11}
			\]
			or
			\[
				\Delta_{\dot{v}}(V_1^\ast)\overset{?}{=}\Delta_{\dot{v}}(V_2)-\Delta_{\dot{v}}(V_1)=f_2-f_1.
			\]
			Only the first computation is then possible and $\Delta_{\dot{v}}(V_1^\ast)=f_1$.
			
			Note that these computations echoes the short exact sequences used to determine the mutated modules:
			\begin{equation}\label{eq1V1}
				2\rightarrow \xymatrix@=-5pt{1&\\&2}\rightarrow 1,\quad
				1\rightarrow \xymatrix@=-5pt{&2\\1&}\rightarrow 2.
			\end{equation}
			and then the mutated module $V_1^\ast$ is the simple module $S_1$. But only one short exact sequence translated into a valid computation regarding $\Delta_{\dot{v}}$-vectors: the first one.

			After this mutation, we have the module of Figure \ref{figureMutation1} (modified arrows are in thick magenta).
			
			\begin{figure}[ht]
				\begin{center}\resizebox{!}{7cm}{
				\begin{tikzpicture}[>=stealth, yscale=0.85]
					\node[blue] (V1) at (-1,-4) {$V_1^\ast=1$};
					\node[red] (V2) at (-2,-2) {$V_2=\xymatrix@=-5pt{&2\\1&}$};
					\node[blue] (V3) at (-3,-4) {$V_3=\xymatrix@=-5pt{1&\\&2}$};
					\node[Green] (V4) at (-4,-6) {$V_4=\xymatrix@=-5pt{1&&\\&2&\\&&3}$};
					\node[Orange] (V5) at (-4.5,-8) {$V_5=\xymatrix@=-5pt{1&&&\\&2&&\\&&3&\\&&&4}$};
					\node[violet] (V6) at (-6,-10) {$V_6=\xymatrix@=-5pt{1&&&&\\&2&&&\\&&3&&\\&&&4&\\&&&&5}$};
					\node[Orange] (V7) at (-7,-8) {$V_7=\xymatrix@=-5pt{1&&&&\\&2&&&\\&&3&&5\\&&&4&}$};
					\node[blue] (V8) at (-8,-4) {$V_8=\xymatrix@=-5pt{&2&\\1&&3\\&2&}$};
					\node[Green] (V9) at (-9,-6) {$V_9=\xymatrix@=-5pt{&2&&&\\1&&3&&5\\&2&&4&\\&&3&&}$};
					\node[Orange] (V10) at (-10,-8) {$V_{10}=\xymatrix@=-5pt{&2&&&\\1&&3&&\\&2&&4&\\&&3&&5\\&&&4&}$};
					\node[blue] (V11) at (-11,-4) {$V_{11}=\xymatrix@=-5pt{&&&&5\\&2&&4&\\1&&3&&\\&2&&&}$};
					\node[Green] (V12) at (-12,-6) {$V_{12}=\xymatrix@=-5pt{&2&&4&\\1&&3&&5\\&2&&4&\\&&3&&}$};
					\node[red] (V13) at (-13,-2) {$V_{13}=\xymatrix@=-5pt{&&&&5\\&&&4&\\&&3&&\\&2&&&\\1&&&&}$};

					\draw (V1);
					\draw (V2);
					\draw (V3);
					\draw (V4);
					\draw (V5);
					\draw (V6);
					\draw (V7);
					\draw (V8);
					\draw (V9);
					\draw (V10);
					\draw (V11);
					\draw (V12);
					\draw (V13);
					
					\draw[->] (V3)--(V8);
					\draw[->] (V8)--(V11);
					
					\draw[->] (V2)--(V13);
					
					\draw[->] (V4)--(V9);
					\draw[->] (V9)--(V12);
					
					\draw[->] (V5)--(V7);
					\draw[->] (V7)--(V10);
					
					\draw[->] (V11)--(V2);
					\draw[->] (V4)--(V3);
					\draw[->] (V7)--(V4);
					\draw[->] (V8)--(V4);
					\draw[->] (V6)--(V5);
					\draw[->] (V10)--(V6);
					\draw[->] (V9)--(V7);
					\draw[->] (V9)--(V8);
					\draw[->] (V10)--(V9);
					\draw[->] (V11)--(V9);
					\draw[->] (V12)--(V10);
					\draw[->] (V12)--(V11);
					\draw[->] (V13)--(V11);
					
					\draw[magenta,very thick,->] (V3)--(V1);
					\draw[magenta,very thick,->] (V2)--(V3);
					\draw[magenta,very thick,->] (V1)--(V2);
					
				\end{tikzpicture}}
				\end{center}
				\caption{Module $\mu_1(V_{\bar{w}})$}
				\label{figureMutation1}
			\end{figure}
			
			We now compute $\Delta_{\dot{v}}(V_3^\ast)$ by the same method.
			
			\[
				\Delta_{\dot{v}}(V_3^\ast)=\Delta_{\dot{v}}(V_1^\ast)+\Delta_{\dot{v}}(V_8)-\Delta_{\dot{v}}(V_3)=f_{11}+f_1+f_6-f_1-f_{11}=f_6
			\]
			
			We have $V_3^\ast=\xymatrix@=-5pt{&2&\\1&&3}$  and the module $\mu_3\circ\mu_1(V_{\bar{w}})$ is the one of Figure \ref{figureMutation2}.
		
			\begin{figure}[h!t]
				\begin{center}\resizebox{!}{7cm}{
				\begin{tikzpicture}[>=stealth, yscale=0.9]
					\node[blue] (V1) at (-1,-4) {$V_1^\ast=1$};
					\node[red] (V2) at (-2,-2) {$V_2=\xymatrix@=-5pt{&2\\1&}$};
					\node[blue] (V3) at (-3,-4) {$V_3^\ast=\xymatrix@=-5pt{&2&\\1&&3}$};
					\node[Green] (V4) at (-4,-6) {$V_4=\xymatrix@=-5pt{1&&\\&2&\\&&3}$};
					\node[Orange] (V5) at (-4.5,-8) {$V_5=\xymatrix@=-5pt{1&&&\\&2&&\\&&3&\\&&&4}$};
					\node[violet] (V6) at (-6,-10) {$V_6=\xymatrix@=-5pt{1&&&&\\&2&&&\\&&3&&\\&&&4&\\&&&&5}$};
					\node[Orange] (V7) at (-7,-8) {$V_7=\xymatrix@=-5pt{1&&&&\\&2&&&\\&&3&&5\\&&&4&}$};
					\node[blue] (V8) at (-8,-4) {$V_8=\xymatrix@=-5pt{&2&\\1&&3\\&2&}$};
					\node[Green] (V9) at (-9,-6) {$V_9=\xymatrix@=-5pt{&2&&&\\1&&3&&5\\&2&&4&\\&&3&&}$};
					\node[Orange] (V10) at (-10,-8) {$V_{10}=\xymatrix@=-5pt{&2&&&\\1&&3&&\\&2&&4&\\&&3&&5\\&&&4&}$};
					\node[blue] (V11) at (-11,-4) {$V_{11}=\xymatrix@=-5pt{&&&&5\\&2&&4&\\1&&3&&\\&2&&&}$};
					\node[Green] (V12) at (-12,-6) {$V_{12}=\xymatrix@=-5pt{&2&&4&\\1&&3&&5\\&2&&4&\\&&3&&}$};
					\node[red] (V13) at (-13,-2) {$V_{13}=\xymatrix@=-5pt{&&&&5\\&&&4&\\&&3&&\\&2&&&\\1&&&&}$};

					\draw (V1);
					\draw (V2);
					\draw (V3);
					\draw (V4);
					\draw (V5);
					\draw (V6);
					\draw (V7);
					\draw (V8);
					\draw (V9);
					\draw (V10);
					\draw (V11);
					\draw (V12);
					\draw (V13);
					
					\draw[->] (V8)--(V11);
					
					\draw[->] (V2)--(V13);
					
					\draw[->] (V4)--(V9);
					\draw[->] (V9)--(V12);
					
					\draw[->] (V5)--(V7);
					\draw[->] (V7)--(V10);
					
					\draw[->] (V11)--(V2);
					\draw[->] (V7)--(V4);
					\draw[->] (V6)--(V5);
					\draw[->] (V10)--(V6);
					\draw[->] (V9)--(V7);
					\draw[->] (V9)--(V8);
					\draw[->] (V10)--(V9);
					\draw[->] (V11)--(V9);
					\draw[->] (V12)--(V10);
					\draw[->] (V12)--(V11);
					\draw[->] (V13)--(V11);

					\draw[magenta,very thick,->] (V1)--(V3);
					\draw[magenta,very thick,->] (V4)--(V1);
					\draw[magenta,very thick,->] (V3)--(V4);
					\draw[magenta,very thick,->] (V8)--(V3);
					\draw[magenta,very thick,->] (V2)--(V8);
					\draw[magenta,very thick,->] (V3)--(V2);
					
				\end{tikzpicture}}
				\end{center}
				\caption{Module $\mu_3\circ\mu_1(V_{\bar{w}})$}
				\label{figureMutation2}
			\end{figure}
			
			And finally we get $\Delta_{\dot{v}}(V_8^\ast)$ by the computation:
			\[
				\Delta_{\dot{v}}(V_8^\ast)=\Delta_{\dot{v}}(V_3^\ast)+\Delta_{\dot{v}}(V_{11})-\Delta_{\dot{v}}(V_8)=f_6+f_1+f_6+f_{10}-f_1-f_6=f_6+f_{10}.
			\]
			
			We then have as $\Delta_{\dot{v}}$-vectors for $\hat{\mu}_1(V_{\bar{w}})$ Table \ref{tableDeltaVecteurs2} and the module is on Figure \ref{figureMutation3}.
			
			\begin{table}[h!t]
				\[
					\begin{array}{|c|c|c|c|c|c|c|c|}
						\hline
						k&1&2&3&4&5&6&7\\
						\hline
						\Delta_{\dot{v}}(R_{k,1})&f_{11}&\textcolor{red}{f_2}&\textcolor{blue}{f_6}&\textcolor{Green}{f_3}+f_{11}&f_{12}&\textcolor{violet}{f_4}+f_{12}&\textcolor{orange}{f_5}+f_{11}\\
						\hline
						k&8&9&10&11&12&13&\\
						\hline
						\Delta_{\dot{v}}(R_{k,1})&\textcolor{blue}{f_6}+f_{10}&\textcolor{Green}{f_3}+f_8&\textcolor{orange}{f_5}+f_{11}+f_{13}&\textcolor{blue}{f_1}+\textcolor{blue}{f_6}+f_{10}&\textcolor{Green}{f_3}+f_8+f_{15}&\textcolor{red}{f_2}+f_7+f_{10}&\\
						\hline
					\end{array}
				\]
				\caption{$\Delta_{\dot{v}}$-vectors of $\hat{\mu}_1(V_{\bar{w}})$}
				\label{tableDeltaVecteurs2}
			\end{table}			
			
			\begin{figure}[ht]
				\begin{center}\resizebox{!}{7cm}{
				\begin{tikzpicture}[>=stealth, scale=1]
					\node[blue] (V1) at (-1,-4) {$V_1^\ast=1$};
					\node[red] (V2) at (-2,-2) {$V_2=\xymatrix@=-5pt{&2\\1&}$};
					\node[blue] (V3) at (-3,-4) {$V_3^\ast=\xymatrix@=-5pt{&2&\\1&&3}$};
					\node[Green] (V4) at (-4,-6) {$V_4=\xymatrix@=-5pt{1&&\\&2&\\&&3}$};
					\node[Orange] (V5) at (-4.5,-8) {$V_5=\xymatrix@=-5pt{1&&&\\&2&&\\&&3&\\&&&4}$};
					\node[violet] (V6) at (-6,-10) {$V_6=\xymatrix@=-5pt{1&&&&\\&2&&&\\&&3&&\\&&&4&\\&&&&5}$};
					\node[Orange] (V7) at (-7,-8) {$V_7=\xymatrix@=-5pt{1&&&&\\&2&&&\\&&3&&5\\&&&4&}$};
					\node[blue] (V8) at (-8,-4) {$V_8^\ast=\xymatrix@=-5pt{&&&&5\\&2&&4&\\1&&3&&}$};
					\node[Green] (V9) at (-9,-6) {$V_9=\xymatrix@=-5pt{&2&&&\\1&&3&&5\\&2&&4&\\&&3&&}$};
					\node[Orange] (V10) at (-10,-8) {$V_{10}=\xymatrix@=-5pt{&2&&&\\1&&3&&\\&2&&4&\\&&3&&5\\&&&4&}$};
					\node[blue] (V11) at (-11,-4) {$V_{11}=\xymatrix@=-5pt{&&&&5\\&2&&4&\\1&&3&&\\&2&&&}$};
					\node[Green] (V12) at (-12,-6) {$V_{12}=\xymatrix@=-5pt{&2&&4&\\1&&3&&5\\&2&&4&\\&&3&&}$};
					\node[red] (V13) at (-13,-2) {$V_{13}=\xymatrix@=-5pt{&&&&5\\&&&4&\\&&3&&\\&2&&&\\1&&&&}$};

					\draw (V1);
					\draw (V2);
					\draw (V3);
					\draw (V4);
					\draw (V5);
					\draw (V6);
					\draw (V7);
					\draw (V8);
					\draw (V9);
					\draw (V10);
					\draw (V11);
					\draw (V12);
					\draw (V13);
					
					
					\draw[->] (V2)--(V13);
					
					\draw[->] (V4)--(V9);
					\draw[->] (V9)--(V12);
					
					\draw[->] (V5)--(V7);
					\draw[->] (V7)--(V10);
					
					\draw[->] (V7)--(V4);
					\draw[->] (V6)--(V5);
					\draw[->] (V10)--(V6);
					\draw[->] (V9)--(V7);
					\draw[->] (V10)--(V9);
					\draw[->] (V12)--(V10);
					\draw[->] (V12)--(V11);
					\draw[->] (V13)--(V11);

					\draw[->] (V1)--(V3);
					\draw[->] (V4)--(V1);
					\draw[->] (V3)--(V4);
					\draw[magenta,very thick,->] (V3)--(V8);
					\draw[magenta,very thick,->] (V8)--(V2);
					\draw[magenta,very thick,->] (V11)--(V8);
					\draw[magenta,very thick,->] (V8)--(V9);
					\draw[magenta,very thick,->] (V9)--(V3);
					
				\end{tikzpicture}}
				\end{center}
				\caption{Module $\hat{\mu}_1(V_{\bar{w}})$}
				\label{figureMutation3}
			\end{figure}

			Having completed $\hat{\mu}_1$, we now rename the summands
			\[
				V_1^\ast\leadsto R_{1,1},\ V_3^\ast\leadsto R_{3,1},\ V_8^\ast\leadsto R_{8,1},\ V_{k}\leadsto R_{k,1},\ k\not\in\{1,3,8\}.
			\]

			We now determine $A_2(\hat{\mu}_1(V_{\bar{w}}))$ et $\hat{\mu}_2$.
			
			We have
			\[
				A_2(R_{1})=\{1\leq i\leq 2\mid\Delta_{\dot{v},2}(R_{i,1})\}=\{2\}
			\]
			and so
			\[
				\hat{\mu}_2=\mu_2=\widetilde{\mu}_2.
			\]
			
			We now apply $\hat{\mu}_2=\mu_2$ to $\hat{\mu}_1(V_{\bar{w}})$, and get Figure \ref{figureMutation4} and $\Delta_{\dot{v}}$-vectors of Table \ref{tableDeltaVecteurs3}.
			
			\begin{table}[h!t]
				\begin{center}
				\resizebox{\textwidth}{!}{
				$
					\begin{array}{|c|c|c|c|c|c|c|c|}
						\hline
						k&1&2&3&4&5&6&7\\
						\hline
						\Delta_{\dot{v}}(R_{k,2})&f_{11}&f_7+f_{10}&\textcolor{blue}{f_6}&\textcolor{Green}{f_3}+f_{11}&f_{12}&\textcolor{violet}{f_4}+f_{12}&\textcolor{orange}{f_5}+f_{11}\\
						\hline
						k&8&9&10&11&12&13&\\
						\hline
						\Delta_{\dot{v}}(R_{k,2})&\textcolor{blue}{f_6}+f_{10}&\textcolor{Green}{f_3}+f_8&\textcolor{orange}{f_5}+f_{11}+f_{13}&\textcolor{blue}{f_1}+\textcolor{blue}{f_6}+f_{10}&\textcolor{Green}{f_3}+f_8+f_{15}&\textcolor{red}{f_2}+f_7+f_{10}&\\
						\hline
					\end{array}
				$
				}
				\end{center}
				\caption{$\Delta_{\dot{v}}$-vectors of $\hat{\mu}_2\circ\hat{\mu}_1(V_{\bar{w}})$}
				\label{tableDeltaVecteurs3}
			\end{table}
			
			\begin{figure}[h!t]
				\begin{center}
				\resizebox{!}{7cm}{
				\begin{tikzpicture}[>=stealth, yscale=0.85]
					\node[blue] (V1) at (-1,-4) {$R_{1,1}=1$};
					\node[red] (V2) at (-2,-2) {$R_{2,1}^\ast=\xymatrix@=-5pt{&&5\\&4&\\3&&}$};
					\node[blue] (V3) at (-3.5,-4) {$R_{3,1}=\xymatrix@=-5pt{&2&\\1&&3}$};
					\node[Green] (V4) at (-4,-6) {$R_{4,1}=\xymatrix@=-5pt{1&&\\&2&\\&&3}$};
					\node[Orange] (V5) at (-4,-8) {$R_{5,1}=\xymatrix@=-5pt{1&&&\\&2&&\\&&3&\\&&&4}$};
					\node[violet] (V6) at (-6,-10) {$R_{6,1}=\xymatrix@=-5pt{1&&&&\\&2&&&\\&&3&&\\&&&4&\\&&&&5}$};
					\node[Orange] (V7) at (-7,-8) {$R_{7,1}=\xymatrix@=-5pt{1&&&&\\&2&&&\\&&3&&5\\&&&4&}$};
					\node[blue] (V8) at (-8,-4) {$R_{8,1}=\xymatrix@=-5pt{&&&&5\\&2&&4&\\1&&3&&}$};
					\node[Green] (V9) at (-9,-6) {$R_{9,1}=\xymatrix@=-5pt{&2&&&\\1&&3&&5\\&2&&4&\\&&3&&}$};
					\node[Orange] (V10) at (-10,-8) {$R_{10,1}=\xymatrix@=-5pt{&2&&&\\1&&3&&\\&2&&4&\\&&3&&5\\&&&4&}$};
					\node[blue] (V11) at (-11,-4) {$R_{11,1}=\xymatrix@=-5pt{&&&&5\\&2&&4&\\1&&3&&\\&2&&&}$};
					\node[Green] (V12) at (-12,-6) {$R_{12,1}=\xymatrix@=-5pt{&2&&4&\\1&&3&&5\\&2&&4&\\&&3&&}$};
					\node[red] (V13) at (-13,-2) {$R_{13,1}=\xymatrix@=-5pt{&&&&5\\&&&4&\\&&3&&\\&2&&&\\1&&&&}$};

					\draw (V1);
					\draw (V2);
					\draw (V3);
					\draw (V4);
					\draw (V5);
					\draw (V6);
					\draw (V7);
					\draw (V8);
					\draw (V9);
					\draw (V10);
					\draw (V11);
					\draw (V12);
					\draw (V13);
					
					\draw[->] (V1)--(V3);
					\draw[magenta,very thick,->] (V13)--(V2);
					\draw[->] (V3)--(V8);
					\draw[->] (V4)--(V9);
					\draw[->] (V5)--(V7);
					\draw[->] (V7)--(V10);
					\draw[->] (V9)--(V12);
					\draw[->] (V11)--(V8);

					\draw[->] (V4)--(V1);
					\draw[magenta,very thick,->] (V2)--(V8);
					\draw[->] (V9)--(V3);
					\draw[->] (V3)--(V4);
					\draw[->] (V7)--(V4);
					\draw[->] (V6)--(V5);
					\draw[->] (V10)--(V6);
					\draw[->] (V9)--(V7);
					\draw[->] (V8)--(V9);
					\draw[->] (V10)--(V9);
					\draw[->] (V12)--(V10);
					\draw[->] (V12)--(V11);
					\draw[->] (V13)--(V11);
					\draw[magenta,very thick,->] (V8)--(V13);
					
				\end{tikzpicture}}
				\end{center}
				\caption{Module $\hat{\mu}_2\circ\hat{\mu}_1(V_{\bar{w}})$}
				\label{figureMutation4}
			\end{figure}

			We again rename the indecomposables factors $R_{2,1}^\ast\leadsto R_{2,2}$ et $R_{k,1}\leadsto R_{k,2}$ pour $k\neq 2$.
			
			We now compute:
			\[
				A_3(R_2)=\{1\leq i\leq 9\mid\Delta_{\dot{v},3}(R_{i,3})\neq 0\}=\{4,9\}
			\]
			and we have $\hat{\mu}_3=\mu_9\circ\mu_4$ and again $\hat{\mu}_3=\widetilde{\mu}_3$.
			
			We get the module $R_3$ represented on Figure \ref{figureMutation5} with $\Delta_{\dot{v}}$-vectors of Table \ref{tableDeltaVecteurs4}.
			
			\begin{figure}[h!t]
				\begin{center}
				\resizebox{!}{7cm}{
				\begin{tikzpicture}[>=stealth, yscale=0.85]
					\node[blue] (V1) at (-1,-4) {$R_{1,2}=1$};
					\node[red] (V2) at (-2,-2) {$R_{2,2}=\xymatrix@=-5pt{&&5\\&4&\\3&&}$};
					\node[blue] (V3) at (-3.5,-4) {$R_{3,2}=\xymatrix@=-5pt{&2&\\1&&3}$};
					\node[Green] (V4) at (-4,-6) {$R_{4,2}^\ast=\xymatrix@=-5pt{&2&&&\\1&&3&&5\\&&&4&}$};
					\node[Orange] (V5) at (-3,-8) {$R_{5,2}=\xymatrix@=-5pt{1&&&\\&2&&\\&&3&\\&&&4}$};
					\node[violet] (V6) at (-6,-10) {$R_{6,2}=\xymatrix@=-5pt{1&&&&\\&2&&&\\&&3&&\\&&&4&\\&&&&5}$};
					\node[Orange] (V7) at (-7,-8) {$R_{7,2}=\xymatrix@=-5pt{1&&&&\\&2&&&\\&&3&&5\\&&&4&}$};
					\node[blue] (V8) at (-8,-4) {$R_{8,2}=\xymatrix@=-5pt{&&&&5\\&2&&4&\\1&&3&&}$};
					\node[Green] (V9) at (-9,-6) {$R_{9,2}^\ast=\xymatrix@=-5pt{&2&&4&\\1&&3&&5\\&&&4&}$};
					\node[Orange] (V10) at (-10,-8) {$R_{10,2}=\xymatrix@=-5pt{&2&&&\\1&&3&&\\&2&&4&\\&&3&&5\\&&&4&}$};
					\node[blue] (V11) at (-11,-4) {$R_{11,2}=\xymatrix@=-5pt{&&&&5\\&2&&4&\\1&&3&&\\&2&&&}$};
					\node[Green] (V12) at (-12,-6) {$R_{12,2}=\xymatrix@=-5pt{&2&&4&\\1&&3&&5\\&2&&4&\\&&3&&}$};
					\node[red] (V13) at (-13,-2) {$R_{13,2}=\xymatrix@=-5pt{&&&&5\\&&&4&\\&&3&&\\&2&&&\\1&&&&}$};

					\draw (V1);
					\draw (V2);
					\draw (V3);
					\draw (V4);
					\draw (V5);
					\draw (V6);
					\draw (V7);
					\draw (V8);
					\draw (V9);
					\draw (V10);
					\draw (V11);
					\draw (V12);
					\draw (V13);
					
					\draw[->] (V13)--(V2);
					\draw[->] (V3)--(V8);
					\draw[->] (V4)--(V9);
					\draw[->] (V5)--(V7);
					\draw[->] (V7)--(V10);
					\draw[magenta,very thick,->] (V12)--(V9);
					\draw[->] (V11)--(V8);

					\draw[magenta,very thick,->] (V1)--(V4);
					\draw[->] (V2)--(V8);
					\draw[magenta,very thick,->] (V4)--(V3);
					\draw[magenta,very thick,->] (V4)--(V7);
					\draw[->] (V6)--(V5);
					\draw[->] (V10)--(V6);
					
					\draw[magenta,very thick,->] (V9)--(V8);
					\draw[magenta,very thick,->] (V9)--(V10);
					\draw[->] (V12)--(V11);
					\draw[->] (V13)--(V11);
					\draw[->] (V8)--(V13);
					\draw[magenta,very thick,->,bend right] (V7)to[bend right](V1);
					\draw[magenta,very thick,->] (V8)--(V4);
					\draw[magenta,very thick,->] (V8)--(V12);
					\draw[magenta,very thick,->] (V10)--(V4);
					
				\end{tikzpicture}}
				\end{center}
				\caption{Module $\hat{\mu}_4\circ\hat{\mu}_3\circ\hat{\mu}_2\circ\hat{\mu}_1(V_{\bar{w}})=\hat{\mu}_3\circ\hat{\mu}_2\circ\hat{\mu}_1(V_{\bar{w}})$}
				\label{figureMutation5}
			\end{figure}
			
			\begin{table}[h!t]
				\begin{center}
				\resizebox{\textwidth}{!}{
				$
					\begin{array}{|c|c|c|c|c|c|c|c|}
						\hline
						k&1&2&3&4&5&6&7\\
						\hline
						\Delta_{\dot{v}}(R_{k,3})&f_{11}&f_7+f_{10}&\textcolor{blue}{f_6}&f_8&f_{12}&\textcolor{violet}{f_4}+f_{12}&\textcolor{orange}{f_5}+f_{11}\\
						\hline
						k&8&9&10&11&12&13&\\
						\hline
						\Delta_{\dot{v}}(R_{k,3})&\textcolor{blue}{f_6}+f_{10}&f_8+f_{15}&\textcolor{orange}{f_5}+f_{11}+f_{13}&\textcolor{blue}{f_1}+\textcolor{blue}{f_6}+f_{10}&\textcolor{Green}{f_3}+f_8+f_{15}&\textcolor{red}{f_2}+f_7+f_{10}&\\
						\hline
					\end{array}
				$
				}
				\end{center}
				\caption{$\Delta_{\dot{v}}$-vectors of $\hat{\mu}_3\circ\cdots\circ\hat{\mu}_1(V_{\bar{w}})$}
				\label{tableDeltaVecteurs4}
			\end{table}

			We rename $R_{4,2}^\ast\leadsto R_{4,3}$, $R_{9,2}^\ast\leadsto R_{9,3}$ et $R_{k,2}\leadsto R_{k,3}$.
			
			We compute
			\[
				A_4(R_3)=\{1\leq i\leq 0|\Delta_{\dot{v},4}(R_{i,3})\neq 0\}=\varnothing
			\]
			and so $\hat{\mu}_4=id=\widetilde{\mu}_4$.			
			
			As $\hat{\mu}_4=id$,we rename $R_{k,3}\leadsto R_{k,4}$.
			
			We compute
			\[
				A_5(R_4)=\{1\leq i\leq 7|\Delta_{\dot{v},5}(R_{i,4})\neq0\}=\{7\}.
			\]
			So $\hat{\mu}_5=\mu_7=\widetilde{\mu}_5$ and then $R_5=\hat{\mu}_5\circ\cdots\circ\hat{\mu}_1(V_{\bar{w}})$ is the module on Figure \ref{figureMutation6} whose $\Delta_{\dot{v}}$-vectors are on Table \ref{tableDeltaVecteurs5}.

			\begin{figure}[ht]
				\begin{center}
				\resizebox{!}{7cm}{
				\begin{tikzpicture}[>=stealth, yscale=0.85]
					\node[blue] (V1) at (-1,-4) {$R_{1,4}=1$};
					\node[red] (V2) at (-2,-2) {$R_{2,4}=\xymatrix@=-5pt{&&5\\&4&\\3&&}$};
					\node[blue] (V3) at (-3.5,-4) {$R_{3,4}=\xymatrix@=-5pt{&2&\\1&&3}$};
					\node[Green] (V4) at (-4,-6) {$R_{4,4}=\xymatrix@=-5pt{&2&&&\\1&&3&&5\\&&&4&}$};
					\node[Orange] (V5) at (-3,-8) {$R_{5,4}=\xymatrix@=-5pt{1&&&\\&2&&\\&&3&\\&&&4}$};
					\node[violet] (V6) at (-6,-11) {$R_{6,4}=\xymatrix@=-5pt{1&&&&\\&2&&&\\&&3&&\\&&&4&\\&&&&5}$};
					\node[Orange] (V7) at (-7,-8) {$R_{7,4}^\ast=\xymatrix@=-5pt{&2&&\\1&&3&\\&&&4}$};
					\node[blue] (V8) at (-8,-4) {$R_{8,4}=\xymatrix@=-5pt{&&&&5\\&2&&4&\\1&&3&&}$};
					\node[Green] (V9) at (-9,-6) {$R_{9,4}=\xymatrix@=-5pt{&2&&4&\\1&&3&&5\\&&&4&}$};
					\node[Orange] (V10) at (-10,-8) {$R_{10,4}=\xymatrix@=-5pt{&2&&&\\1&&3&&\\&2&&4&\\&&3&&5\\&&&4&}$};
					\node[blue] (V11) at (-11,-4) {$R_{11,4}=\xymatrix@=-5pt{&&&&5\\&2&&4&\\1&&3&&\\&2&&&}$};
					\node[Green] (V12) at (-12,-6) {$R_{12,4}=\xymatrix@=-5pt{&2&&4&\\1&&3&&5\\&2&&4&\\&&3&&}$};
					\node[red] (V13) at (-13,-2) {$R_{13,4}=\xymatrix@=-5pt{&&&&5\\&&&4&\\&&3&&\\&2&&&\\1&&&&}$};

					\draw (V1);
					\draw (V2);
					\draw (V3);
					\draw (V4);
					\draw (V5);
					\draw (V6);
					\draw (V7);
					\draw (V8);
					\draw (V9);
					\draw (V10);
					\draw (V11);
					\draw (V12);
					\draw (V13);
					
					\draw[->] (V13)--(V2);
					\draw[->] (V3)--(V8);
					\draw[->] (V4)--(V9);
					\draw[magenta,very thick,->] (V7)--(V5);
					\draw[magenta,very thick,->] (V10)--(V7);
					\draw[->] (V12)--(V9);
					\draw[->] (V11)--(V8);

					\draw[->] (V2)--(V8);
					\draw[->] (V4)--(V3);
					\draw[magenta,very thick,->] (V7)to[bend left](V4);
					\draw[->] (V6)--(V5);
					\draw[->] (V10)--(V6);
					
					\draw[->] (V9)--(V8);
					\draw[->] (V9)--(V10);
					\draw[->] (V12)--(V11);
					\draw[->] (V13)--(V11);
					\draw[->] (V8)--(V13);
					\draw[magenta,very thick,->] (V1)--(V7);
					\draw[->] (V8)--(V4);
					\draw[->] (V8)--(V12);
					
					\draw[magenta,very thick,->] (V5)to[bend left](V10);
					\draw[magenta,very thick,->] (V5)--(V1);
					
				\end{tikzpicture}}
				\end{center}
				\caption{Module $\hat{\mu}_5\circ\hat{\mu}_4\circ\hat{\mu}_3\circ\hat{\mu}_2\circ\hat{\mu}_1(V_{\bar{w}})$}
				\label{figureMutation6}
			\end{figure}
			
			\begin{table}[ht]
				\begin{center}
				\resizebox{\textwidth}{!}{
				$
					\begin{array}{|c|c|c|c|c|c|c|c|}
						\hline
						k&1&2&3&4&5&6&7\\
						\hline
						\Delta_{\dot{v}}(R_{k,5})&f_{11}&f_7+f_{10}&\textcolor{blue}{f_6}&f_8&f_{12}&\textcolor{violet}{f_4}+f_{12}&f_{11}+f_{13}\\
						\hline
						k&8&9&10&11&12&13&\\
						\hline
						\Delta_{\dot{v}}(R_{k,5})&\textcolor{blue}{f_6}+f_{10}&f_8+f_{15}&\textcolor{orange}{f_5}+f_{11}+f_{13}&\textcolor{blue}{f_1}+\textcolor{blue}{f_6}+f_{10}&\textcolor{Green}{f_3}+f_8+f_{15}&\textcolor{red}{f_2}+f_7+f_{10}&\\
						\hline
					\end{array}
				$
				}
				\end{center}
				\caption{$\Delta_{\dot{v}}$-vectors of $\hat{\mu}_5\circ\cdots\circ\hat{\mu}_1(V_{\bar{w}})$}
				\label{tableDeltaVecteurs5}
			\end{table}

			We rename $R_{7,4}^\ast=R_{7,5}$ et $R_{k,4}=R_{k,5}$ et $k\neq 7$.
			
			\[
				A_6=\{1\leq i\leq 3|\Delta_{\dot{v},6}(R_5)\neq 0\}=\{3\}
			\]
			and so we have $\hat{\mu}_6=\mu_3=\widetilde{\mu}_6$.
			
			That gives $\hat{\mu}_6\circ\cdots\circ\hat{\mu}_1(V_{\bar{w}})$ represented on Figure \ref{figureMutation7} whose $\Delta_{\dot{v}}$-vectors are on Table \ref{tableDeltaVecteurs6}.
			
			\begin{table}[ht]
				\begin{center}
				\resizebox{\textwidth}{!}{
				$
					\begin{array}{|c|c|c|c|c|c|c|c|}
						\hline
						k&1&2&3&4&5&6&7\\
						\hline
						\Delta_{\dot{v}}(R_{k,6})&f_{11}&f_7+f_{10}&f_{10}&f_8&f_{12}&\textcolor{violet}{f_4}+f_{12}&f_{11}+f_{13}\\
						\hline
						k&8&9&10&11&12&13&\\
						\hline
						\Delta_{\dot{v}}(R_{k,6})&\textcolor{blue}{f_6}+f_{10}&f_8+f_{15}&\textcolor{orange}{f_5}+f_{11}+f_{13}&\textcolor{blue}{f_1}+\textcolor{blue}{f_6}+f_{10}&\textcolor{Green}{f_3}+f_8+f_{15}&\textcolor{red}{f_2}+f_7+f_{10}&\\
						\hline
					\end{array}
				$
				}
				\end{center}
				\caption{$\Delta_{\dot{v}}$-vectors of $\hat{\mu}_6\circ\cdots\circ\hat{\mu}_1(V_{\bar{w}})$}
				\label{tableDeltaVecteurs6}
			\end{table}
		
			\begin{figure}[h!t]
				\begin{center}
				\resizebox{!}{7cm}{
				\begin{tikzpicture}[>=stealth, yscale=0.8]
					\node[blue] (V1) at (-1,-4) {$R_{1,5}=1$};
					\node[red] (V2) at (-2,-2) {$R_{2,5}=\xymatrix@=-5pt{&&5\\&4&\\3&&}$};
					\node[blue] (V3) at (-3.5,-4) {$R_{3,5}^\ast=\xymatrix@=-5pt{&5\\4&}$};
					\node[Green] (V4) at (-4,-6) {$R_{4,5}=\xymatrix@=-5pt{&2&&&\\1&&3&&5\\&&&4&}$};
					\node[Orange] (V5) at (-3,-8) {$R_{5,5}=\xymatrix@=-5pt{1&&&\\&2&&\\&&3&\\&&&4}$};
					\node[violet] (V6) at (-6,-11) {$R_{6,5}=\xymatrix@=-5pt{1&&&&\\&2&&&\\&&3&&\\&&&4&\\&&&&5}$};
					\node[Orange] (V7) at (-7,-8) {$R_{7,5}=\xymatrix@=-5pt{&2&&\\1&&3&\\&&&4}$};
					\node[blue] (V8) at (-8,-4) {$R_{8,5}=\xymatrix@=-5pt{&&&&5\\&2&&4&\\1&&3&&}$};
					\node[Green] (V9) at (-9,-6) {$R_{9,5}=\xymatrix@=-5pt{&2&&4&\\1&&3&&5\\&&&4&}$};
					\node[Orange] (V10) at (-10,-8) {$R_{10,5}=\xymatrix@=-5pt{&2&&&\\1&&3&&\\&2&&4&\\&&3&&5\\&&&4&}$};
					\node[blue] (V11) at (-11,-4) {$R_{11,5}=\xymatrix@=-5pt{&&&&5\\&2&&4&\\1&&3&&\\&2&&&}$};
					\node[Green] (V12) at (-12,-6) {$R_{12,5}=\xymatrix@=-5pt{&2&&4&\\1&&3&&5\\&2&&4&\\&&3&&}$};
					\node[red] (V13) at (-13,-2) {$R_{13,5}=\xymatrix@=-5pt{&&&&5\\&&&4&\\&&3&&\\&2&&&\\1&&&&}$};

					\draw (V1);
					\draw (V2);
					\draw (V3);
					\draw (V4);
					\draw (V5);
					\draw (V6);
					\draw (V7);
					\draw (V8);
					\draw (V9);
					\draw (V10);
					\draw (V11);
					\draw (V12);
					\draw (V13);
					
					\draw[->] (V13)--(V2);
					\draw[magenta,very thick,->] (V8)--(V3);
					\draw[->] (V4)--(V9);
					\draw[->] (V7)--(V5);
					\draw[->] (V10)--(V7);
					\draw[->] (V12)--(V9);
					\draw[->] (V11)--(V8);

					\draw[->] (V2)--(V8);
					\draw[magenta,very thick,->] (V3)--(V4);
					\draw[->] (V7)to[bend left](V4);
					\draw[->] (V6)--(V5);
					\draw[->] (V10)--(V6);
					
					\draw[->] (V9)--(V8);
					\draw[->] (V9)--(V10);
					\draw[->] (V12)--(V11);
					\draw[->] (V13)--(V11);
					\draw[->] (V8)--(V13);
					\draw[->] (V1)--(V7);
					\draw[->] (V8)--(V12);
					
					\draw[->] (V5)to[bend left](V10);
					\draw[->] (V5)--(V1);
					
				\end{tikzpicture}}
				\end{center}
				\caption{Module $\hat{\mu}_6\circ\hat{\mu}_5\circ\hat{\mu}_4\circ\hat{\mu}_3\circ\hat{\mu}_2\circ\hat{\mu}_1(V_{\bar{w}})$}
				\label{figureMutation7}
			\end{figure}
			
			We rename $R_{3,5}^\ast\leadsto R_{3,6}$ and $R_{k,5}\leadsto R_{k,6}$ for $k\neq 3$.

			In order to get $\mu_\bullet$, we compute $(k_\Max)^{\alpha(k,\ell(v))-}$ for any color. For color $\textcolor{red}{1}$, we have $(2_\Max)^{\alpha(2,6)-}=13^-=2$.  For color $\textcolor{blue}{2}$, we have $(1_\Max)^{\alpha(1,6)-}=11^{2-}=3$. For color $\textcolor{Green}{3}$, we have $(4_\Max)^{\alpha(4,6)-}=12^-=9$. For color $\textcolor{orange}{4}$, we have $(5_\Max)^{\alpha(5,6)-}=10^-=7$. For color $\textcolor{violet}{5}$, we have $(6_\Max)^{\alpha(6,6)-}=6^{-}=0$. 
			
			Then $\mu_\bullet$ removes the indecomposable summands of indices $6,8,10,11,12,13$ and we get module of Figure \ref{figureMutation8} (with $R_{2,6}$ a non-connected vertex).
			
			\begin{figure}[h!t]
				\begin{center}
				\resizebox{!}{3cm}{
				\begin{tikzpicture}[>=stealth, yscale=0.7]
					\node[blue] (V1) at (-1,-4) {$R_{1,6}=1$};
					\node[red] (V2) at (-8,-4) {$R_{2,6}=\xymatrix@=-5pt{&&5\\&4&\\3&&}$};
					\node[blue] (V3) at (-3.5,-4) {$R_{3,6}=\xymatrix@=-5pt{&5\\4&}$};
					\node[Green] (V4) at (-4,-6) {$R_{4,6}=\xymatrix@=-5pt{&2&&&\\1&&3&&5\\&&&4&}$};
					\node[Orange] (V5) at (-3,-8) {$R_{5,6}=\xymatrix@=-5pt{1&&&\\&2&&\\&&3&\\&&&4}$};
					\node[Orange] (V7) at (-7,-8) {$R_{7,6}=\xymatrix@=-5pt{&2&&\\1&&3&\\&&&4}$};
					\node[Green] (V9) at (-9,-6) {$R_{9,6}=\xymatrix@=-5pt{&2&&4&\\1&&3&&5\\&&&4&}$};

					\draw (V1);
					\draw (V2);
					\draw (V3);
					\draw (V4);
					\draw (V5);
					\draw (V7);
					\draw (V9);
					
					\draw[->] (V4)--(V9);
					\draw[->] (V7)--(V5);

					\draw[,->] (V3)--(V4);
					\draw[->] (V7)to[bend left](V4);
					
					\draw[->] (V1)--(V7);
					
					\draw[->] (V5)--(V1);
					
				\end{tikzpicture}}
				\end{center}
				\caption{Module $\mu_\bullet(V_{\bar{w}})$}
				\label{figureMutation8}
			\end{figure}
			
			Then, according to Theorem \ref{theoremMain}, $\mu^\bullet(V_{\bar{w}})$ is a $\C_{v,w}$-maximal rigid basic module, whose quiver is the one of Figure \ref{figureMutation8}.
			\end{ex}
\newpage

\bibliographystyle{alpha}
\bibliography{../biblioGen}
\end{document}